\documentclass[final,leqno,onefignum,onetabnum]{siamltex1213}
\usepackage[linesnumbered,ruled,vlined]{algorithm2e}
\usepackage{amsfonts}
\usepackage{amsmath}
\usepackage{array}
\usepackage{cases}
\usepackage{color}

\newcommand{\vertiii}{{\vert\kern-0.25ex\vert\kern-0.25ex\vert}}
\newcommand{\lvertiii}{{\left\vert\kern-0.25ex\left\vert\kern-0.25ex\left\vert}}
\newcommand{\rvertiii}{{\right\vert\kern-0.25ex\right\vert\kern-0.25ex\right\vert}}
\newcommand{\vertIII}{{\Big\vert\kern-0.25ex\Big\vert\kern-0.25ex\Big\vert}}

\title{A conjugate gradient method for electronic structure
       calculations \thanks{This work was supported by the National Science
       Foundation of China under grant 9133202 and 11671389, the Funds for Creative Research Groups of
       China under grant 11321061,   the Key Research Program of Frontier Sciences of the Chinese Academy of Sciences under grant QYZDJ-SSW-SYS010.}}

\author{Xiaoying Dai\footnotemark[2], Zhuang Liu\footnotemark[2], Liwei Zhang\footnotemark[2],  and Aihui Zhou\footnotemark[2]}

\begin{document}

\maketitle

\renewcommand{\thefootnote}{\fnsymbol{footnote}}

\footnotetext[2]{LSEC, Institute of Computational Mathematics and Scientific/Engineering Computing,
Academy of Mathematics and Systems Science, Chinese Academy of Sciences,  Beijing 100190, China; and School of Mathematical Sciences,
University of Chinese Academy of Sciences, Beijing 100049, China. daixy, liuzhuang, zhanglw, azhou@lsec.cc.ac.cn}
%\author{
%Xiaoying Dai\thanks{LSEC, Institute of Computational Mathematics
%and Scientific/Engineering Computing, Academy of Mathematics and
%Systems Science, Chinese Academy of Sciences, Beijing 100190, China
%(daixy@lsec.cc.ac.cn).} \and  Zhuang Liu\thanks{LSEC, Institute of Computational Mathematics
%and Scientific/Engineering Computing, Academy of Mathematics and
%Systems Science, Chinese Academy of Sciences, Beijing 100190, China
%(liuzhuang@lsec.cc.ac.cn).} \and  Liwei Zhang\thanks{LSEC, Institute of Computational Mathematics
%and Scientific/Engineering Computing, Academy of Mathematics and
%Systems Science, Chinese Academy of Sciences, Beijing 100190, China
%(zhanglw@lsec.cc.ac.cn).}
%\and Aihui Zhou\thanks{LSEC, Institute of Computational Mathematics
%and Scientific/Engineering Computing, Academy of Mathematics and
%Systems Science, Chinese Academy of Sciences, Beijing 100190, China
%(azhou@lsec.cc.ac.cn).}
%}

\renewcommand{\thefootnote}{\arabic{footnote}}
\slugger{mms}{xxxx}{xx}{x}{x--x}%slugger should be set to mms, siap, sicomp, sicon,
% sidma, sima, simax, sinum, siopt, sisc, or sirev

\begin{abstract}
In this paper, we study a conjugate gradient method for electronic structure calculations.
We propose a Hessian based step size strategy, which together with three orthogonality
approaches yields three algorithms for computing the ground state energy of atomic and
molecular systems. Under some mild assumptions, we prove that our algorithms
converge locally. It is shown by our numerical experiments that the conjugate gradient
method is efficient.
\end{abstract}



\pagestyle{myheadings}
\thispagestyle{plain}
% \markboth{TEX PRODUCTION}{USING SIAM'S \LaTeX\ MACROS}

\section{Introduction}

Kohn-Sham density functional theory (DFT) \cite{HK,KS,PY} is widely
used in electronic structure calculations. It is often formulated as a nonlinear eigenvalue
problem or a direct minimization problem under orthogonality constraint \cite{dai-zhou15,Ma,saad10}.

 The nonlinear eigenvalue problem is usually solved by using the self consistent
field (SCF) iterations,  by which
the central computation in solving such nonlinear eigenvalue problems is the repeated solution of
some algebraic eigenvalue problems. However, the convergence of SCF is not
guaranteed, especially for large scale systems with small band gaps,  for which the performance of the SCF iterations
is unpredictable \cite{YMW, ZZWZ}. % even with the help of the
%density or potential mixing techniques \cite{Jo, Pu}.

Therefore, people turn to investigate constrained minimization approaches for
the Kohn-Sham direct energy minimization models, see e.g. \cite{FMM, GLCY, SRNB, TOYJSH, UWYKL, WYLZ, WY, YMW,
ZZWZ} and references therein.  In  \cite{YMW}, the authors constructed the search direction
 from the subspace spanned by the current orbital
 approximations, the associated preconditioned gradient,
and the previous search direction. The optimal search direction
and step size are computed by solving a smaller scale nonlinear eigenvalue problem, which is
complicated for large systems. While the authors in \cite{WY, ZZWZ} applied gradient type methods to the
minimization problem, in which the gradient is chosen as the search direction.
It has been shown in \cite{ZZWZ} that the gradient type methods are quite efficient and can outperform
the SCF iterations on many practical systems.

 We should point out that there are also several works using some conjugate gradient (CG) methods
 to the electronic structure calculations, see e.g., \cite{PTAJ, SCPB, TePaAl}. However, in these works,
 their starting points are to solve the nonlinear eigenvalue problems other than
the minimization problems, the CG methods are used to find the subspaces by updating each orbital along
its corresponding conjugate gradient direction, and these orbitals are updated successively.
After all the orbitals are updated, a subspace diagonalization (Rayleigh-Ritz) procedure
is then carried out. As pointed out in \cite{YMW},  due to the ``band-by-band'' nature,
 the algorithms are not very efficient.

In this paper, following \cite{WY, ZZWZ}, we apply the similar idea to construct some
novel conjugate gradient method, where the search
direction is replaced by the conjugate gradient direction. We utilize the
so-called  WY \cite{WY} and QR strategies to keep the orthogonality of the Kohn-Sham orbitals,
which were also employed
 in \cite{ZZWZ}. In addition, we also apply the polar decomposition (PD) \cite{AbMaSe} based approach
 to do the orthogonalization.
 %We  prove the convergence
% of the WY, QR, and PD based conjugate gradient algorithms, and  observe that
%the QR based conjugate gradient  algorithm usually performs best in our numerical experiments.

\iffalse We note that the authors in \cite{GLCY} proposed a new first order framework orthogonal constrained optimization problems recently,
 where the gradient reflection was applied to keep the orthogonality, which is a new orthogonality preserving strategy.\fi

We understand that an important issue in a conjugate gradient method is the choice of step size% for the conjugate gradient direction
. To set up the step size in our conjugate gradient method, we introduce a Hessian based
strategy, which is based on the local second order
Taylor expansion of the total energy functional. We  prove the convergence
 of the WY, QR, and PD based conjugate gradient algorithms, and  observe that
the QR based conjugate gradient  algorithm usually performs best in our numerical experiments.
Although we need some energy descent
property, which is widely used in many optimization algorithms, to prove the convergence of
the algorithms theoretically (see Theorem \ref{thm-conv}), our algorithms perform
well without using any backtracking procedure. We compare our algorithms with the algorithm OptM-QR
recently proposed in \cite{ZZWZ} which is a gradient type method with Barzilai-Borwein (BB) step sizes, and observe from our numerical experiments that our algorithms need less iterations and less computational time to obtain
the results with  same accuracy. \iffalse We also see that for small systems, our algorithms cost
roughly the same time; while for large systems, our WY and PD based conjugate gradient
algorithms do not perform very well compared with our QR based conjugate gradient algorithm.\fi
In addition, it is shown by our numerical experiments that the algorithm OptM-QR is less stable
than our algorithms, which may be  caused by the nonmonotonic behavior of
the BB step size \cite{Dai2}.

We see that the standard Armijo and the exact line search strategies are applied to the
geometric CG method on matrix manifold (see Algorithm 13 of \cite{AbMaSe}).
However, the exact line search is not recommended to the Kohn-Sham total energy minimization problem due to its high computational cost.
%We also observe that our  Hessian based strategy is more efficient than the
%Armijo linear search method in providing a good initial value, since a backtracking procedure
%is carried out in the line search so as to ensure some energy decent property in each iteration
%based on the initial value.
Note that the Armijo line search method uses a fixed initial value and performs a backtracking procedure
to ensure some energy decent property. We believe that our Hessian based strategy is more efficient
in providing a good initial value for the step size and reducing the times of backtracking.
 In particular, there is no convergence
analysis for Algorithm 13 given in \cite{AbMaSe}. We point out that Smith \cite{Smith93, Smith94} has also proposed a CG algorithm on the Riemannian manifold, where the orthogonality was preserved by using the Riemannian exponential map, whose computational cost is larger than the WY, QR, and PD strategies.
%In addition, the exact line search strategy was used for the step size, which is not recommended to the Kohn-Sham minimization problem.

We should mention that the convergence
of SCF iteration was proved in  \cite{LWWUY,LWWY,YGM}
under the assumptions that the gap between the occupied states and
unoccupied states is sufficiently large and the second-order derivatives of the exchange
correlation functional are uniformly bounded from above. Anyway, such investigations are
theoretically significant. %Although such theoretical investigations are
%significant, these assumptions are strict in application.

%We should point out that our algorithms proposed in this paper are in the similar framework
%with Algorithm 13 in \cite{AbMaSe}, which is a geometric CG method for optimization
%on matrix manifolds. However, as stated in the previous paragraph, our step size strategy is
%very different from the Algorithm 13, which uses standard Armijo
%or exact line search method. In the Armijo line search method, a fixed initial value is set
%for the step size, then a backtracking procedure is performed to ensure some energy decent
%property in each iteration. The times of the backtracking are dependent on the initial value,
%so we think our fixed form step size is more efficient than the Armijo line search method, at
%least better in providing a good initial value; while due to the complexity, the exact line
%search is almost impossible for the DFT minimization problems. In addition, no convergence
%theory of Algorithm 13 is given in \cite{AbMaSe}.

The rest of this paper is organized as follows: in Section \ref{sec-pre}, we provide a brief
introduction to the Kohn-Sham \iffalse density functional theory\fi model and the associated
Grassmann manifold. We then propose our conjugate gradient algorithms in
Section \ref{sec-alg} and prove the convergence of the three algorithms in Section
\ref{sec-cvg}. In Section \ref{sec-modalg}, we present some restarted versions of our algorithms.
We report several numerical experiments in Section \ref{sec-num}
that demonstrate the accuracy and efficiency of our algorithms. After that, we give some concluding
remarks in Section \ref{sec-cln}. Finally, we provide the proof of Lemma \ref{lema-uniq} in
Appendix \ref{app-proof} and present several numerical tests using different step sizes
and different calculation formulas for the Hessian in  Appendix \ref{app-numtest} that lead
to our recommendations.
\section{Preliminaries} \label{sec-pre}
\subsection{Kohn-Sham model}
By Kohn-Sham density functional theory, the ground state % Kohn-Sham DFT model
of a system consisting of $M$ nuclei of charges and $N$ electrons can be obtained by solving the
following constrained optimization problem
\begin{equation}\label{Emin}
\begin{split}
&\inf_{U = (u_1, \dots, u_N) \in (H^1(\mathbb{R}^3))^N} \ \ \ \ \ \ \ E(U) \\
s.t.  &\int_{\mathbb{R}^3} u_iu_j = \delta_{ij},  1\leq i,j\leq N,
\end{split}
%\begin{split}
%&\inf_{U = (u_1, \dots, u_N) \in (H^1(\mathbb{R}^3))^N} \ \ \ \ \ \ \ E(U) \\
%s.t.  &\int_{\mathbb{R}^3} u_i(r)u_j(r) dr = \delta_{ij},  1\leq i,j\leq N,
%\end{split}
\end{equation}
where the Kohn-Sham total energy $E(U)$ is defined by
\begin{eqnarray}\label{energy}
E(U)&=&\frac{1}{2} \int_{\mathbb{R}^3} \sum_{i=1}^N|\nabla u_i(r)|^2 dr +
\frac{1}{2}\int_{\mathbb{R}^3}\int_{\mathbb{R}^3}\frac{\rho(r)\rho(r')}{|r-r'|}drdr' \nonumber\\
&&+\int_{\mathbb{R}^3} \sum_{i=1}^N u_i(r)V_{ext}(r)u_i(r)dr+ \int_{\mathbb{R}^3}
\varepsilon_{xc}(\rho)(r)\rho(r)dr,
\end{eqnarray}
and $u_i\in H^1(\mathbb{R}^3), i = 1, \cdots, N$ are the  Kohn-Sham orbitals.
Here $\rho(r)=\sum \limits_{i=1}^N|u_i(r)|^2$ is the electronic density,
$V_{ext}(r)$ is the external potential generated by the nuclei: for full potential calculations,
$V_{ext}(r) = -\sum\limits_{I = 1}^{M}\frac{Z_I}{|r - R_I|}$, $Z_I$ and $R_I$ are the nuclei charge
and position of the $I$-th nuclei respectively; while for pseudo potential
approximations, the formula for the energy is still (2.2) but $V_{ext}u_i (r)$ is replaced by
\begin{equation*}
\sum \limits_{I=1}^M(V^I_{loc}u_i)(r)+(V^I_{nloc}u_i)(r),
\end{equation*}
where
$(V^I_{loc}u_i)(r)$ is the local part and $(V_{nloc}u_i)(r)$ is the nonlocal part, which usually have the following form
\begin{equation*}
(V^I_{nloc} u_i)(r)= \sum_l\int_{\mathbb{R}^3}\xi_l^{I}(r') u_i(r') dr'\xi_l^{I}(r),
\end{equation*}
with $\xi_l^{I}\in L^2(\mathbb{R}^3)$ \cite{Ma}.
For convenience, the following analysis formally focuses on the full potential case.
In fact, the results in this paper hold true for any kind of finite-dimensional functional
that satisfies the assumptions in Section 2.3, of course including the pseudo-potential case,
for which the energy functional has  higher regularity than that for the full potential case.
 % We still denote the pseudo potential by $V_{ext}(r)$ for convenience.
%where the Kohn-Sham total energy $E(U)$ is defined by
%\begin{eqnarray}\label{energy}
%E(U)&=&\frac{1}{2} \int_{\mathbb{R}^3} \sum_{i=1}^N|\nabla u_i(r)|^2 dr +
%\frac{1}{2}\int_{\mathbb{R}^3}\int_{\mathbb{R}^3}\frac{\rho(r)\rho(r')}{|r-r'|}drdr' \nonumber\\
%&&{\color{red}{+\int_{\mathbb{R}^3} V_{loc}(r)\rho(r)dr+ \int_{\mathbb{R}^3} \sum_{i=1}^N u_i(r)V_{nloc}(u_i)(r)dr}} \nonumber\\
%&&+ \int_{\mathbb{R}^3}{\color{red}{\varepsilon_{xc}(\rho)(r)}}\rho(r)dr,
%\end{eqnarray}
%and $u_i\in H^1(\mathbb{R}^3), i = 1, \cdots, N$ are the  Kohn-Sham orbitals.
%Here $\rho(r)=\sum \limits_{i=1}^N|u_i(r)|^2$ is the electronic density. {\color{red}{For full potential calculations,
%$V_{loc}(r) = -\sum\limits_{I = 1}^{M}\frac{Z_I}{|r - R_I|}$ and $V_{nloc}=0$ where $Z_I$ and $R_I$ are the nuclei charge
%and position of the $I$-th nuclei respectively; while for pseudo potential
%approximations, $V_{loc}(r)$ is the local part of the pseudo potential, $V_{nloc}$
%is the nonlocal pseudopotential operator, which usually has the following form}}
%\begin{equation*}
%V_{nloc} (u) (r)= \sum_l\int_{\mathbb{R}^3}\xi_l^{I}(r') u(r') dr'\xi_l^{I}(r),
%\end{equation*}
%with $\xi_l^{I}\in L^2(\mathbb{R}^3)$ \cite{Ma}.
The $\varepsilon_{xc}(\rho)(r)$ in the forth
term is the exchange-correlation functional, describing the many-body effects of exchange and
correlation, which is not known explicitly, and some approximation (such as local density approximation (LDA),
generalized gradient approximation (GGA))
has to be used \cite{Ma}.

\subsection{Gradient and Hessian on Grassmann manifold}
\newtheorem{assume}[theorem]{Assumption}

We first introduce some notation.
%\begin{definition}
Let $\Psi = (\psi_1,\cdots,\psi_N)\in (L^2(\mathbb{R}^3))^N$, $\Phi = (\phi_1,\cdots,\phi_N)
\in (L^2(\mathbb{R}^3))^N$.
 Define
\begin{equation*}
 \Psi^T\Phi  = (\langle\psi_i,\phi_j\rangle)_{i,j=1}^{N}\in
\mathbb{R}^{N\times N},
\end{equation*}
where $\langle\psi_i,\phi_j\rangle = \displaystyle \int_{\mathbb{R}^3} \psi_i(r)\phi_j(r) dr$
is the usual $L^2$ inner product in $L^2(\mathbb{R}^3)$.
%{\color{red}We also define the inner product of $(L^2(\mathbb{R}^3))^N$ by
%\begin{equation*}
%\langle\Psi, \Phi\rangle = \text{tr}\langle\Psi^T\Phi\rangle.
%\end{equation*}
For $\Psi = (\psi_1,\cdots,\psi_N)\in
(L^2(\mathbb{R}^3))^N$, we define its norm as
%Then the norm of $\Psi\in(L^2(\mathbb{R}^3))^N$ induced by its inner product is given by
\begin{equation}\label{norm-o}
\vertiii \Psi \vertiii = \left(\text{tr} (\Psi^T\Psi)\right)^{\frac{1}{2}}.
%= \Big(\sum_{i=1}^N\|\psi_i\|^2_{L^2}\Big)^{\frac{1}{2}}.
\end{equation}
%\begin{equation}\label{norm-o}
%\vertiii \Psi \vertiii = \left(\langle\Psi, \Psi\rangle\right)^{\frac{1}{2}}=\left(\text{tr}\langle\Psi^T\Psi\rangle\right)^{\frac{1}{2}}, \forall\Psi\in(L^2(\mathbb{R}^3))^N.
%= \Big(\sum_{i=1}^N\|\psi_i\|^2_{L^2}\Big)^{\frac{1}{2}}.
%\end{equation}}
For a matrix $A=(a_{ij})_{i, j=1}^N \in \mathbb{R}^{N\times N}$, the Frobenius norm  is defined as $\|A\|_F = \left(\sum\limits_{i, j = 1}^N |a_{ij}|^2\right)^{1/2}$, and the 2-norm is defined as $\|A\|_2 = \sigma_1$, where $\sigma_1$ is the largest singular value of $A$.% and
% for simplicity, we leave out $F$ and use $\|A\|$ to denote the Frobenius norm of matrix $A$.
%\end{definition}

Now we introduce two lemmas that will be used in our analysis without proof.
The proof of Lemma \ref{lemma111} can refer to \cite{GoLo,MAT.ANAL} and   Lemma \ref{lemma222} can be obtained by a standard analysis.
For convenience, we denote $\mathcal{O}^{N\times N}$ the set of orthogonal matrices of order $N$. %{\color{red}{More details about Lemma \ref{lemma111} can be seen in \cite{GoLo,MAT.ANAL} while Lemma \ref{lemma222} follows by standard analysis.}}
\begin{lemma}\label{lemma111}
(1) The Frobenius norm of $A \in \mathbb{R}^{N\times N}$ is orthogonal invariant, that is, if
$P, Q \in \mathcal{O}^{N\times N}$,  then
\begin{equation}\label{ort-invar}
  \|PAQ\|_F = \|A\|_F.
\end{equation}
(2) Suppose $A\in \mathbb{R}^{N\times N}$ is symmetric, if there exists
$P \in \mathbb{R}^{N\times N}$ reversible such that $B = P^{-1}AP$ is also symmetric, then
\begin{equation}\label{similar}
\|B\|_F= \|P^{-1}AP\|_F = \|A\|_F.
\end{equation}
(3) Suppose $A,B\in \mathbb{R}^{N\times N}$, then
\begin{equation}\label{normF-conc3}
  \|AB\|_F \le \|A\|_2\|B\|_F.
\end{equation}
\end{lemma}
%\begin{proof}
%Equation \eqref{ort-invar} is nothing but (2.3.14) in \cite{GoLo}.

%We then prove \eqref{similar}. We see from Corollary 2.4.3 of
%\cite{GoLo} that
%\begin{equation*}
%\|A\|_F^2=\sum_{i=1}^N \sigma_i^2,
%\end{equation*}
%where $\{\sigma_i\}_{i=1}^N$ are the singular values of $A$.
%Since the singular values are the absolute values of eigenvalues for any symmetry matrix,
% we have
%\begin{equation}\label{norm-sum}
%\|A\|_F^2=\sum_{i=1}^N \lambda_i^2,
%\end{equation}
%where $\{\lambda_i\}_{i=1}^N$ are the eigenvalues of $A$.
%Therefore, \eqref{similar} holds because of the fact that similar matrices have the same eigenvalues.
%

%We now prove \eqref{normF-conc3}. If $A$ is a diagonal matrix, then it is easy to verify from the definition of Frobenius norm that
%\begin{equation}\label{ADiag}
%  \|AB\|_F \le \max_{i}\{|A_{ii}|\}\|B\|_F=\|A\|_2\|B\|_F.
%\end{equation}
%here, $A_{ii}$ is the $i$-th diagonal element of A.

%If $A$ is not diagonal, let $A=P \Sigma Q^T$ be the Singular Value decomposition(SVD) of $A$, then $\|A\|_2=\|\Sigma\|_2$ by definition, which combining with \eqref{ort-invar} and \eqref{ADiag} leads %to
%\begin{equation*}
%  \|AB\|_F =\|P^T\Sigma QB\|_F=\|\Sigma QB\|_F\le\|\Sigma\|_2\|QB\|_F=\|\Sigma\|_2\|B\|_F=\|A\|_2\|B\|_F.
%\end{equation*}
%
%\end{proof}

\begin{lemma}\label{lemma222}
Let $\Psi = (\psi_1,\cdots,\psi_N)\in (L^2(\mathbb{R}^3))^N$, $\Phi = (\phi_1,\cdots,\phi_N)
\in (L^2(\mathbb{R}^3))^N$, and matrix $A\in \mathbb{R}^{N\times N}$. There
hold
\begin{eqnarray}
\|  \Psi^T\Phi  \|_F \le \vertiii \Psi \vertiii \vertiii \Phi \vertiii, \label{norm-oo} \\
\vertiii \Psi A \vertiii \le \vertiii \Psi \vertiii\|A\|_F\label{norm-om},
\end{eqnarray}
which mean that the norms for orbitals and matrices are compatible.
In further, if $  \Phi^T\Phi   =I_N$, then
\begin{eqnarray} \label{norm-conc3}
\vertiii\Phi A\vertiii=||A||_F.
\end{eqnarray}
\end{lemma}

%\begin{proof}
%It is easy to verify that
%\begin{eqnarray*}
%\|\langle \Psi^T\Phi \rangle\|_F^2 &=& \sum_{i,j=1}^N|\langle\psi_i,\phi_j\rangle|^2
%\le \sum_{i,j=1}^N \|\psi_i\|_{L^2}^2\|\phi_j\|_{L^2}^2 \\
%&=& \left(\sum_{i=1}^N\|\psi_i\|_{L^2}^2\right)\left(\sum_{j=1}^N\|\phi_j\|_{L^2}^2\right) =
%\vertiii \Psi \vertiii^2\vertiii \Phi\vertiii^2,
%\end{eqnarray*}
%and
%\begin{eqnarray*}
%\vertiii \Psi A \vertiii^2 &=& \sum_{j=1}^N \left\|\sum_{i=1}^N \psi_ia_{ij}\right\|_{L^2}^2 = \sum_{j=1}^N\int_{\mathbb{R}^3}
% \left(\sum_{i=1}^N \psi_i(r)a_{ij}\right)^2 dr \\
%&\le& \sum_{j=1}^N \left( \left(\sum_{i=1}^N\|\psi_i\|_{L^2}^2\right) \left(\sum_{i=1}^N a_{ij}^2\right) \right)
%= \vertiii \Psi \vertiii^2\|A\|_F^2.
%\end{eqnarray*}
%The conclusions \eqref{norm-oo} and \eqref{norm-om} follow by taking the square roots of the two inequalities respectively.

%Now we turn to prove \eqref{norm-conc3}. If $\langle \Phi^T\Phi \rangle=I_N$, then
%\begin{eqnarray*}
%\vertiii\Phi A \vertiii^2=\text{tr}(A^T\langle \Phi^T\Phi \rangle A)= \text{tr}(A^TA)=||A||_F^2.
%\end{eqnarray*}
% This completes the proof.
%\end{proof}

The feasible set of constrained problem \eqref{Emin} is a Stiefel manifold, which
is defined as
\begin{equation}
\mathcal{M}^N = \{U
= (u_i)_{i=1}^N|u_i\in H^1(\mathbb{R}^3),   U^TU  = I_N\}.
\end{equation}

We  see from \eqref{energy} that $E(U) = E(UP)$ for any $P \in \mathcal{O}^{N\times N}$. To get rid of the non-uniqueness, we consider the problem on the Grassmann manifold,
which is the quotient of the Stiefel manifold and is defined as follows
\begin{equation*}
\mathcal{G}^N = \mathcal{M}^N/\sim.
\end{equation*}
Here, $\sim$ denotes the equivalence relation and is defined as: we say  $\hat{U} \sim U$,  if there exists $P \in \mathcal{O}^{N\times N}$, such that $\hat{U} = UP$. For any $U \in \mathcal{M}^N$, we denote the equivalence class  by $[U]$, that is,
\begin{equation*}
[U] = \{U P: P \in  \mathcal{O}^{N\times N}\}.
\end{equation*}
 Then the problem
\eqref{Emin} on the Grassmann manifold is
\begin{equation}\label{Emin-Grass}
\inf_{[U] \in \mathcal{G}^N} \ \ \ E(U),
\end{equation}
where $E(U)$ is defined by \eqref{energy}.

Now, we define the distance on the Grassmann manifold $\mathcal{G}^N$, which will be used
in our analysis. Let $[\Psi], [\Phi] \in \mathcal{G}^N$, with
 $\Psi = (\psi_1,\cdots,\psi_N)\in \mathcal{M}^N$, and
$\Phi = (\phi_1,\cdots,\phi_N)\in \mathcal{M}^N$, we define the distance between $[\Psi]$ and $[\Phi]$ on $\mathcal{G}^N$ by
\begin{equation}\label{dist}
\text{dist}([\Psi], [\Phi]) = \min_{P\in \mathcal{O}^{N\times N}} \vertiii \Psi - \Phi P \vertiii.
\end{equation}
The following result tells us how to get $P$ such that the right-hand side of \eqref{dist} achieves its minimum, which is in principle shown in \cite{EAS},
but without proof. For completeness, we present a proof here.
\begin{lemma}\label{lema-dist}
Let $ASB^T$ be the singular value decomposition(SVD) of the matrix
$  \Psi^T\Phi  $. \iffalse There holds $P = BA^T$. \fi Then $P_0 = BA^T$ minimizes the right-hand side of \eqref{dist}.
\end{lemma}
\begin{proof}
For any $P \in \mathcal{O}^{N\times N}$, we derive from \eqref{norm-o} and
$\Psi, \Phi \in \mathcal{M}^N$ that
\begin{equation*}
\begin{split}
\vertiii \Psi - \Phi P \vertiii^2 & = 2N - \text{tr}((\Psi^T \Phi) P) -
                       \text{tr}(P^T(\Phi^T \Psi)) \\
                    & =2N - \text{tr}(ASB^TP) - \text{tr}(P^TBSA^T) \\
                    & =2N - \text{tr}(SC) - \text{tr}(C^TS),
\end{split}
\end{equation*}
where $C = B^TPA \in \mathcal{O}^{N\times N}$. It is easy to verify
that the minimum will achieve at $C = I_N$, which means that $P_0 = BA^T$ minimizes the right-hand side of \eqref{dist}.
\end{proof}

For $[U] \in \mathcal{G}^N$, the tangent space on the Grassmann manifold is defined as the following set \cite{SRNB}
\begin{equation}\label{tan}
\mathcal{T}_{[U]}\mathcal{G}^N = \{W\in (H^1(\mathbb{R}^3))^N|
W^TU   = 0 \in \mathbb{R}^{N\times N}\} =
(\text{span}\{u_1,\cdots,u_N\}^{\bot})^N.
\end{equation}

We understand that $E(U)$ is differentiable (as a functional in $(H^1(\mathbb{R}^3))^N$)
when the exact exchange-correction functional
is replaced by some approximation. In such a case, we denote $\nabla E(U)$ the gradient of $E(U)$.
We have that $\nabla E(U)=(E_{u_1},\cdots,E_{u_N})\in (H^{-1}(\mathbb{R}^3))^N$,
where $E_{u_i}$ is the derivative of $E(U)$ to the $i$-th orbital. It is easy to see that
\begin{equation} \label{der}
E_{u_i} = \mathcal{H}(\rho) u_i.
\end{equation}
Here,
%the Kohn-Sham Hamiltonian operator $\mathcal{H}(\rho): H^1(\mathbb{R}^3) \to H^{-1}(\mathbb{R}^3)$ is defined by

\begin{equation*}
\mathcal{H}(\rho) = -\frac{1}{2}\Delta + V_{ext} + \int_{\mathbb{R}^3}\frac{\rho(r')}{|r-r'|}dr' +
v_{xc}(\rho)
\end{equation*}
is the Kohn-Sham Hamiltonian operator which is from $H^1(\mathbb{R}^3)$ to $H^{-1}(\mathbb{R}^3)$
with the exchange correlation potential $$v_{xc}(\rho)
= \frac{\delta(\rho\varepsilon_{xc}(\rho))}{\delta\rho}.$$
%Let $\nabla E(U) = (E_{u_1}, \cdots, E_{u_N})= \mathcal{H}(\rho)U$.
%Let $\nabla E(U) = (E_{u_1}, \dots, E_{u_N})$. We see from \cite{EAS} that
For convenience, we may view $\nabla E(U)$ as an element in $(H^{1}(\mathbb{R}^3))^N$ in our following analysis\footnote[1]{For any $F\in H^{-1}(\mathbb{R}^3)$, there exists a unique element $\tilde{F}\in H^1(\mathbb{R}^3)$ satisfying $$F(V)=\langle \tilde{F},V\rangle, \forall \ V\in H^{1}(\mathbb{R}^3).$$ For simplicity, we sometimes view $F$ as $\tilde{F}$ in this paper.}.

We see from \cite{EAS} that the gradient $\nabla_G E(U)$ at
$[U]$ on the Grassmann manifold $\mathcal{G}^N$ is %defined as the
%element of $\mathcal{T}_{\lfloor U \rfloor}\mathcal{G}^N$ that satisfies
%\begin{equation*}
%\text{tr}\langle \nabla E(U)^T W \rangle  = \text{tr}\langle \nabla_G E(U)^T W \rangle , \ \forall \
%W \in \mathcal{T}_{\lfloor U \rfloor}\mathcal{G}^N.
%\end{equation*}
%A simple verification shows that \cite{EAS, SRNB}
\begin{equation*}
\begin{split}
    \nabla_G E(U) & = (I - UU^T)\nabla E(U) \\
    & = \nabla E(U) - U(U^T\nabla E(U)) ,
\end{split}
\end{equation*}
%{\color{red}{where $U^T:L^2(\mathbb{R}^3)\to \mathbb{R}^N$ is defined as
%\begin{equation*}
%U^T\phi=(\langle u_1,\phi\rangle,\langle u_2,\phi\rangle,\cdots,\langle u_N,\phi\rangle)^T, \forall\phi\in V_{N_g}
%\end{equation*}}}
and therefore
\begin{equation}\label{gra}
\nabla_G E(U)= \nabla E(U) - U\Sigma = \mathcal{H}(\rho)U - U\Sigma,
\end{equation}
where $\Sigma = U^T\nabla E(U)= U^T(\mathcal{H}(\rho)U)$ is symmetric since
\iffalse the Hamiltonian \fi$\mathcal{H}(\rho)$ is a symmetric operator.

The Hessian of $E(U)$ on the Grassmann manifold is defined as \cite{EAS}
\begin{equation*}
\text{Hess}_GE(U)[V, W] = \text{tr} (V^T
E''(U) W) - \text{tr}( V^TW\Sigma), \forall \ W, V \in
\mathcal{T}_{[U]}\mathcal{G}^N,
\end{equation*}
%where $E''(U)$ is the second order derivative of E with respect to $U$ in contribution sense.
where $(E''(U) W)_i= \sum\limits_j(E_{u_i})_{u_j}w_j$.
Therefore,
%The exact Hessian is \cite{WMUZ}
%\begin{eqnarray}
%\text{Hess}_GE(U)(V, W) &=& \text{tr} \langle V^T\mathcal{H}(\rho)
%W\rangle - \text{tr}\langle V^TW\Sigma\rangle \notag \\
%&+& 2\int_{\mathbb{R}^3}\int_{\mathbb{R}^3}\frac{(\sum_iu_i(r)v_i(r))(\sum_ju_j(r')w_j(r'))}
%{|r-r'|}drdr' \label{hes}\\
%&+& 2\int_{\mathbb{R}^3}\frac{\delta^2(\varepsilon_{xc}(\rho)\rho)}{\delta\rho^2}(r)
%(\sum_iu_i(r)v_i(r))(\sum_ju_j(r)w_j(r))dr. \notag
%\end{eqnarray}
\begin{equation}\label{hes}
\begin{split}
\text{Hess}_GE(U)[V, W] &= \text{tr} ( V^T\mathcal{H}(\rho)
W) - \text{tr}( V^TW\Sigma) \\
&+ 2\int_{\mathbb{R}^3}\int_{\mathbb{R}^3}\frac{(\sum_iu_i(r)v_i(r))(\sum_ju_j(r')w_j(r'))}
{|r-r'|}drdr' \\
&+ 2\int_{\mathbb{R}^3}\frac{\delta^2(\varepsilon_{xc}(\rho)\rho)}{\delta\rho^2}(r)
(\sum_iu_i(r)v_i(r))(\sum_ju_j(r)w_j(r))dr
\end{split}
\end{equation}
provided that the approximated exchange-correlation functional is second order differentiable. The LDA, which satisfies the smoothness conditions when $\rho > 0$\footnote[2]{Although it is of physics, it is still open whether $\rho>0$ \cite{FHH}.}, is used in our implementation. We will see from the numerical results in  Appendix B that the last two terms in \eqref{hes} are small compared with the whole term,
 which means that they  can be neglected.
 Therefore,  we can use the approximate Hessian to replace the exact Hessian, that is,
%
%Due to \eqref{der}, we may use some approximation of the second order derivatives
%of the energy functional, such as
%\begin{equation*}
%(E_{u_i})_{u_j} \approx \mathcal{H}(\rho)\delta_{ij}.
%\end{equation*}
%As a result,
\begin{equation}\label{hes-pra}
\text{Hess}_GE(U)[V, W] \approx \text{tr} ( V^T
\mathcal{H}(\rho)W) - \text{tr}( V^TW\Sigma).
\end{equation}

\subsection{The discretized Kohn-Sham model}
The finite dimensional discretizations for the Kohn-Sham model can be divided
into three classes: the plane wave method, the local basis set method, and the real
space method \cite{DGZZ, dai-zhou15}. Our approaches in this paper can be applied
to any given discretization method among  these three classes. For completeness,
we give a brief introduction of the discretized Kohn-Sham model here.

Let $\{\varphi_s\}_{s=1}^{N_g}$ be the basis for a finite dimensional space
$V_{N_g}\subset H^1(\mathbb{R}^3)$, where $N_g$ is the dimension of $V_{N_g}$. Then each discrete Kohn-Sham orbital
$u_i$ can be expressed as
\begin{equation*}
u_i(r) = \sum_{s = 1}^{N_g} c_{i,s} \varphi_s(r)
\end{equation*}
while  density $\rho(r) =\sum \limits_{i=1}^N\sum\limits_{s,t = 1}^{N_g}c_{i,s}c_{i,t}
\varphi_s(r)\varphi_t(r)$ and the Kohn-Sham total energy
\begin{align*}
E(U)&=\frac{1}{2}\sum_{i = 1}^N\sum_{s,t = 1}^{N_g}c_{i,s}c_{i,t}\int_{\mathbb{R}^3}
      \nabla\varphi_s(r)\nabla\varphi_t(r)dr \\
&  +  \frac{1}{2}\sum_{i = 1}^N\sum_{s,t = 1}^{N_g}c_{i,s}c_{i,t}\int_{\mathbb{R}^3}\int_{\mathbb{R}^3}
      \frac{\varphi_s(r)\varphi_t(r)\rho(r')}{|r-r'|}drdr' \\
& +\sum_{i = 1}^N\sum_{s,t = 1}^{N_g}c_{i,s}c_{i,t}\int_{\mathbb{R}^3} \varphi_s(r)V_{ext}(r)\varphi_t(r) dr \\
& +\sum_{i = 1}^N\sum_{s,t = 1}^{N_g}c_{i,s}c_{i,t}\int_{\mathbb{R}^3} \varepsilon_{xc}(\rho)(r)\varphi_s(r)
    \varphi_t(r)dr.
\end{align*}
%{\color{red}{The inner product of $u_i$, $u_j\in V_{N_g}$ is defined as
%\begin{equation*}
%\langle u_i, u_j \rangle = \sum_{s,t = 1}^{N_g}c_{i,s}c_{j,t}\int_{\mathbb{R}^3}
%     \varphi_s(r)\varphi_t(r)dr.
%\end{equation*}}}

If we use the finite difference discretization under a uniform grid, each
Kohn-Sham orbital $u_i$ can be represented as a vector of length $N_g$, where $N_g$ is
the degree of freedom for the computational domain, and we denote the set of all vectors
of length $N_g$ by $V_{N_g}$. Let $h_x, h_y, h_z$ be the mesh sizes
for the discretization in $x, y$ and $ z$ directions, respectively.
For simplicity, we also denote the discretized external potential (operator) by $V_{ext}$.
The Laplacian can be approximated by a matrix $L \in
\mathbb{R}^{N_g \times N_g}$ under a selected differential stencil. The density $\rho$ is
then a vector of length $N_g$, and $\rho =\displaystyle \sum_{i=1}^N u_i\odot u_i$, where $\odot$ is
the Hadamard product of two matrices (here for two vectors). The Hartree potential
$\displaystyle\int_{\mathbb{R}^3}\frac{\rho(r')}{|r-r'|}dr'$ can be represented by the product of
matrix $L^\dagger$ with $\rho$, where $L^\dagger$ is the generalized inverse of the discretized
Laplace operator. If we define the $L^2$ inner product of two vectors $\psi, \phi \in
V_{N_g}$ as follows
\begin{equation*}
\langle \psi, \phi \rangle = h_xh_yh_z\Big(\sum_{j=1}^{N_g} \psi(j)\phi(j)\Big),
\end{equation*}
then the discretized Kohn-Sham total energy can be expressed as
\begin{equation*}
E(U) = \frac{1}{2}\text{tr}( U^TLU) +\frac{1}{2} \langle \rho, L^\dagger
       \rho\rangle + \text{tr}( U^T V_{ext} U) + \langle \rho,
       \varepsilon_{xc}(\rho)\rangle.
\end{equation*}
%The gradient and Laplacian operator can be approximated by selected
%differential stencils.

Whichever discretization method we use, the minimization problem \eqref{Emin} becomes
\begin{equation}\label{dis-emin}
\begin{split}
&\min_{u_i\in V_{N_g}} \ \ \ E(U) \\
s.t.\ \  & \langle u_i, u_j \rangle = \delta_{ij}, \ \ 1\leq i,j\leq N,
\end{split}
\end{equation}
where $V_{N_g}$ is some $N_g$  dimensional space spanned by either some functions in $H^1(\mathbb{R}^{3})$ (for instance, resulted from the
finite element discretization) or some
vectors in $\mathbb{R}^{N_g}$ (for instance, resulted from the finite difference discretization).
We then introduce another  Stiefel manifold as
\begin{equation*}
\mathcal{M}^N_{N_g} = \{U = (u_i)_{i=1}^N|u_i\in V_{N_g},   U^TU  = I_N\},
\end{equation*}
and the corresponding  Grassmann manifold as  $\mathcal{G}^N_{N_g} = \mathcal{M}^N_{N_g}/\sim$,
where $\sim$ is the same as  that in the previous subsection. For $[U]\in
\mathcal{G}^N_{N_g}$,
the tangent space on the  Grassmann manifold $\mathcal{G}^N_{N_g}$ becomes
\begin{equation}\label{dis-tan}
\mathcal{T}_{[U]}\mathcal{G}^N_{N_g} = \{W\in (V_{N_g})^N|
W^TU   = 0\}.
\end{equation}
The gradient $\nabla_G E(U)$ and Hessian of $E(U)$ on the  Grassmann
manifold $\mathcal{G}^N_{N_g}$  have the same forms as that in \eqref{gra} and \eqref{hes} respectively.

Unless stated explicitly, the discussions in the rest of this paper are addressed for the
discretized model. For any element in $(V_{N_g})^N$, we define the corresponding norm $\vertiii \cdot \vertiii$
as \eqref{norm-o} with the inner product $\langle\cdot,\cdot\rangle$ in \eqref{norm-o} being replaced by the inner product of $V_{N_g}$.
% For vector in the Stiefel  manifold   $\mathcal{M}^N_{N_g}$, we define its norm $\vertiii \cdot \vertiii$
%as \eqref{norm-o} with the inner product being replaced by the inner product of $V_{N_g}$.
It is easy to see that the conclusions in Lemma \ref{lemma222} are valid for elements in $(V_{N_g})^N$. %and the corresponding norm $\vertiii \cdot \vertiii$.
For simplicity, we also refer to the conclusions in Lemma \ref{lemma222} in our following analysis.

Now, we introduce some assumptions that will be used in our analysis in Section \ref{sec-cvg}.
First, we need the following assumption.
\begin{assume}\label{assum_lip}
The gradient $\nabla E(U)$ of the energy functional is Lipschitz continuous. That is, there exists $L_0>0$
such that
\begin{equation*}
\vertiii \nabla E(U) - \nabla E(V) \vertiii \leq L_0\vertiii  U - V \vertiii, \ \ \forall \ U, V
\in \mathcal{M}^N_{N_g}.
\end{equation*}
\end{assume}

Note that the same assumption is used and discussed, for instance, in \cite{LWWUY, UWYKL}.
%Specifically, Lemma 3.3 shows that it
%is satisfied under a given exchange and correlation functional.
%We refer to  \cite{LWWUY} for more discussions.

From Assumption \ref{assum_lip}, there is a constant $C_0>0$, such that
\begin{equation}\label{der-bound}
\vertiii \nabla E(\Psi) \vertiii \leq C_0, \ \ \forall \ \Psi \in \mathcal{M}^N_{N_g},
\end{equation}
which implies
\begin{equation} \label{lip}
\vertiii \nabla_G E(U) - \nabla_G E(V) \vertiii \leq L_1\vertiii U - V \vertiii, \ \ \forall \ U, V
\in \mathcal{M}^N_{N_g},
\end{equation}
where $L_1 = 2L_0 + 2\sqrt{N}C_0$. That is, the gradient of the energy functional on the
Grassmann manifold is $L_1$-Lipschitz continuous, too.

We assume that there exists
a local minimizer $[U^*]$ of \eqref{dis-emin}, on which the following assumption will be
imposed.
\begin{assume}\label{assum_hess}
%Let $[U^*]$ be a local minimizer of \eqref{dis-emin}
There exists $\delta_1>0$,  such that
\begin{equation}\label{pos-bnd}
  \nu_1\vertiii D \vertiii^2 \le \textup{Hess}_GE(U)[D, D] \le
  \nu_2\vertiii D \vertiii^2 , \ \forall \ [U] \in B([U^*],\delta_1),
  \forall \ D \in \mathcal{T}_{[U]}\mathcal{G}^N_{N_g},
\end{equation}
%where $\nu_1$ and $\nu_2 > 0$ are constants, and
where $[U^*]$ is a local minimizer of \eqref{dis-emin}, $\nu_1, \nu_2 > 0$ are constants, and
\begin{equation*}
B([U],\delta) := \{[V]\in \mathcal{G}_{N_g}^N: \textup{dist}([V],
  [U]) \le \delta\}.
\end{equation*}
\end{assume}

We see that the first inequality in \eqref{pos-bnd} is nothing but the coercivity assumption and
 has been introduced in \cite{SRNB} at the minimizer of \eqref{dis-emin}. Here we
 require that it is true in a neighbourhood of the minimizer; while we refer to
 Assumption 4.1 and Lemma 4.2 in \cite{WMUZ} for a discussion of
the second inequality in \eqref{pos-bnd}.
Due to Assumption \ref{assum_hess}, we will see that the convergence we obtain  in Section
\ref{sec-cvg} is local (c.f., also, Section \ref{sec-cln}). Besides, the uniqueness of the local minimizer is guaranteed by Assumption \ref{assum_lip} and Assumption \ref{assum_hess}. We describe it as the following lemma and refer to Appendix A for its proof.
%In further, the uniqueness of the local minimizer is guaranteed by the following lemma. \iffalse See ``Appendix" for its proof.\fi We refer to Appendix A for its proof.

%Before stating the following lemma, we  first introduce some notation which will be used in our analysis. For a diagonal matrix $\mathcal{D} =\text{diag}(d_1, d_2, \cdots, d_N)$, we use $\sin{\mathcal{D}}$  to denote $\text{diag}(\sin{d_1}, \sin{d_2}, \cdots, \sin{d_N})$,  with $\cos{\mathcal{D}}$, $\arcsin{\mathcal{D}}$, $\arccos{\mathcal{D}}$ the similar meanings.

\begin{lemma}\label{lema-uniq}
If Assumptions \ref{assum_lip} and \ref{assum_hess} hold true, then there exists
only one stationary point in $B([U^*], \delta_1)$, which is $[U^*]$.
Furthermore, if $[V_n]\in B([U^*], \delta_1)$ $(n=1,2,\cdots)$
satisfies $\lim\limits_{n\to \infty} E(V_n) = E(U^*)$, then
\begin{equation}\label{dist-conv}
  \lim_{n\to \infty} \textup{dist}([V_n], [U^*]) = 0.
\end{equation}
\end{lemma}

\section{A conjugate gradient method} \label{sec-alg}
\newtheorem{remark}[theorem]{Remark}

In general, there are two main issues in a line search based optimization method, one is a search
direction, the other is a step size. For the unconstrained conjugate gradient methods, the conjugate
gradient  direction is used as the search direction, which is a linear combination of the
negative gradient direction and the previous search direction. For the manifold constrained
optimization problem \eqref{dis-emin},
% it is better to choose the search direction on the tangent
%space, and
 we need to keep each iterative point
on the constrained manifold. Consequently, we have to introduce our orthogonality preserving strategies.

\subsection{Orthogonality preserving strategies}\label{subsec-ort}
Let $U\in \mathcal{M}^N_{N_g}$, $\tau \in \mathbb{R}$ and
$D \in \mathcal{T}_{[U]}\mathcal{G}^N_{N_g}$ be the step size and search direction, respectively.
Note that
$\widetilde{U}(\tau) = U + \tau D$ may not be on the Stiefel manifold $\mathcal{M}^N_{N_g}$. %Similar to \cite{ZZWZ},
%We apply three approaches to deal with this problem.
We resort to some orthogonalization strategies to deal with this problem. Here, we choose the following three strategies: the WY strategy, the QR strategy, and the PD strategy, since the QR and PD strategies are the well-known and commonly used orthogonalization strategies, while the WY strategy is a recently proposed strategy and is proved to be very efficient \cite{WY, ZZWZ}.
We will give some brief introduction to the three strategies.

We first see the WY strategy.
%Our first strategy is the WY strategy, which is the constraint-preserving scheme proposed
%in \cite{WY}.
Note that WY strategy does not preserve the subspace spanned by the column vectors of $\widetilde{U}(\tau)$. First, we define
\begin{equation} \label{antisy}
\mathcal{W} = DU^T - UD^T.
\end{equation}
Since $D \in \mathcal{T}_{[U]}\mathcal{G}^N_{N_g}$, we have $\mathcal{W}U= D$.
We choose the next iterative point to be
\begin{equation}\label{up-wy}
U_{WY}(\tau) = U + \tau \mathcal{W}\Big(\frac{U + U_{WY}(\tau)}{2}\Big).
\end{equation}
Equation \eqref{up-wy} is an implicit definition of $U_{WY}(\tau)$.
Since $\mathcal{W}$ is antisymmetric, we have that the real parts of all the eigenvalues of $I - \frac{\tau}{2} \mathcal{W}$ are equal to 1.
As a result, $I - \frac{\tau}{2} \mathcal{W}$ is invertible.
\iffalse Note that the operator $\mathcal{W}$ in \eqref{antisy} is compact and antisymmetric. We observe
from the spectral theory for compact and antisymmetric operator that
$I - \frac{\tau}{2} \mathcal{W}$ is invertible.\fi
Then we can rewrite $U_{WY}$ explicitly as
\begin{equation} \label{WYthrm}
U_{WY}(\tau) = \Big(I - \frac{\tau}{2} \mathcal{W}\Big)^{-1}\Big(I + \frac{\tau}{2}
\mathcal{W}\Big)U.
\end{equation}
We see that $  U_{WY}(\tau)^TU_{WY}(\tau)  = I_N$ and $U_{WY}'(0) = D$. The
formula \eqref{WYthrm} is not easy to implement due to the inversion of
$I - \frac{\tau}{2} \mathcal{W}$. Fortunately,  we have the following helpful property for
 $U_{WY}(\tau)$.
\begin{lemma}\label{lemma-wy}
$U_{WY}(\tau)$ has the low rank expression
\begin{equation} \label{WYcal}
U_{WY}(\tau) = U + \tau D\Big(I_N+\frac{\tau^2}{4}  D^TD  \Big)^{-1}-
\frac{\tau^2}{2}U\Big(I_N+\frac{\tau^2}{4}  D^TD  \Big)^{-1}  (D^TD)  .
\end{equation}
\end{lemma}
\begin{proof}
Let $X = (D, U)$, $Y = (U, -D)$. Then $\mathcal{W} = XY^T$.
We see from Lemma 4 in \cite{WY} that
\begin{eqnarray}\label{cal-wy}
U_{WY}(\tau) &=& U + \tau X\Big(I_{2N} - \frac{\tau}{2}  Y^TX  \Big)^{-1}  Y^TU .
\end{eqnarray}
In addition, since $U \in \mathcal{M}^N_{N_g}$ and
$D \in \mathcal{T}_{[U]}\mathcal{G}^N_{N_g}$, we have $  U^TU   = I_N$
and $  D^TU   =   U^TD   = 0$. Therefore, the matrix
$I_{2N} - \frac{\tau}{2}   Y^TX   $ is invertible and
\begin{eqnarray*}
&&\Big(I_{2N} - \frac{\tau}{2}   Y^TX  \Big)^{-1}=
\left(\begin{array}{cc}
 I_N & -\frac{\tau}{2} I_N \\
 \frac{\tau}{2}  D^TD   & I_N
\end{array}\right)^{-1}  \\
&=&
\left(\begin{array}{cc}
 (I_N + \frac{\tau^2}{4}  D^TD  )^{-1} &
 \frac{\tau}{2}(I_N + \frac{\tau^2}{4}  D^TD  )^{-1} \\
 -\frac{\tau}{2}(I_N + \frac{\tau^2}{4}  D^TD  )^{-1}  (D^TD)   &
 (I_N + \frac{\tau^2}{4}  D^TD  )^{-1}
\end{array}\right).
\end{eqnarray*}
%Note that there holds (see  Lemma 4 in \cite{WY})
%\begin{eqnarray}\label{cal-wy}
%U_{WY}(\tau) &=& U + \tau X\Big(I_{2N} - \frac{\tau}{2} \langle Y^TX \rangle\Big)^{-1}\langle Y^TU\rangle.
%\end{eqnarray}
%We can use the Sherman-Morrison-Woodbury formula for $\Big(I - \frac{\tau}{2}\mathcal{W}\Big)^{-1}$.
Thus, we have
\begin{eqnarray*}
U_{WY}(\tau) = U + \tau D\Big(I_N+\frac{\tau^2}{4}  D^TD  \Big)^{-1}-
  \frac{\tau^2}{2}U\Big(I_N+\frac{\tau^2}{4}  D^TD  \Big)^{-1}  (D^TD)  .
\end{eqnarray*}
%$Y^TU = \left(\begin{array}{c} I_N  \\ 0 \end{array}\right)$ to complete the proof.
This completes the proof.
\end{proof}

We should point out that \eqref{WYcal} is more stable than \eqref{cal-wy} in
computation, since $I_N+\frac{\tau^2}{4}  D^TD  $ is symmetric positive definite,
its inversion is stable by using the Cholesky factorization.
Our numerical experiments also show that formula \eqref{WYcal} preserves the orthogonality
of the orbitals very well and there is no need to perform the reorthogonalization of the
orbitals in application.

Different from the WY strategy, the QR and PD strategies are to orthogonalize $\widetilde{U}(\tau) = U + \tau D$ directly.
First, we have the following lemma.
 \begin{lemma}\label{lemma-pos}
%The matrix $\langle \widetilde{U}(\tau)^T\widetilde{U}(\tau)\rangle$ is symmetric positive
%definite, and satisfies
%\begin{equation}\label{est-eigen}
%1 \le \lambda\left(\langle \widetilde{U}(\tau)^T\widetilde{U}(\tau)\rangle\right) \le 1 +
%\tau^2\vertiii D \vertiii^2,
%\end{equation}
%where $\lambda\left(\langle \widetilde{U}(\tau)^T\widetilde{U}(\tau)\rangle\right)$ is any
%eigenvalue of the matrix $\langle \widetilde{U}(\tau)^T\widetilde{U}(\tau)\rangle$.
Let $\sigma\left(  \widetilde{U}(\tau)^T\widetilde{U}(\tau) \right)$ be the set of
eigenvalues of $ \widetilde{U}(\tau)^T\widetilde{U}(\tau) $. Then
\begin{equation}\label{est-eigen}
\sigma\left(  \widetilde{U}(\tau)^T\widetilde{U}(\tau) \right)
\subset[1,1+\tau^2\vertiii D \vertiii^2].
\end{equation}
\end{lemma}
\begin{proof}
Since $D \in \mathcal{T}_{[U]}\mathcal{G}^N_{N_g}$, we have
\begin{equation}\label{eq111}
  \widetilde{U}(\tau)^T\widetilde{U}(\tau)  = I_N
+ \tau^2  D^T D  .
\end{equation}
Note that $  D^T D  $ is symmetric semi-positive
definite, the  smallest eigenvalue of $ \widetilde{U}(\tau)^T\widetilde{U}(\tau) $ is larger than 1.
 \iffalse the first inequality holds.\fi On the other hand, the largest eigenvalue of
$  D^T D  $ is not larger than its Frobenius norm $\|  D^T D  \|_F$, we
get the second inequality by \eqref{norm-oo}.
\end{proof}

Due to \eqref{est-eigen}, the matrix $  \widetilde{U}(\tau)^T\widetilde{U}(\tau) $
is well conditioned under a suitable step size $\tau$, which means that it is easy to perform
the orthogonalization of the matrix.

%Therefore, our second strategy is the QR strategy,
We now see the  QR strategy, which
performs the orthogonalization by the QR factorization,
\begin{equation*}
\widetilde{U}(\tau) =  Q(\tau)R(\tau),
\end{equation*}
and then set $U_{QR}(\tau)$ to be the column-orthogonal orbitals $Q$, that is,
\begin{equation*}
U_{QR}(\tau) = Q(\tau)=\widetilde{U}(\tau)R(\tau)^{-1}.
\end{equation*}
In other words, $\widetilde{U}(\tau)=U_{QR}(\tau)R(\tau)$ and hence
\begin{equation*}
R(\tau)^TR(\tau)= \widetilde{U}(\tau)^T\widetilde{U}(\tau) =I+\tau^2  D^TD .
\end{equation*}
%\begin{equation*}
%U_{QR}(\tau) = \text{qr}(\widetilde{U}(\tau)) = \text{qr}(U + \tau D),
%\end{equation*}
%where $\text{qr}(\widetilde{U}(\tau))$ is the column-orthogonal orbitals $Q$ corresponding
%to the QR factorization of $\widetilde{U}(\tau) = QR$.
We can carry out the orthogonalization by the Cholesky factorization. Suppose lower triangular
matrix $L(\tau)$ with positive diagonal elements satisfies
\begin{equation}\label{def-L}
L(\tau)L(\tau)^T =   \widetilde{U}(\tau)^T\widetilde{U}(\tau) ,
\end{equation}
then we have $R(\tau)=L(\tau)^T$ and
\begin{equation}\label{Proj}
U_{QR}(\tau) =   \widetilde{U}(\tau)L(\tau)^{-T}.
\end{equation}

%Our third strategy is the PD strategy,
%At last, we now see that PD strategy
At last, we turn to see the PD strategy, which
performs the orthogonalization by
the polar decomposition. Since
$  \widetilde{U}(\tau)^T\widetilde{U}(\tau)$ is positive definite,
$\big(\widetilde{U}(\tau)^T\widetilde{U}(\tau)\big)^
{-\frac{1}{2}}$ is well defined, and
\begin{equation}\label{Polar}
U_{PD}(\tau) = \widetilde{U}(\tau)
\big(\widetilde{U}(\tau)^T\widetilde{U}(\tau)\big)^{-\frac{1}{2}}
=\widetilde{U}(\tau)\big(I_N + \tau^2  D^T D  \big)^{-\frac{1}{2}}.
\end{equation}
Here, $\big(\widetilde{U}(\tau)^T\widetilde{U}(\tau)\big)^{-\frac{1}{2}}$
can be calculated by the eigen-decomposition of
$  \widetilde{U}(\tau)^T\widetilde{U}(\tau) $. That is,
suppose $  \widetilde{U}(\tau)^T\widetilde{U}(\tau)  = P \Lambda P^T$,
where $P \in \mathcal{O}^{N\times N}$, $\Lambda$ is diagonal, we get
$\big(\widetilde{U}(\tau)^T\widetilde{U}(\tau)\big)^{-\frac{1}{2}}
= P \Lambda^{-\frac{1}{2}}P^T$.

It is easy to see that
\begin{equation} \label{dtau}
U'_{QR}(0) = D, \ U'_{PD}(0) = D.
\end{equation}

For convenience, we introduce a macro $\text{ortho}(U,D,\tau)$ to denote
 one step starting from point $U \in \mathcal{M}^N_{N_g}$ with search direction $D$ and step size $\tau$ to next  point, which is also in $\mathcal{M}^N_{N_g}$. For the three strategies introduced above, the definition for  $\text{ortho}(U,D,\tau)$ is as follows:

%More generally, we can denote the next iterative point obtained by any orthogonality preserving strategy as $\text{ortho}(U,D,\tau)$.
%Specifically, it has the following forms for the above three strategies:
\begin{itemize}
      \item{for WY:}\begin{eqnarray*}
       \text{ortho}(U,D,\tau) &=& U + \tau D\Big(I_N+\frac{\tau^2}{4}
       D^TD  \Big)^{-1} \\
       && - \frac{\tau^2}{2}U\Big(I_N+\frac{\tau^2}{4}
         D^TD  \Big)^{-1}  (D^TD)  ;
      \end{eqnarray*}
      \item{for QR:}
      \begin{equation*}
          \text{ortho}(U,D,\tau) = (U + \tau D)L^{-T},
       \end{equation*}
       where $L$ is the lower triangular matrix such that
       \begin{equation*}
         LL^T = I_N + \tau^2
        D^T D   ;
       \end{equation*}
       \item{for PD:}\begin{equation*}
         \text{ortho}(U,D,\tau)=(U + \tau D)\big(I_N + \tau^2
        D^T D  \big)^{-\frac{1}{2}}.
      \end{equation*}\\
      \end{itemize}

\subsection{The step size strategy} \label{subsec-step}
Note that it is too expensive
to use the exact line search in electronic structure calculations, we introduce our step size strategy in this subsection, in which the Hessian
of the energy functional will be used.
Suppose we have $U_n\in \mathcal{M}^N_{N_g}$, step size $\tau \in \mathbb{R}$,
and search direction
$D_n \in \mathcal{T}_{[U_n]}\mathcal{G}^N_{N_g}$,
satisfying $\text{tr}( \nabla_G E(U_n)^TD_n ) \le 0$.
Expanding $E(U_n + \tau D_n)$ at $U_n$ approximately, we obtain
\begin{equation}\label{step-expan}
\begin{split}
E(U_n + \tau D_n) \approx\ & E(U_n) + \tau \text{tr}(
\nabla_G E(U_n)^TD_n) \\
 & + \frac{\tau^2}{2}\text{Hess}_GE(U_n)[D_n, D_n].
\end{split}
\end{equation}
To ensure the reliability of \eqref{step-expan},  we should do some restrictions to the step size $\tau_n$, for example, we may restrict the
step size $\tau_n$ to satisfy
$\tau_n\vertiii D_n \vertiii \le \theta$,
where $0 < \theta < 1$ is a given parameter. Note that the right hand
side of \eqref{step-expan} is a quadratic
function of $\tau$, we choose $\tilde{\tau}_n$ to be the minimizer of the quadratic function
in the interval $(0, \theta/\vertiii D_n\vertiii]$, which can be divided into two cases
based on whether the following condition is satisfied or not
\begin{equation}\label{pos-hess}
\text{Hess}_GE(U_n)[D_n, D_n]>0.
\end{equation}
A simple calculation shows that
\begin{equation}\label{stepsize}
\tilde{\tau}_n=
\begin{cases}
\min\left(-\frac{\text{tr}( \nabla_G E(U_n)^T D_n)}{\text{Hess}_GE
(U_n)[D_n, D_n]}, \frac{\theta}{\vertiii D_n \vertiii}\right), & \mbox{if
 \eqref{pos-hess} holds},\\
\frac{\theta}{\vertiii  D_n\vertiii}, & \mbox{otherwise}.
\end{cases}
\end{equation}

In our analysis, to ensure the energy reduction, we need the backtracking for the
step size, that is
\begin{equation}\label{step-back}
\tau_n = t^{m_n}\tilde{\tau}_n,
\end{equation}
where $t\in(0,1)$ is a given parameter, $m_n$ is the smallest nonnegative integer to satisfy
\begin{align}\label{back-cond}
  E(U_{n+1}(\tau_n)) \leq  \ E(U_n) + \eta\tau_n(\nabla_G E(U_n)^T
  D_n),
\end{align}
where $0 < \eta < 1$ is a constant parameter. We see that \eqref{back-cond} will be satisfied when
$\tau_n$ is sufficiently small, which implies the existence of such $m_n$.
% As in Remark \ref{mark-pos}, a more important
%case is when \eqref{pos-hess} is satisfied.
%Assume $\lfloor U_n \rfloor \in B(\lfloor U^* \rfloor, \delta_1)$.
%Due to \eqref{pos-bnd}, we are able to choose a non
%negative integer $m_n$, such that
%\begin{align}\label{expan}
%  E(U_{n+1}(\tau_n)) \leq & \ E(U_n) + \tau_n\text{tr}\langle(\nabla_G E(U_n))^T
%  D_n\rangle \nonumber \\
%  & + \alpha(\tau_n)^2\text{Hess}_GE(U_n)(D_n, D_n),
%\end{align}
%where $\frac{1}{2}<\alpha<1$ is a constant parameter
%independent of $n$. In fact, note that
%\begin{align*}
%E(U_{n+1}(\tau)) = & \ E(U_n) + \tau\text{tr}\langle
%(\nabla_G E(U_n))^TD_n\rangle \\
% & + \frac{\tau^2}{2}\text{Hess}_GE(U_n)(D_n, D_n)
% + \mathcal{O}(\tau^3),
%\end{align*}
%and Hessian $\text{Hess}_GE(U_n)$ is positive definite, we see that \eqref{expan}
%can be satisfied when
%$\tau_n\ll 1$.
In a word, we define our step size strategy, the Hessian based strategy, as follows

\begin{center}
\begin{tabular}{|p{115mm}|}\hline
 \begin{center}
  {\bf Hessian based strategy$(\theta, t, \eta)$} \vskip 0.2cm
 \end{center}
 \begin{enumerate}
   %\item Given $t\in(0,1),\ \theta, \ \eta$.
   \item Choose $\tilde{\tau}_n$ by \eqref{stepsize}.
   \item Calculate the step size
     \begin{equation*}
       \tau_n = t^{m_n}\tilde{\tau}_n,
     \end{equation*}
     where $m_n \in \mathbb{N}$ is the smallest nonnegative integer to satisfy
     \eqref{back-cond}.
 \end{enumerate} \\
 \hline
\end{tabular}
\end{center}

\begin{remark}\label{mark-pos}
{\rm %The condition $\text{Hess}_GE(U_n)(D_n, D_n) > 0$ is the second order optimality
%condition, and the authors in
%\cite{SRNB} assume that it is satisfied at the minimizer of \eqref{dis-emin}. However, it is not
%easy to see whether this condition is satisfied in general. For matrix eigenvalue problem,
%a clear description is given in \cite{SRNB}, that is, there is a gap between the $N$-th and
%$N+1$-th eigenvalue. To analyse the convergence of our algorithms, we also assume that this
%condition is satisfied, and we make a stricter assumption (see \eqref{pos-bnd} in
%Assumption \ref{assum_hess}).
%This condition can be viewed as some convex property of the energy functional
%at the minimizer, at least we may assume that it is satisfied in a neighbourhood of the minimizer.
%As stated in Assumption \ref{assum_hess},  \eqref{pos-hess} is the second order
%optimality condition, which has been used in
%\cite{SRNB}  at the minimizer of \eqref{dis-emin}.
Note that \eqref{pos-hess} is the second order optimality condition,
for an algebraic  eigenvalue problem, we understand that \eqref{pos-hess} is satisfied if there is a
gap between the $N$-th and $(N+1)$-th eigenvalues \cite{SRNB}.
 Under Assumption \ref{assum_hess}, \eqref{pos-hess} is satisfied if
$[U_n] \in B([U^*], \delta_1)$, and we will show how to ensure this property
in Section \ref{sec-cvg}.
In application, %it is not so easy to get the exact Hessian and we may use some approximated Hessian,
the condition
$\text{Hess}_GE(U_n)[D_n, D_n] > 0$ is satisfied for all examples in Section \ref{sec-num}
for the CG algorithms with the Hessian \eqref{hes} or the approximate  Hessian \eqref{hes-pra}.}
% if we use equation \eqref{hes-pra} to approximate the Hessian.} %For the completeness of our algorithm,
%if $\text{Hess}_GE(U_n)(D_n, D_n) \le 0$, we may
%choose $\tau_n = \theta$, where $\theta$ is a preset parameter.
\end{remark}

%The last step is to choose the conjugate gradient parameter $\beta_n$, which is used
%to generate the conjugate gradient direction from the previous conjugate gradient direction,
%as shown in the algorithms in subsection \ref{subsec-alg}.
\subsection{The choice of conjugate gradient parameter}
As is well known, the conjugate gradient direction is a linear combination of the negative gradient direction
and the previous search direction. Therefore, another important issue is to decide $\beta$,
the coefficient of the previous search direction (see step 3 of our algorithms in Section \ref{subsec-alg}),
which we call conjugate gradient parameter here. For a pure quadratic objective function, $\beta$ is
fixed. However, for a non-quadratic objective function, there are many different options. Here, we list
some famous choices as follows:
\begin{equation*}
\beta_n = \frac{\vertiii \nabla_G E(U_n)\vertiii^2}{\vertiii \nabla_G E(U_{n-1})\vertiii^2},
\end{equation*}
\begin{equation*}
\beta_n = \frac{\text{tr}( (\nabla_G E(U_n) - \nabla_G E(
   U_{n-1}))^T\nabla_G E(U_n))}{\vertiii \nabla_G E(U_{n-1}) \vertiii^2},
\end{equation*}
\begin{equation*}
\beta_n = \frac{\text{tr}( (\nabla_G E(U_n) - \nabla_G E(U_{n-1}))^T
              \nabla_G E(U_n) )}{\text{tr}( {F_{n-1}}^T(\nabla_G E(U_n)
              - \nabla_G E(U_{n-1})) )},
\end{equation*}
\begin{equation*}
\beta_n = \frac{\vertiii \nabla_G E(U_n) \vertiii^2}{
              \text{tr}({F_{n-1}}^T(\nabla_G E(U_n) - \nabla_G E(U_{n-1})))
              },
\end{equation*}
where $F_{n-1}$ is the previous conjugate gradient direction (see the algorithms in
Section \ref{subsec-alg}).
They are called Fletcher-Reeves (FR) formula, Polak-Ribi\'ere-Polyak (PRP) formula,
Hestenes-Stiefel (HS) formula and Dai-Yuan (DY) formula \cite{DaY}, respectively.
We choose the PRP formula in this paper, that is
\begin{equation}\label{cgprp}
\beta_n = \frac{\text{tr}( (\nabla_G E(U_n) - \nabla_G E(
   U_{n-1}))^T\nabla_G E(U_n))}{\vertiii \nabla_G E(U_{n-1}) \vertiii^2}.
\end{equation}
In fact, we have tested some  other choices, and find that there is no obvious difference for
the performance of using different choices in our numerical experiments.
\subsection{The conjugate gradient algorithms}\label{subsec-alg}
%Based on the three orthogonality preserving strategies we introduced in Section \ref{subsec-ort},
Based on some orthogonality preserving strategy,
the Hessian based strategy, and \eqref{cgprp},
we propose \iffalse WY, QR and PD based \fi our conjugate gradient algorithm as follows.%, and {\color{red}{denote}} them as
%Algorithm CG-WY, Algorithm CG-QR and Algorithm CG-PD, respectively.
%{\color{red}{These algorithms are described as follows.}}
%\begin{algorithm}
%\caption{CG-WY}
%\begin{algorithmic}[1]\label{Alg_WY}
%\STATE Given $\epsilon \in (0,1)$, initial data $U_0, \ s.t. \ \langle(U_0)^TU_0
%       \rangle = I_N$, $F_{-1}= 0$, let $n = 0$;
%\WHILE {($\vertiii \nabla_G E(U_n)\vertiii> \epsilon$)}
%\STATE Calculate  gradient $\nabla_G E(U_n)$ and conjugate gradient parameter
%       $\beta_n$ by \eqref{cgprp}, let $F_n = -\nabla_G E(U_n) +
%       \beta_nF_{n-1}$;
%
%\STATE Project the search direction on to the tangent space of $U_n$:
%       \begin{equation*}
%        D_n = F_n - U_n\langle(U_n)^TF_n\rangle;
%       \end{equation*}
%\STATE Set $F_n = -F_n\mbox{sign}(\text{tr}\langle (\nabla
%       E(U_n))^TD_n \rangle)$, \\
%       \ \ \ \ \ $D_n = -D_n\text{sign}(\mbox{tr}\langle(\nabla
%       E(U_n))^TD_n \rangle)$;
%\STATE Calculate the step size $\tau_n$ by the Hessian based strategy, and update
%       $U_{n+1}$ via  WY strategy, that is
%       \begin{eqnarray*}
%         U_{n+1} &=& U_n + \tau_n D_n\Big(I_N+\frac{(\tau_n)^2}{4}\langle
%          (D_n)^TD_n \rangle\Big)^{-1} \\
%          && - \frac{(\tau_n)^2}{2}U_n\Big(I_N+\frac{(\tau_n)^2}{4}
%          \langle (D_n)^TD_n \rangle\Big)^{-1}\langle (D_n)^TD_n \rangle;
%       \end{eqnarray*}
%\STATE  $n=n+1$;
%\ENDWHILE
%\end{algorithmic}
%\end{algorithm}
\renewcommand{\thealgocf}{} % delete number in the algorithm
\begin{algorithm}
%\caption{CG-WY/CG-QR/CG-PD}
\caption{Conjugate gradient method}
 Given $\epsilon, \theta, t, \eta \in (0,1)$, %$\epsilon \in (0,1)$, $\theta \in (0,1)$, $t \in (0,1)$, $\eta \in (0,1)$,
  initial data $U_0, \ s.t. \   U_0^TU_0  = I_N$, $U_{-1}=U_0$, $F_{-1}= 0$,
  calculate the gradient $\nabla_G E(U_0)$, let $n = 0$\;
 \While{$\vertiii \nabla_G E(U_n) \vertiii> \epsilon$}{
  Calculate the conjugate gradient parameter $\beta_n$ by \eqref{cgprp},
      let $F_n = -\nabla_G E(U_n)+\beta_nF_{n-1}$\;
  Project the search direction to the tangent space of $U_n$:
      $D_n = F_n - U_n(U_n^TF_n)$\;
  Set $F_n = -F_n\mbox{sign}(\text{tr} (\nabla_G E(U_n)^TD_n
      ))$, \ $D_n = -D_n\text{sign}(\mbox{tr}(\nabla_G
      E(U_n)^TD_n ))$\;
  Calculate the step size $\tau_n$ by the {\bf Hessian based strategy$(\theta, t, \eta)$};

   %using some orthogonality preserving strategy, that is,
   Set $U_{n+1}=\text{ortho}(U_n,D_n,\tau_n)$;
      %$U_{n+1}$ via WY, QR or PD strategy, that is,
      \iffalse \begin{itemize}
      \item{for WY:}\begin{eqnarray*}
       U_{n+1} &=& U_n + \tau_n D_n\Big(I_N+\frac{\tau_n^2}{4}\langle
       D_n^TD_n \rangle\Big)^{-1} \\
       && - \frac{\tau_n^2}{2}U_n\Big(I_N+\frac{\tau_n^2}{4}
       \langle D_n^TD_n \rangle\Big)^{-1}\langle D_n^TD_n \rangle,
      \end{eqnarray*}
      \item{for QR:}
      \begin{equation*}
         U_{n+1} = (U_n + \tau_nD_n)L_{n+1}^{-T},
       \end{equation*}
       where $L_{n+1}$ is the lower triangular matrix such that
       \begin{equation*}
         L_{n+1}{L_{n+1}}^T = I_N + \tau_n^2\langle
        D_n^T D_n \rangle ,
       \end{equation*}
       \item{for PD:}\begin{equation*}
        U_{n+1}=(U_n + \tau_nD_n)\big(I_N + \tau_n^2\langle
        D_n^T D_n \rangle\big)^{-\frac{1}{2}};
      \end{equation*}\\
      \end{itemize} \fi

  Let $n=n+1$, calculate the gradient $\nabla_G E(U_n)$\;
 }
\end{algorithm}
Step 5 of our conjugate gradient algorithm is to make sure that
\begin{equation}\label{cond1}
\text{tr} (\nabla_G E(U_n)^TD_n ) \le 0.
\end{equation}

By using the WY, QR or PD strategy in step 7, we obtain three different algorithms and denote them as CG-WY, CG-QR and CG-PD, respectively.

\section{Convergence analysis} \label{sec-cvg}

 In this section, we prove the convergence of our algorithms. First, we provide two estimations in Propositions \ref{diff} and \ref{der-diff}
 for our orthogonality preserving strategies introduced in Section \ref{subsec-ort}. %Note that Propositions \ref{diff} and \ref{der-diff} are independent of our algorithms.}} and show that every iteration generated by our algorithms satisfies Assumption \ref{assum_hess}.
 %We then turn to prove that the stepsize $\tau_n$ in our algorithms is bounded below and finally reach our convergence result by contradiction.}}

\begin{proposition}\label{diff}
For the WY strategy \eqref{WYthrm}, QR strategy \eqref{Proj}, and PD strategy \eqref{Polar},
there exists a constant $C_1 > 0$, such that
\begin{equation}\label{prop4.1}
\vertiii U_*(\tau) - U \vertiii \le C_1\tau \vertiii  D \vertiii, \ \forall \
\tau>0,
\end{equation}
where $U_*(\tau)$ represents $U_{WY}(\tau)$, or $U_{QR}(\tau)$, or $U_{PD}(\tau)$. Here, $C_1$ can be chosen as $2$.
\end{proposition}
\begin{proof}
(1) For the WY strategy, we obtain from \eqref{WYthrm} that
\begin{eqnarray*}
\vertiii U_{WY}(\tau) - U \vertiii &=& \vertIII \Big(I - \frac{\tau}{2} \mathcal{W}\Big)^{-1}
          \Big(I + \frac{\tau}{2} \mathcal{W}\Big)U - U \vertIII \\
          &=& \vertIII \Big(I - \frac{\tau}{2} \mathcal{W}\Big)^{-1}\tau \mathcal{W} U\vertIII \\
          &\leq & \tau \Big\|\Big(I - \frac{\tau}{2} \mathcal{W}\Big)^{-1}\Big\|_2 \vertiii D \vertiii.
\end{eqnarray*}
Since the operator $\mathcal{W}$ is antisymmetric, we have
\begin{eqnarray*}
\Big\|\Big(I - \frac{\tau}{2} \mathcal{W}\Big)^{-1}\Big\|_2^2 &=& \lambda_{\max}
\left(\Big(I - \frac{\tau}{2} \mathcal{W}\Big)^{-T}\Big(I - \frac{\tau}{2}
\mathcal{W}\Big)^{-1}\right) \\
&=& \lambda_{\max}\Big((I - \frac{\tau^2}{4} \mathcal{W}^2)^{-1}\Big).
\end{eqnarray*}
Note that the eigenvalues of $\mathcal{W}$ are imaginary numbers. We see that the eigenvalues
of $I - \frac{\tau^2}{4} \mathcal{W}^2$ are not smaller than $1$, and hence the eigenvalues of
$(I - \frac{\tau^2}{4} \mathcal{W}^2)^{-1}$  belong to \iffalse the interval \fi $(0, 1]$. Consequently,
\begin{equation}\label{WY-eq}
\Big\|\Big(I - \frac{\tau}{2} \mathcal{W}\Big)^{-1}\Big\|_2 \leq 1
\end{equation}
and
\begin{equation}\label{WY-conc1}
\vertiii U_{WY}(\tau) - U \vertiii \le \tau \vertiii D \vertiii, \ \forall \ \tau > 0.
\end{equation}

(2) For the QR strategy, we derive from \eqref{Proj} that
\begin{eqnarray*}
 U_{QR}(\tau)L(\tau)^T=\widetilde{U}(\tau)=U+\tau D, \\
\end{eqnarray*}
from which we have
\begin{eqnarray*}
U=U_{QR}(\tau)L(\tau)^T-\tau D.
\end{eqnarray*}
Therefore,
\begin{eqnarray*}
\vertiii U_{QR}(\tau) - U\vertiii
=  \vertIII U_{QR}(\tau)-U_{QR}(\tau)L(\tau)^T+\tau D\vertIII.
\end{eqnarray*}
By the triangle inequality and  \eqref{norm-conc3}, % the compatible inequality \eqref{norm-om},
 we obtain
\begin{eqnarray}\label{lemma-eq1}
\vertiii U_{QR}(\tau) - U\vertiii
&\le& \vertIII U_{QR}(\tau)(I_N-L(\tau)^T)\vertIII+\tau\vertiii D\vertiii \nonumber \\
&=& \|L(\tau)^T-I_N\|_F+\tau\vertiii D\vertiii.
% &\le& \vertiii U_{QR}(\tau)\vertiii\|L(\tau)^T-I_N\|_F+\tau\vertiii D\vertiii.  \\
\end{eqnarray}
%Note that \eqref{norm-o} implies
%\begin{equation*}
% \vertiii U_{QR}(\tau)\vertiii=\left(\text{tr}\langle U_{QR}(\tau)^T U_{QR}(\tau)\rangle\right)
% ^{\frac{1}{2}}=\sqrt{N}.
% \end{equation*}
%As a result, we have
%\begin{eqnarray*}
%\vertiii U_{QR}(\tau) - U\vertiii  \le   \sqrt{N}\|L(\tau)^T-I_N\|_F+\tau\vertiii D\vertiii.
%\end{eqnarray*}

We now turn to estimate $\|L(\tau)^T-I_N\|_F$.
Let
\begin{equation*}
L(\tau)^T-I_N=B(\tau) \ \text{with} \ B=(b_{ij})_{i,j=1}^N,
\end{equation*}
then we get from  \eqref{eq111} and \eqref{def-L} that
\begin{equation*}
L(\tau)L(\tau)^T=(I_N+B(\tau)^T)(I_N+B(\tau))=I_N+\tau^2  D^T D  ,
\end{equation*}
namely,
\begin{equation*}
B(\tau)+B(\tau)^T+B(\tau)^TB(\tau)=\tau^2  D^T D  .
\end{equation*}
Thus we conclude from \eqref{norm-o} that
\begin{eqnarray*}
\tau^2\vertiii D\vertiii^2=\text{tr}(\tau^2  D^T D  )
&=&\text{tr}(B(\tau)+B(\tau)^T+B(\tau)^TB(\tau)) \\
&=& 2\sum_{i=1}^Nb_{ii}+\|B(\tau)\|_F^2.
\end{eqnarray*}
Let $\{l_{ii}\}_{i=1}^N$ be the diagonal elements (i.e. the eigenvalues) of $L(\tau)^T$.
Then for any $i\in\{1,2,\cdots,N\}$, there exists an eigenvector
$\alpha_i \in \mathbb{R}^{N\times1}$, such that
\begin{equation*}
L(\tau)^T\alpha_i=l_{ii}\alpha_i, \ \alpha_i^T\alpha_i=1.
\end{equation*}
We observe that
\begin{equation*}
l_{ii}^2=\alpha_i^TL(\tau)L(\tau)^T\alpha_i=
1 + \tau^2\alpha_i^T  (D^T D)\alpha_i\geq1,
\ \forall \ i\in\{1,2,\cdots,N\},
\end{equation*}
which together with $l_{ii}>0$ yields $l_{ii}\geq1$. Then we see $b_{ii}=l_{ii}-1\geq 0,
\ \forall \ i\in\{1,2,\cdots,N\}$, which implies
\begin{equation*}
\tau^2\vertiii D\vertiii^2=2\sum_{i=1}^Nb_{ii}+\|B(\tau)\|_F^2\geq\|B(\tau)\|_F^2.
\end{equation*}
Consequently, there holds
\begin{equation}\label{lemma-eq2}
\|L(\tau)^T-I_N\|_F=\|B(\tau)\|_F\le\tau\vertiii D\vertiii.
\end{equation}

Combining \eqref{lemma-eq1} and \eqref{lemma-eq2}, we get
\begin{equation}\label{QR-conc1}
\vertiii U_{QR}(\tau) - U\vertiii\le \tau\vertiii D\vertiii+\tau\vertiii D\vertiii
=2 \tau\vertiii D\vertiii, \ \forall \ \tau\geq 0.
\end{equation}

(3) For the PD strategy, we derive from \eqref{Polar} that
\begin{eqnarray*}
&U_{PD}(\tau)\big(I_N + \tau^2  D^T D  \big)^{\frac{1}{2}}=\widetilde{U}(\tau)=U+\tau D \\
&U=U_{PD}(\tau)\big(I_N + \tau^2  D^T D  \big)^{\frac{1}{2}}-\tau D.
\end{eqnarray*}
As a result, there holds
\begin{eqnarray*}
\vertiii U_{PD}(\tau) - U\vertiii
&=&   \vertIII U_{PD}(\tau)-U_{PD}(\tau)\big(I_N + \tau^2  D^T D  \big)^{\frac{1}{2}}+\tau D \vertIII.
\end{eqnarray*}
By the triangle inequality and \eqref{norm-conc3}, %the compatible inequality \eqref{norm-om}
 we have
\begin{eqnarray}\label{lemma-eq3}
\vertiii U_{PD}(\tau) - U\vertiii &\le &  \vertIII U_{PD}(\tau)(I_N-\big(I_N + \tau^2  D^T D  \big)^{\frac{1}{2}})\vertIII+\tau\vertiii D\vertiii \nonumber \\
&=& \|\big(I_N + \tau^2  D^T D  \big)^{\frac{1}{2}}-I_N\|_F+\tau\vertiii D\vertiii.
 % &\le & \vertiii U_{PD}(\tau)\vertiii\|\big(I_N + \tau^2\langle D^T D  \big)^{\frac{1}{2}}-I_N\|_F+\tau\vertiii D\vertiii
\end{eqnarray}
%Note that \eqref{norm-o} implies
%\begin{equation}\label{PDP2}
% \vertiii U_{PD}(\tau)\vertiii=\left(\text{tr}\langle U_{PD}(\tau)^TU_{PD}(\tau)\rangle\right)
% ^{\frac{1}{2}}=\sqrt{N}.
% \end{equation}
%Therefore,
%\begin{eqnarray*}
%\vertiii U_{PD}(\tau) - U\vertiii  &\le&   \sqrt{N}\|\big(I_N + \tau^2\langle D^T D \rangle\big)^{\frac{1}{2}}-I_N\|_F+\tau\vertiii D\vertiii.
%\end{eqnarray*}

Now we start to estimate $\|\big(I_N + \tau^2  D^T D  \big)^{\frac{1}{2}}-I_N\|_F$.
Let
\begin{equation*}
\big(I_N + \tau^2  D^T D  \big)^{\frac{1}{2}}-I_N=\tilde{B}(\tau) \ \text{with} \ \tilde{B}=(\tilde{b}_{ij})_{i,j=1}^N,
\end{equation*}
then it holds
\begin{equation*}
\big(I_N + \tau^2  D^T D  \big)=(I_N+\tilde{B}(\tau))^2,
\end{equation*}
namely,
\begin{equation*}
\tilde{B}(\tau)^2+2\tilde{B}(\tau)=\tau^2  D^T D  .
\end{equation*}
Thus  we conclude from \eqref{norm-o} that
\begin{eqnarray*}
\tau^2\vertiii D\vertiii^2 = \text{tr}(\tau^2  D^T D  )
&=&\text{tr}(\tilde{B}(\tau)^2+2\tilde{B}(\tau)) Â¡Â¡=  2\sum_{i=1}^N\tilde{b}_{ii}+\|\tilde{B}(\tau)\|_F^2.
\end{eqnarray*}
It is easy to verify that $\tilde{B}(\tau)$ is symmetric semi-positive  definite, which means $\displaystyle \sum_{i=1}^N\tilde{b}_{ii}\ge0$, hence
\begin{equation*}
\tau^2\vertiii D\vertiii^2=2\sum_{i=1}^NÂ¡Â¡\tilde{b}_{ii}+\|\tilde{B}(\tau)\|_F^2\geq\|\tilde{B}(\tau)\|_F^2.
\end{equation*}
Consequently, there holds
\begin{equation}\label{PDP1}
\|\big(I_N + \tau^2  D^T D  \big)^{\frac{1}{2}}-I_N\|_F=\|\tilde{B}(\tau)\|_F\le\tau\vertiii D\vertiii.
\end{equation}
From \eqref{lemma-eq3} and \eqref{PDP1}, we obtain
\begin{equation}\label{PD-conc1}
\vertiii U_{PD}(\tau) - U\vertiii\le  \tau\vertiii D\vertiii+\tau\vertiii D\vertiii
= 2 \tau\vertiii D\vertiii, \ \forall \ \tau\geq 0.
\end{equation}

Therefore, we get \eqref{prop4.1} from \eqref{WY-conc1}, \eqref{QR-conc1}, and \eqref{PD-conc1}, Â¡Â¡
where $C_1$ can be chosen as $C_1 = 2$.
%\begin{equation*}
%C_1 = 2.
%\end{equation*}
\end{proof}

\begin{proposition}\label{der-diff}
For the WY strategy \eqref{WYthrm}, QR strategy \eqref{Proj} and PD strategy \eqref{Polar},
there exists a constant $C_2 > 0$, such that
\begin{equation}\label{prop4.2}
\begin{split}
 \vertiii U_{\ast}'(\tau) - U_{\ast}'(0) \vertiii = \vertiii U_{\ast}'(\tau) - D \vertiii
\le C_2\tau \vertiii D\vertiii^2, \ \forall \  \tau > 0,
\end{split}
\end{equation}
where $U_{\ast} (\tau)$ represents $U_{WY}(\tau)$, or $U_{QR}(\tau)$, or $U_{PD}(\tau)$,
and $U_{\ast}'(\tau)$ is the derivative of $U_{\ast}(\tau)$ with respect to $\tau$. Here, $C_2$ can be chosen as $1+\sqrt{2}$.
\end{proposition}

\begin{proof}
(1) For the WY strategy, we obtain from \eqref{WYthrm} that
\begin{eqnarray*}
(I-\frac{\tau}{2}\mathcal{W})U_{WY}(\tau) = (I+\frac{\tau}{2}\mathcal{W})U = U+\frac{\tau}{2}D.
\end{eqnarray*}
Thus,
\begin{eqnarray*}
U_{WY}'(\tau)=(I-\frac{\tau}{2}\mathcal{W})^{-1}\left(\frac{D+\mathcal{W}U_{WY}(\tau)}{2}\right).
\end{eqnarray*}
So we have
\begin{eqnarray*}
&& \vertiii U_{WY}'(\tau) - D\vertiii \\
&=& \vertiii(I-\frac{\tau}{2}\mathcal{W})^{-1}\left(\frac{D}{2}+\frac{\mathcal{W}U_{WY}(\tau)}{2}\right)-D \vertiii\\
&=& \vertiii(I-\frac{\tau}{2}\mathcal{W})^{-1}\left(\frac{D}{2}+\frac{\mathcal{W}U_{WY}(\tau)}{2} - D + \frac{\tau}{2}\mathcal{W}D\right) \vertiii,
\end{eqnarray*}
which together with the fact that $D = \mathcal{W}U$ leads to
\begin{eqnarray*}
&& \vertiii U_{WY}'(\tau) - D\vertiii \\
%&=& \vertiii(I-\frac{\tau}{2}\mathcal{W})^{-1}\left(\frac{D+\mathcal{W}U}{2}+\frac{\mathcal{W}U_{WY}(\tau)-\mathcal{W}U}{2}\right)-D \vertiii\\
&=&\vertiii(I-\frac{\tau}{2}\mathcal{W})^{-1}\left(\frac{\tau}{2}\mathcal{W}D+\frac{1}{2}\mathcal{W}(U_{WY}(\tau)-U)\right)\vertiii. \\
\end{eqnarray*}
Since
\begin{eqnarray*}
 \mathcal{W}D = (DU^T - UD^T)D = - U(D^T D)
\end{eqnarray*}
and
\begin{eqnarray*}
 U_{WY}(\tau) - U = \Big(I - \frac{\tau}{2} \mathcal{W}\Big)^{-1}
          \Big(I + \frac{\tau}{2} \mathcal{W}\Big)U - U = \Big(I - \frac{\tau}{2} \mathcal{W}\Big)^{-1}\tau \mathcal{W} U,
\end{eqnarray*}
we get
    \begin{eqnarray*}
&& \vertiii U_{WY}'(\tau) - D\vertiii \\
&=&\vertiii(I-\frac{\tau}{2}\mathcal{W})^{-1}\left(-\frac{\tau}{2}U(D^T D)+\frac{\tau}{2}\mathcal{W}(I-\frac{\tau}{2}\mathcal{W})^{-1}D\right)\vertiii \\
&=& \vertiii(I-\frac{\tau}{2}\mathcal{W})^{-1}\left(-\frac{\tau}{2}U(D^T D)-\frac{\tau}{2}(I-\frac{\tau}{2}\mathcal{W})^{-1}U(D^T D) \right)\vertiii.\\
\end{eqnarray*}
By the triangle inequality and \iffalse the compatible inequalities\fi \eqref{norm-oo} and  \eqref{norm-om}, we have
\begin{eqnarray*}
 \vertiii U_{WY}'(\tau) - D\vertiii &\le&\| (I-\frac{\tau}{2}\mathcal{W})^{-1}\|_2\left(\frac{\tau}{2}\vertiii U(D^T D) \vertiii +\frac{\tau}{2}\| (I-\frac{\tau}{2}\mathcal{W})^{-1}\|_2\vertiii U(D^T D) \vertiii\right),
\end{eqnarray*}
which together with  \eqref{norm-conc3} yields
\begin{eqnarray*}
 \vertiii U_{WY}'(\tau) - D\vertiii &\le&\| (I-\frac{\tau}{2}\mathcal{W})^{-1}\|_2\left(\frac{\tau}{2} \vertiii D\vertiii^2+\frac{\tau}{2}\| (I-\frac{\tau}{2}\mathcal{W})^{-1}\|_2 \vertiii D\vertiii^2\right).
\end{eqnarray*}
We then obtain from \eqref{WY-eq} that
\begin{eqnarray}\label{WY-conc2}
 \vertiii U_{WY}'(\tau) - D\vertiii \le   \tau\vertiii D\vertiii^2 .
\end{eqnarray}
%Â¡Â¡
%\begin{eqnarray*}
%&& \vertiii U_{WY}'(\tau) - D\vertiii \\
%&=& \vertiii(I-\frac{\tau}{2}\mathcal{W})^{-1}\frac{(D+\mathcal{W}U_{WY}(\tau))}{2}-D \vertiii\\
%&=& \vertiii(I-\frac{\tau}{2}\mathcal{W})^{-1}\frac{(D+\mathcal{W}U-\mathcal{W}U+\mathcal{W}U_{WY}(\tau))}{2}-D \vertiii\\
%&=&\vertiii(I-\frac{\tau}{2}\mathcal{W})^{-1}(\frac{\tau\mathcal{W}}{2}D+\frac{\mathcal{W}}{2}(U_{WY}(\tau)-U))\vertiii \\
%&\le&\| (I-\frac{\tau}{2}\mathcal{W})^{-1}\|(\frac{\tau}{2}\|\mathcal{W}\|_F\vertiii D\vertiii+\frac{1}{2}\|\mathcal{W}\|_F\vertiii U_{WY}(\tau)-U\vertiii).
%\end{eqnarray*}
%Notice that $\|\mathcal{W}\|_F=\|DU^T-UD^T\|_F\le2\vertiii U\vertiii\vertiii D\vertiii=2\sqrt{N}\vertiii D\vertiii$, we obtain from \eqref{WY-eq} and \eqref{WY-conc1} that
%\begin{eqnarray}\label{WY-conc2}
% \vertiii U_{WY}'(\tau) - D\vertiii \le \sqrt{N}\tau\vertiii D\vertiii^2+\sqrt{N}\tau\vertiii D\vertiii^2=2\sqrt{N}\tau\vertiii D\vertiii^2 .Â¡Â¡
%\end{eqnarray}

(2) For the QR strategy, we have from $U_{QR}(\tau)L(\tau)^T=U+\tau D$ that
\begin{equation*}
U_{QR}'(\tau)=(D-U_{QR}(\tau)L'(\tau)^T) L(\tau)^{-T},
\end{equation*}
and hence
\begin{eqnarray}\label{QR1}
&&\vertiii U_{QR}'(\tau)-D\vertiii  \notag \\
&=&\vertiii D(L(\tau)^{-T}-I_N)-U_{QR}(\tau)L'(\tau)^T L(\tau)^{-T}\vertiii  \notag \\
&=&\vertiii D L(\tau)^{-T}(I_N-L(\tau)^T)-U_{QR}(\tau)L'(\tau)^T L(\tau)^{-T}\vertiii  \notag \\
&\le&\vertiii D \vertiii \|L(\tau)^{-T}\|_2 \|L(\tau)^T-I_N\|_F +
       \|L'(\tau)^T L(\tau)^{-T}\|_F,
\end{eqnarray}
where \eqref{normF-conc3}, \eqref{norm-om}, and \eqref{norm-conc3} are used in the last line. Let
$L(\tau)^T=P\Sigma Q^{T}$ be the SVD of $L(\tau)^T$. We see from  \eqref{est-eigen} and \eqref{def-L}
that $\Sigma$ is a diagonal matrix with diagonal elements larger than 1.
Thus, we obtain from the definition of 2-norm that
\begin{equation}\label{QR2}
\|L(\tau)^{-T}\|_2=\|Q \Sigma^{-1} P^T\|_2=\|\Sigma^{-1}\|_2\le1.
\end{equation}
%
% Thus, we
%obtain from \eqref{ort-invar} that
%\begin{equation}\label{QR2}
%\|L(\tau)^{-T}\|_F=\|Q \Sigma^{-1} P^T\|_F=\|\Sigma^{-1}\|_F\le\sqrt{N}.
%\end{equation}
In further, we get from $L(\tau) L(\tau)^T=I+\tau^2  D^T D $
and the fact $L(\tau)$ is invertible that
\begin{equation*}
L'(\tau)^{T} L(\tau)^{-T}+L^{-1}(\tau)L'(\tau)=
2\tau L^{-1}(\tau)(D^T D) L(\tau)^{-T}
\end{equation*}
and
\begin{equation*}
L^{-1}(\tau)=L(\tau)^T (I+\tau^2  D^T D )^{-1}.
\end{equation*}
Since $L'(\tau)^{T}L(\tau)^{-T}$ is upper triangular and
$(L'(\tau)^{T}L(\tau)^{-T})^T=L^{-1}(\tau)L'(\tau)$, we have
\begin{eqnarray*}
\sqrt{2}\|L'(\tau)^{T} L(\tau)^{-T}\|_F&\le&\|L'(\tau)^{T} L(\tau)^{-T}+L^{-1}(\tau)L'(\tau)\|_F  \\
&=&2\tau\|L^{-1}(\tau)(D^T D) L(\tau)^{-T}\|_F   \\
&=&2\tau\|L(\tau)^T (I_N+\tau^2 D^T D)^{-1}( D^T D) L(\tau)^{-T}\|_F.
\end{eqnarray*}
Note that $  D^T D $ and $(I_N+\tau^2  D^T D )^{-1}$
are commutable, and both
\begin{equation*}
  L^{-1}(\tau)( D^T D) L(\tau)^{-T}
\end{equation*}
  and
  \begin{equation*}
  (I_N+\tau^2  D^T D  )^{-1}  (D^T D)
  \end{equation*}
are symmetry, we derive from \eqref{similar}, \eqref{normF-conc3}, % in Lemma \ref{lemma111}
and
\begin{eqnarray*}
\|(I+\tau^2   D^T D )^{-1}\|_2 \le 1
\end{eqnarray*}
that
\begin{eqnarray}\label{QR3}
\|L'(\tau)^{T}L(\tau)^{-T}\|_F&\le&\sqrt{2}\tau\|L(\tau)^T(I_N+\tau^2  D^T
D  )^{-1}(D^T D) L(\tau)^{-T}\|_F \notag\\
&=&\sqrt{2}\tau\|(I+\tau^2  D^T D  )^{-1}  D^T D  \|_F \notag  \\
&\leq&\sqrt{2}\tau\|(I+\tau^2  D^T D  )^{-1}\|_2  \|   D^T D  \|_F \notag  \\
&\le&\sqrt{2}\tau\vertiii D\vertiii^2.
\end{eqnarray}

Combining \eqref{lemma-eq2}, \eqref{QR1}, \eqref{QR2}, and \eqref{QR3}, we get
\begin{equation}\label{QR-conc2}
\vertiii U_{QR}'(\tau)-D\vertiii  \le \tau\vertiii D\vertiii^2+
\sqrt{2}\tau\vertiii DÂ¡Â¡\vertiii^2=(1+\sqrt{2})\tau\vertiii DÂ¡Â¡\vertiii ^2.
\end{equation}

(3) For the PD strategy, we have from $U_{PD}(\tau)\big(I_N + \tau^2  D^T D  \big)^{\frac{1}{2}}=U+\tau D$ that
\begin{equation*}
U_{PD}'(\tau)=\left(D-\tau U_{PD}(\tau)(D^T D) \big(I_N + \tau^2  D^T D  \big)^{-\frac{1}{2}}\right)\big(I_N + \tau^2  D^T D  \big)^{-\frac{1}{2}}.
\end{equation*}
Hence
\begin{eqnarray}\label{PD11}
&&\vertiii U_{PD}'(\tau)-D\vertiii  \notag \\
&=&\vertiii D(\big(I_N + \tau^2  D^T D  \big)^{-\frac{1}{2}}-I_N)-\tau U_{PD}(\tau)( D^T D ) \big(I_N + \tau^2  D^T D  \big)^{-1}\vertiii  \notag \\
&=&\vertiii D(\big(I_N + \tau^2  D^T D  \big)^{-\frac{1}{2}}-I_N)-\tau U_{PD}(\tau) \big(I_N + \tau^2  D^T D  \big)^{-1}  D^T D  \vertiii  \notag \\
&\le&\vertiii D\vertiii \|\big(I_N + \tau^2  D^T D  \big)^{-\frac{1}{2}}\|_2\|\big(I_N + \tau^2  D^T D  \big)^{\frac{1}{2}}-I_N\|_F+ \notag \\
&&\tau\|(I_N+\tau^2  D^T D  \big)^{-1}\| _2\|  D^T D \|_F\notag,
\end{eqnarray}
where \eqref{normF-conc3}, \eqref{norm-om}, and \eqref{norm-conc3} are used in the last inequality.
 We see from the definition of 2-norm that $\Big\|\Big(I_N+\tau^2  D^T D \Big)
^{-\frac{1}{2}}\Big\|_2 \le 1$ and $\Big\|\Big(I_N+\tau^2  D^T D \Big)
^{-1}\Big\|_2 \le 1$. Thus, we get from \eqref{PDP1} that
\begin{eqnarray}\label{PD-conc2}
\vertiii U_{PD}'(\tau)-D\vertiii  &\le& \vertiii D\vertiii \|\big(I_N + \tau^2  D^T D  \big)^{-\frac{1}{2}}\|_2\|\big(I_N + \tau^2  D^T D  \big)^{\frac{1}{2}}-I_N\|_F  \nonumber  \\
& & \tau  \|(I_N+\tau^2  D^T D  \big)^{-1}\| _2\|  D^T D \|_F  \nonumber \\
&\le& \tau\vertiii D\vertiii^2+\tau\vertiii D\vertiii^2=2\tau\vertiii D\vertiii^2.
%\end{split}
\end{eqnarray}
%
%
%
%We see from  Lemma \ref{lemma-pos} that $\Big\|\Big(I_N+\tau^2\langle D^T D\rangle\Big)
%^{-\frac{1}{2}}\Big\|_F \le \sqrt{N}$ and $\Big\|\Big(I_N+\tau^2\langle D^T D\rangle\Big)
%^{-1}\Big\|_F \le \sqrt{N}$. Thus, we get from \eqref{PDP2} and \eqref{PDP1} that
%\begin{eqnarray}\label{PD-conc2}
%%\begin{split}
%\vertiii U_{PD}'(\tau)-D\vertiii  &\le& \vertiii D\vertiii \|\big(I_N + \tau^2\langle D^T D \rangle\big)^{-\frac{1}{2}}\|_F\|\big(I_N + \tau^2\langle D^T D \rangle\big)^{\frac{1}{2}}-I_N\|_FÂ¡Â¡ \nonumber  \\
%& &Â¡Â¡+ \tau\vertiii U_{PD}(\tau)\vertiii  \|\langle D^T D\rangle\|_F\|(I_N+\tau^2\langle D^T D \rangle\big)^{-1}\| _F  Â¡Â¡\nonumber \\
%&\le& \sqrt{N}\tau\vertiii D\vertiii^2+N\tau\vertiii D\vertiii^2=(\sqrt{N}+N)\tau\vertiii D\vertiii^2
%%\end{split}
%\end{eqnarray}

Therefore, combining \eqref{WY-conc2}, \eqref{QR-conc2}, and \eqref{PD-conc2}, we obtain \eqref{prop4.2}, where $C_2$ can be chosen as
\begin{eqnarray*}
C_2 = \max&(1,
 1 + \sqrt{2},  2) = 1 + \sqrt{2}.
\end{eqnarray*}
\end{proof}

\begin{remark}
Similar conclusions as \eqref{prop4.1} and \eqref{prop4.2} for the WY strategy are
shown in \cite{JD}, and our conclusions  are obtained without using any requirement for $\tau$.
We note also that there are no similar estimations for either
QR or PD strategy in the literature.
\end{remark}

To use Assumption \ref{assum_hess} in our convergence proof, we should keep every
iteration point $[U_n] \in B([U^*], \delta_1)$. We now prove that every iteration point generated by our algorithms is in fact in $B([U^*], \delta_1)$. We obtain from Lemma \ref{lema-uniq} that
for any $\delta_2 \in (0, \delta_1/(1+\frac{C_1}{\nu_1}L_1))$, there exists an $E_0$ and the corresponding  level set
\begin{equation}\label{level-set-L}
\mathcal{L} = \{[U]\in \mathcal{G}_{N_g}^N: E(U)\le E_0\},
\end{equation}
 such that
\begin{equation}\label{level}
\{[U]: [U] \in \mathcal{L} \cap B([U^*], \delta_1)\}
\subset B([U^*], \delta_2).
\end{equation}
In our following analysis, we use a fixed $\delta_2 \in (0, \delta_1/(1+\frac{C_1}{\nu_1}L_1))$ and the
corresponding $E_0$. We have the following lemma.
%where $E_0$  is
%sufficiently close to $E(U^*)$.
%In fact, from the first inequality in \eqref{pos-bnd}, there
%is only one minimizer in $B(\lfloor U^* \rfloor, \delta_1)$, and \eqref{level} can be satisfied when $E_0$ .

\begin{lemma} \label{lema1}
Let Assumptions \ref{assum_lip} and \ref{assum_hess} hold true. For the sequence  $\{U_n\}_{n\in \mathbb{N}_0}$ generated by
Algorithm CG-WY, or Algorithm CG-QR, or Algorithm CG-PD,  if $[U_0] \in B([U^*], \delta_2) \cap \mathcal{L}$,
%and $m_n$  satisfies \eqref{back-cond},
then  $$[U_n] \in B([U^*], \delta_2) \cap \mathcal{L}, \forall \ n \in \mathbb{N}_{0}.$$
\end{lemma}
\begin{proof}
Let us prove the conclusion by induction. Since $[U_0] \in
B([U^*], \delta_2)\cap \mathcal{L}$,
we see that the conclusion is true for $n=0$. We assume that $[U_n] \in
B([U^*], \delta_2) \cap\mathcal{L}$, which implies
that $\text{Hess}_GE(U_n)[D_n, D_n] > 0$. Then we have from \eqref{stepsize} that
\begin{eqnarray*}
\tilde{\tau}_n \le -\frac{\text{tr} (\nabla_G E(U_n)^T D_n)}{\text{Hess}_GE
(U_n)[D_n, D_n]}.
\end{eqnarray*}
Therefore, from  Proposition \ref{diff} we have
\begin{align*}
\vertiii U_{n+1} - U_n\vertiii  & \le C_1 \tau_n\vertiii D_n\vertiii  \le C_1 \tilde{\tau}_n\vertiii D_n\vertiii  \\
                        & \le C_1\frac{|\text{tr} (\nabla_G E(U_n)^T D_n)|}
                         {\text{Hess}_GE(U_n)[D_n, D_n]}\vertiii D_n\vertiii  \\
                        & \le C_1\frac{\vertiii \nabla_G E(U_n)\vertiii  \vertiii D_n\vertiii }
                         {\text{Hess}_GE(U_n)[D_n, D_n]}\vertiii D_n\vertiii,
\end{align*}
which together with  Assumption \ref{assum_hess} leads to
\begin{align*}
\vertiii U_{n+1} - U_n\vertiii \le \frac{C_1}{\nu_1}\vertiii \nabla_G E(U_n)\vertiii.
\end{align*}
 We obtain from Lemma \ref{lema-dist} that there exists  $P_n \in \mathcal{O}^{N\times N}$,
such that
 \begin{eqnarray*}
 \text{dist}([U_n], [U^*]) = \vertiii U_n - U^* P_n \vertiii,
\end{eqnarray*}
which together with $\nabla_G E(U^*) = 0$ and Assumption
\ref{assum_lip} leads to
 \begin{eqnarray*}
  \vertiii U_{n+1} - U_n\vertiii &\le& \frac{C_1}{\nu_1}\vertiii \nabla_G E(U_n)- \nabla_G E(U^*)P_n\vertiii  \\
    &=& \frac{C_1}{\nu_1}\vertiii \nabla_G E(U_n)- \nabla_G E(U^*P_n)\vertiii \\
    &\le& \frac{C_1}{\nu_1}L_1\vertiii U_n - U^*P_n\vertiii  \le \frac{C_1}{\nu_1}L_1\delta_2,
\end{eqnarray*}
where the assumption $[U_n] \in B([U^*], \delta_2) \cap \mathcal{L}$ is used in the last inequality.
Consequently,
\begin{eqnarray*}
%\text{dist}([U_{n+1}], [U^*])  &\le& \text{dist}([U_{n+1}],
%[U_n]) + \text{dist}([U_n], [U^*]) \\
\text{dist}([U_{n+1}], [U^*]) &\leq & \vertiii U_{n+1}-U^*P_n\vertiii  \\
 &\leq & \vertiii U_{n+1}-U_n\vertiii+\vertiii U_n-U^*P_n\vertiii \\
  &\le& \vertiii U_{n+1} - U_n\vertiii + \delta_2 \le (1+\frac{C_1}{\nu_1}L_1)\delta_2\le \delta_1.
\end{eqnarray*}
We see from \eqref{back-cond} and the fact $\text{tr}(\nabla_G E(U_n)^TD_n ) \le 0$ (see \eqref{cond1}) that
\begin{align*}
  E(U_{n+1}(\tau_n)) \leq  \ E(U_n) + \eta\tau_n\text{tr}(\nabla_G E(U_n)^T
  D_n)  \leq E(U_n).
\end{align*}
Therefore we get $[U_{n+1}] \in B([U^*], \delta_1) \cap \mathcal{L}$.
Finally, we obtain from \eqref{level} that
$[U_{n+1}] \in B([U^*], \delta_2) \cap \mathcal{L}$ and complete the proof.
\end{proof}

\iffalse All the properties shown above are as preparations, we\fi We now turn to prove the convergence of our algorithms. The basic idea is as follows: We first prove that the step sizes \ $\tau_n \ $ used in our algorithms are bounded from below; then we show that if the limit inferior of $\vertiii \nabla_G E(U_n)\vertiii$ is larger than 0, then the conjugate gradient parameter $\beta_n$ must go down to 0 as $n$ goes up to $\infty$ under our assumptions; we finally reach our convergence result by contradiction.%We Now we show that the step size $\tau_n$ used in Algorithm CG-WY, Algorithm CG-QR, and Algorithm CG-PD is bounded below.
\begin{lemma}\label{lema-lbd}
  Let $\{U_n\}_{n\in \mathbb{N}_0}$ be a sequence  generated by
Algorithm CG-WY, or Algorithm CG-QR, or Algorithm CG-PD. If Assumption \ref{assum_lip} holds true,  then for the step size $\tau_n$, we have
\begin{equation}\label{stepsize-lb}
\tau_n \ge \min\left(\tilde{\tau}_n, \frac{2t(\eta-1)\text{tr}(\nabla_G E(U_n)^T D_n )}{(C_0C_2+L_0C_1)\vertiii D_n\vertiii ^2}\right).
\end{equation}
\end{lemma}
\begin{proof}Â¡Â¡
Let $U_{n+1}(s)$ be the update of $U_n$ which is generated by our Algorithm CG-WY, or Algorithm CG-QR, or Algorithm CG-PD but with the step size $\tau_n$ being replaced by $s$, and $(U_{n+1})'(s)$ be the derivative of $U_{n+1}(s)$ with respect to $s$. Then for $s \ge 0$, we obtain from \eqref{der-bound} and Proposition \ref{der-diff} that
\begin{eqnarray*}
&&\text{tr}(\nabla E(U_{n+1}(s))^T((U_{n+1})'(s) - (U_{n+1})'(0))) \\
&\le& \vertiii \nabla E(U_{n+1}(s))\vertiii  \vertiii (U_{n+1})'(s) - (U_{n+1})'(0)\vertiii  \le C_0C_2s\vertiii D_n\vertiii ^2,
\end{eqnarray*}
and from Assumption \ref{assum_lip} and Proposition \ref{diff} that
\begin{eqnarray*}
&&\text{tr}((\nabla E(U_{n+1}(s)) - \nabla E(U_n))^T(U_{n+1})'(0)) \\
&\le& \vertiii \nabla E(U_{n+1}(s)) - \nabla E(U_n)\vertiii \vertiii (U_{n+1})'(0)\vertiii  \\
&\le& L_0 \vertiii U_{n+1}(s) - U_n\vertiii \vertiii D_n\vertiii  \le L_0C_1 s \vertiii D_n\vertiii ^2.
\end{eqnarray*}
Due to $(U_{n+1})'(0) = D_n$, we have
\begin{eqnarray*}
 (\nabla E(U_n))^T(U_{n+1})'(0) = \nabla E(U_n)^T D_n ,
\end{eqnarray*}
which together with  the fact $  {U_n}^T D_n   = 0$ leads to
\begin{eqnarray*}
&& \nabla E(U_n)^T(U_{n+1})'(0)  = \nabla E(U_n)^T\left(\big(I-U_n{U_n}^T\big)D_n\right) \\
&=& \big(\big(I-U_n{U_n}^T\big)\nabla E(U_n)\big)^TD_n  =  \nabla_G E(U_n)^TD_n.
\end{eqnarray*}
Then, for any $\tau>0$,
\begin{eqnarray*}
&& E(U_{n+1}(\tau)) - E(U_n) \\
&=& \int_0^{\tau} \text{tr}
 (\nabla E(U_{n+1}(s))^T(U_{n+1})'(s)) ds \\
&=& \int_0^{\tau} \text{tr}
 \left(\nabla E(U_{n+1}(s))^T((U_{n+1})'(s) - (U_{n+1})'(0))\right) ds \\
 && + \int_0^{\tau} \text{tr}(
 (\nabla E(U_{n+1}(s)) - \nabla E(U_n))^T(U_{n+1})'(0)) ds \\
 && + \int_0^{\tau} \text{tr}
 (\nabla E(U_n)^T(U_{n+1})'(0)) ds.
\end{eqnarray*}
Hence,
\begin{equation}\label{hess-upbnd}
\begin{split}
& E(U_{n+1}(\tau)) - E(U_n) \\
\le& \int_0^{\tau} (C_0C_2 + L_0C_1) s\vertiii D_n\vertiii ^2 + \text{tr}
 (\nabla_G E(U_n)^T D_n ) ds \\
=& \tau\text{tr}(\nabla_G E(U_n)^T D_n ) +
\frac{C_0C_2+L_0C_1}{2}\tau^2 \vertiii D_n\vertiii ^2.
\end{split}
\end{equation}
Therefore, from \eqref{hess-upbnd}, we have that if
\begin{equation}\label{lb}
\tau \le \frac{2(\eta-1)\text{tr}(\nabla_G E(U_n)^T D_n )}
{(C_0C_2+L_0C_1)\vertiii D_n\vertiii ^2},
\end{equation}
then $\tau $ satisfies \eqref{back-cond}.

Besides, since $\{U_n\}_{n\in \mathbb{N}_0}$ are generated by
Algorithm CG-WY, or Algorithm CG-QR, or Algorithm CG-PD, we have that $\tau_n$ satisfies \eqref{back-cond} from the requirement for $\tau_n$Â¡Â¡
in the Hessian based strategy. Â¡Â¡
We now divide our proof into two cases.

First, we consider the case that $\tilde{\tau}_n$ satisfies \eqref{back-cond}. In this case,
 $\tau_n = \tilde{\tau}_n$, and of course we have
\begin{equation*}
\tau_n \ge \min\left(\tilde{\tau}_n, \frac{2t(\eta-1)\text{tr}(\nabla_G E(U_n)^T D_n )}{(C_0C_2+L_0C_1)\vertiii D_n\vertiii ^2}\right).
\end{equation*}

Then, we consider the case that $\tilde{\tau}_n$ does not satisfy \eqref{back-cond}. In this case, we must do the backtracking which implies the previous step size $\tau_nt^{-1}$ does not satisfy \eqref{back-cond}. And we claim that $\tau_nt^{-1}$ also does not satisfy \eqref{lb}, or else $\tau_nt^{-1}$ will satisfy \eqref{back-cond} and hence a contradiction. As a result
    \begin{equation*}
\tau_nt^{-1} \ge \frac{2(\eta-1)\text{tr}(\nabla_G E(U_n)^T D_n )}
{(C_0C_2+L_0C_1)\vertiii D_n\vertiii ^2},
\end{equation*}
which indicates that
\begin{equation*}
\tau_n \ge \frac{2t(\eta-1)\text{tr}(\nabla_G E(U_n)^T D_n )}
{(C_0C_2+L_0C_1)\vertiii D_n\vertiii ^2}.
\end{equation*}
Therefore,
\begin{equation*}
\tau_n \ge \min\left(\tilde{\tau}_n, \frac{2t(\eta-1)\text{tr}(\nabla_G E(U_n)^T D_n )}{(C_0C_2+L_0C_1)\vertiii D_n\vertiii ^2}\right).
\end{equation*}
This completes the proof.

\end{proof}

%Following \cite{LI}, we also need an assumption as follows
%\begin{assume}\label{assum_diverge}
%The series $$\sum_{n, D_n\ne 0} (1-\alpha t^{m_n})t^{m_n}$$ is divergent, namely
%\begin{equation*}
%\sum_{n, D_n\ne 0} (1-\alpha t^{m_n})t^{m_n} = \infty.
%\end{equation*}
%\end{assume}
%Note that if there is an upper bound for $m_n$, that is, there exists $ M>0, \ s.t.\ m_n\le M,\
%\forall \ n$, which means we perform finite steps of backtracking, the assumption is satisfied.
%Therefore Assumption \ref{assum_diverge} is quite mild.

%Now we assume that $\lfloor U_0 \rfloor \in B(\lfloor U^* \rfloor, \delta_2)\cap \mathcal{L}$
%and $m_n$ satisfies \eqref{back-cond} in our following analysis.
\begin{lemma}\label{lema3}
Let Assumptions \ref{assum_lip} and \ref{assum_hess} hold true. Assume that  $\{U_n\}_{n\in \mathbb{N}_0}$ is a sequence generated by
Algorithm CG-WY, or Algorithm CG-QR, or Algorithm CG-PD.  If
$[U_0] \in B([U^*], \delta_2)\cap \mathcal{L}$  and
\begin{equation*}
\liminf_{n\to\infty} \vertiii \nabla_G E(U_n)\vertiii  > 0,
\end{equation*}
then
\begin{equation*}
\lim_{n\to\infty} \beta_n = 0.
\end{equation*}
\end{lemma}
\begin{proof}
We first show that
\begin{equation}\label{lessthaninf}
\sum_{n, D_n\ne 0}\eta\tau_n(-\text{tr}(\nabla_G E(U_n)^T
  D_n)) < \infty.
\end{equation}
We obtain from \eqref{back-cond} that
\begin{equation*}
E(U_n) - E(U_{n+1}) \geq \eta\tau_n(-\text{tr}(\nabla_G E(U_n)^T D_n )).
\end{equation*}
% Since $\tau_n_1 = -\frac{\text{tr}\langle
%(\nabla_G E(U_n))^T D_n \rangle}{\text{Hess}_GE(U_n)(D_n, D_n)}$, we get
%\begin{equation*}
%E(U_n) - E(U_{n+1}) \geq (1-\alpha t^{m_n})t^{m_n}
%\frac{(\text{tr}\langle(\nabla_G E(U_n))^T
%D_n\rangle)^2}{\text{Hess}_GE(U_n)(D_n, D_n)}
%\end{equation*}
%provided $D_n\ne 0$.
Since $[U^*]$ is the minimizer of \eqref{dis-emin} in $B([U^*], \delta_1)$,
$[U_n] \in B([U^*], \delta_1)$
and the energy is non-increasing during the iteration, we have
%\begin{equation*}
%\infty > E(U_0) - \lim_{n\to\infty} E(U_n) \geq \sum_{n, D_n\ne 0}
%(1-\alpha t^{m_n})t^{m_n}
%\frac{(\text{tr}\langle(\nabla_G E(U_n))^TD_n\rangle)^2}{\text{Hess}_GE(U_n)(D_n,
%D_n)}.
%\end{equation*}
\begin{align*}
&\sum_{n, D_n\ne 0}
\eta\tau_n(-\text{tr}(\nabla_G E(U_n)^T D_n )) \\
\le & \ E(U_0) - \lim_{n\to\infty} E(U_n) \le E(U_0) - E(U^*) < \infty.
\end{align*}
We get from Proposition \ref{diff} that
\begin{equation*}
\sum_{n=1}^{\infty} \vertiii U_{n+1} - U_n\vertiii ^2 \le C_1^2\sum_{n=1}^{\infty}(\tau_n)^2\vertiii D_n\vertiii ^2
\le  C_1^2 \sum_{n,D_n\ne0} \tilde{\tau}_n\vertiii D_n\vertiii ^2\tau_n.
\end{equation*}
Note that Assumption \ref{assum_hess} indicates that for $D_n \ne 0$,
\begin{eqnarray*}
\tilde{\tau}_n \vertiii D_n\vertiii ^2 &\le& -\frac{\text{tr}(\nabla_G E(U_n)^T D_n )}{\text{Hess}_GE
(U_n)[D_n, D_n]} \vertiii D_n\vertiii ^2 \\
 & \le & \frac{1}{\nu_1} (-\text{tr}(\nabla_G E(U_n)^T D_n )),
\end{eqnarray*}
which together with \eqref{lessthaninf} leads to
\begin{equation*}
\begin{split}
&\sum_{n=1}^{\infty} \vertiii U_{n+1} - U_n\vertiii ^2\\
\le&  \frac{C_1^2}{\nu_1}\sum_{n, D_n\ne 0} (-\text{tr}(\nabla_G E(U_n)^T D_n ))\tau_n \\
\leq & \frac{C_1^2}{\nu_1\eta}\sum_{n, D_n\ne 0}
(-\text{tr}(\nabla_G E(U_n)^T D_n ))\eta\tau_n < \infty.
\end{split}
\end{equation*}
Thus we arrive at
\begin{equation}\label{cvg-u}
\lim_{n\to\infty} \vertiii U_{n+1} - U_n\vertiii ^2 = 0.
\end{equation}
We also see from Assumption
\ref{assum_lip} that $\vertiii \nabla_G E(U_n)\vertiii $ are bounded. Hence we get from \eqref{lip} that
\begin{equation*}
\lim_{n\to\infty} \text{tr}\left( (\nabla_G E(U_n) - \nabla_G E(U_{n-1}))^T\nabla_G
              E(U_n) \right) = 0.
\end{equation*}
If $\liminf\limits_{n\to\infty} \vertiii \nabla_G E(U_n)\vertiii  > 0$, then there exits $\delta>0$, such that
\begin{equation*}
\vertiii \nabla_G E(U_n)\vertiii  > \delta, \ \ \forall \ \ n.
\end{equation*}
Consequently, using the definition of $\beta_n$ \eqref{cgprp}, we obtain
\begin{eqnarray*}
&&|\text{tr}\left( (\nabla_G E(U_n) - \nabla_G E(U_{n-1}))^T\nabla_G E(U_n)\right)| \\
&=&\vertiii \nabla_G E(U_{n-1})\vertiii ^2 |\beta_n| \ge |\beta_n|\delta^2
\end{eqnarray*}
and conclude that $\lim\limits_{n\to\infty} \beta_n = 0$.
\end{proof}

Now, we state and prove our main theorem.
\begin{theorem}\label{thm-conv}
Let Assumptions \ref{assum_lip} and \ref{assum_hess} hold true.  For the sequence  $\{U_n\}_{n\in \mathbb{N}_0}$  generated by
Algorithm CG-WY, or Algorithm CG-QR, or Algorithm CG-PD. If $[U_0] \in B([U^*], \delta_2)\cap \mathcal{L}$,   then
\begin{equation}\label{eq-conv}
\lim_{n\to\infty} \textup{dist}([U_n], [U^*]) = 0,
\end{equation}
which means that $[U_n]$ converge to $[U^*]$ on the Grassmann manifold $\mathcal{G}^N_{N_g}$.

Consequently,
\begin{equation}\label{conc-der}
\lim_{n\to\infty} \vertiii \nabla_G E(U_n)\vertiii  = 0.
\end{equation}
\end{theorem}
\begin{proof}
We prove our conclusions by two steps.

(1) First, we prove that under our conditions, there holds
\begin{equation}\label{conv-der}
\liminf_{n\to\infty} \vertiii \nabla_G E(U_n)\vertiii  = 0.
\end{equation}

We prove \eqref{conv-der} by contradiction. Assume that $\vertiii \nabla_G E(U_n)\vertiii \ge \delta, \forall \ n$ for
some $\delta > 0$.
We get from Lemma \ref{lema3} that
\begin{equation*}
\lim_{n\to\infty} \beta_n = 0.
\end{equation*}
Since
\begin{equation*}
\vertiii F_n\vertiii  \le \vertiii \nabla_G E(U_n)\vertiii  + |\beta_n|\vertiii F_{n-1}\vertiii ,
\end{equation*}
we obtain that $\vertiii F_n\vertiii $ are bounded from the fact that $\vertiii \nabla_G E(U_n)\vertiii $ are bounded.
Due to $ {U_n}^T \nabla_G E(U_n)  = 0$, we have
\begin{eqnarray*}
D_n &=& (I-U_n{U_n}^T) F_n \\
        &=& (I-U_n{U_n}^T) (-\nabla_G E(U_n) + \beta_nF_{n-1}) \\
        &=& -\nabla_G E(U_n) + \beta_n(F_{n-1} - U_n
      ({U_n}^TF_{n-1})),
\end{eqnarray*}
thus
\begin{equation*}
\begin{split}
& |\text{tr}(\nabla_G E(U_n)^T D_n )| \\
 = &|\text{tr} (\nabla_G E(U_n)^T (-\nabla_G E(U_n) + \beta_n(F_{n-1} - U_n
      ({U_n}^TF_{n-1}))))| \\
\ge& \vertiii \nabla_G E(U_n)\vertiii ^2 - |\beta_n|\vertiii \nabla_G E(U_n)\vertiii \vertiii F_{n-1} - U_n
      ({U_n}^TF_{n-1})\vertiii .
\end{split}
\end{equation*}
Note that ${U_n}^TU_n = I_N$ implies
\begin{align*}
\vertiii F_{n-1}\vertiii ^2 & = \vertiii F_{n-1} - U_n({U_n}^TF_{n-1}) +
U_n({U_n}^TF_{n-1}) \vertiii ^2 \\
& = \vertiii F_{n-1} - U_n({U_n}^TF_{n-1}) \vertiii ^2 +
\vertiii U_n({U_n}^TF_{n-1}) \vertiii ^2,
\end{align*}
we obtain $\vertiii F_{n-1} - U_n({U_n}^TF_{n-1})\vertiii  \le \vertiii F_{n-1}\vertiii $ and
arrive at
\begin{equation*}
|\text{tr} (\nabla_G E(U_n)^T D_n)| \ge
\vertiii \nabla_G E(U_n)\vertiii ^2 - |\beta_n|\vertiii \nabla_G E(U_n)\vertiii \vertiii F_{n-1}\vertiii .
\end{equation*}
Furthermore, since $\lim\limits_{n\to\infty} \beta_n = 0$, $\vertiii \nabla_G E(U_n)\vertiii $
and $\vertiii F_{n-1}\vertiii $
are bounded, we get
\begin{equation}\label{up}
|\text{tr} (\nabla_G E(U_n)^T D_n)| \ge \vertiii \nabla_G E(U_n)\vertiii ^2/2 \ge \delta^2/2
\end{equation}
provided $n\gg 1$. Since $\vertiii F_n\vertiii $ are bounded, we see that $\vertiii D_n\vertiii $ are
bounded. We conclude from Assumption \ref{assum_hess} that there is a constant $C_3$, such that
\begin{equation}\label{down}
\text{Hess}_GE(U_n)[D_n, D_n] \le \nu_2\vertiii D_n\vertiii ^2 \le C_3.
\end{equation}
Combining \eqref{up} and \eqref{down}, we have that for $D_n\ne 0$ and $n\gg 1$,
\begin{equation*}
\frac{(\text{tr} (\nabla_G E(U_n)^T D_n))^2}
{\text{Hess}_GE(U_n)[D_n, D_n]} \ge \frac{\delta^4}{4C_3}
\end{equation*}
and
\begin{equation*}
\frac{\theta}{\vertiii D_n\vertiii }(-\text{tr} (\nabla_G E(U_n)^T D_n))
\ge \frac{\delta^2\theta}{2}\sqrt{\frac{\nu_2}{C_3}}.
\end{equation*}
From the definition of $\tilde{\tau}_n$ \eqref{stepsize}, for $D_n\ne 0$ and $n\gg 1$,
there holds
\begin{equation*}
\tilde{\tau}_n(-\text{tr} (\nabla_G E(U_n)^T D_n)) \ge
\min\left(\frac{\delta^4}{4C_3}, \frac{\delta^2\theta}{2}
\sqrt{\frac{\nu_2}{C_3}}\right).
\end{equation*}
Thus we get from \eqref{stepsize-lb} that
\begin{equation*}
\tau_n(-\text{tr} (\nabla_G E(U_n)^T D_n)) \ge
\min\left(\frac{\delta^4}{4C_3}, \frac{\delta^2\theta}{2}\sqrt{\frac{\nu_2}{C_3}},
\frac{t(1-\eta)\delta^4\nu_2}{2(C_0C_2+L_0C_1)C_3}\right),
\end{equation*}
which leads to
\begin{equation*}
\sum_{n, D_n\ne 0}\eta\tau_n(-\text{tr} (\nabla_G E(U_n)^T D_n)) = \infty.
\end{equation*}
This contradicts with \eqref{lessthaninf}. Therefore, we arrive at
\begin{equation*}
\liminf_{n\to\infty} \vertiii \nabla_G E(U_n)\vertiii  = 0.
\end{equation*}

(2) Then, we turn to prove our main conclusion  \eqref{eq-conv}.

We obtain from \eqref{conv-der} that there exists a subsequence $\{U_{n_k}\}_{k=1}^\infty$ of
$\{U_n\}_{n=1}^\infty$,
such that
\begin{equation}\label{conv-sub-der}
\lim_{k\to\infty} \vertiii \nabla_G E(U_{n_k})\vertiii  = 0.
\end{equation}
Furthermore, we can prove
\begin{equation}\label{conv-sub}
\lim_{k\to\infty} \text{dist}([U_{n_k}],
[U^*]) = 0.
\end{equation}
Let us prove \eqref{conv-sub} by contradiction. Assume that
$\lim\limits_{k\to\infty} \text{dist}([U_{n_k}],
[U^*]) \ne 0$, then there exists $\tilde{\delta} > 0$ and a subsequence
$\{U_{n_{k_j}}\}_{j=1}^\infty$ of $\{U_{n_k}\}_{k=1}^\infty$, such that
\begin{equation*}
\text{dist}([U_{n_{k_j}}], [U^*]) \ge \tilde{\delta}, \ \forall \ j\ge 0.
\end{equation*}
We obtain from Lemma \ref{lema-dist} that for each $j$, there exists
$P_{n_{k_j}} \in \mathcal{O}^{N\times N}$, such that
\begin{equation*}
\vertiii  U_{n_{k_j}}P_{n_{k_j}} - U^*\vertiii  = \text{dist}([U_{n_{k_j}}],
[U^*]) \ge \tilde{\delta}.
\end{equation*}
Since $\{U_{n_{k_j}}P_{n_{k_j}}\}_{j=1}^\infty$ are bounded and $\mathcal{M}^N_{N_g}$ is compact,
we get that there exists a $\bar{U}$, and a subsequence of $\{U_{n_{k_j}}P_{n_{k_j}}\}_{j=1}^\infty$,
for simplicity of notation, we denote the subsequence also by $\{U_{n_{k_j}}P_{n_{k_j}}\}_{j=1}^\infty$,
such that $\lim\limits_{j\to\infty}\vertiii U_{n_{k_j}}P_{n_{k_j}} - \bar{U}\vertiii  = 0$.
Then we obtain from \eqref{conv-sub-der} that
\begin{equation*}
\begin{split}
  \lim_{j\to\infty} \vertiii \nabla_G E(U_{n_{k_j}}P_{n_{k_j}})\vertiii  & =
      \lim_{j\to\infty} \vertiii \nabla_G E(U_{n_{k_j}})P_{n_{k_j}}\vertiii  \\
    & = \lim_{j\to\infty} \vertiii \nabla_G E(U_{n_{k_j}})\vertiii = 0,
\end{split}
\end{equation*}
which implies $\nabla_G E(\bar{U}) = 0$ since the Lipschitz continuous condition of the gradient
in Assumption \ref{assum_lip}. We get from $[U_{n_{k_j}}P_{n_{k_j}}] \in
B([U^*], \delta_2)$ that $\bar{U} \in B([U^*], \delta_2)$,
and we conclude from Lemma \ref{lema-uniq} that  $[\bar{U}] = [U^*]$,
this contradicts with  $\text{dist}([U_{n_{k_j}}], [U^*])\ge \tilde{\delta}$.
Therefore we have \eqref{conv-sub}, and
\begin{equation*}
\lim_{j\to\infty} E(U_{n_{k_j}}) = \lim_{j\to\infty} E(U_{n_{k_j}}P_{n_{k_j}}) = E(U^*).
\end{equation*}
We observe from the proof of Lemma
\ref{lema3} that the energy is non-increasing during the iteration and bounded below,
then we get
\begin{equation}\label{conv-energy}
\lim_{n\to\infty} E(U_n) = E(U^*).
\end{equation}
Finally, we obtain \eqref{eq-conv} from Lemma \ref{lema-uniq}.

\eqref{conc-der} is just a consequence of \eqref{eq-conv}.
\end{proof}

%\begin{corollary}
%Let Assumptions \ref{assum_lip} and \ref{assum_hess} hold true. If the initial value
%$[U_0] \in B([U^*], \delta_2)$, $E(U_0) \le E_0$,
%and $m_n$ satisfies \eqref{back-cond}, then
%\begin{equation*}
%\lim_{n\to\infty} \vertiii \nabla_G E(U_n) \vertiii = 0.
%\end{equation*}
%\end{corollary}
%\begin{proof}
%This is a direct corollary of Theorem \ref{thm-conv}.
%\end{proof}
\section{A restarted version}\label{sec-modalg}

%Note that our problem \eqref{dis-emin} is not quadratic,  and there are
%several practical issues that should be taken into account. For instance, our conjugate gradient
%algorithms may be further modified by using the restarting approach in implementation that
%may improve the efficiency.
In iteration methods, restarting approach is a commonly used strategy which may improve the performance of the methods. Here, we also apply this strategy to our conjugate gradient algorithms. In fact, the restarting approach has been used in nonlinear conjugate gradient methods to cure the problem of jamming, refer to \cite{HZ, Powell, Smith93} and the references therein for more information.
The restarting approach is to reset the search
direction when it is necessary. For the conjugate gradient type method, the restarting approach means to set the conjugate parameter $\beta_n = 0$.
%In \cite{Smith93}, the restarting approach was used and the author restarted the algorithm at
%periodic intervals. However, it is not always a good choice.
We now propose an new indicator to tell us when we need to restart the calculation.

We define the relative change of the residual $dg_n$ by
\begin{equation}\label{re-change}
dg_n=  \Big| \frac{\vertiii \nabla_G E(U_n)\vertiii - \vertiii \nabla_G E(U_{n-1})\vertiii }{\vertiii \nabla_G E(U_{n-1})\vertiii} \Big|.
\end{equation}
%Then the indicator $\zeta_n$ is defined as
We consider to use the average of $dg_n$ in 3 successive iterations as the indicator $\zeta_n$, that is,
\begin{equation}\label{rest-ind}
\zeta_n=\frac{dg_n+dg_{n-1}+dg_{n-2}}{3}.
\end{equation}
 Let $g_{tol}\in(0,1)$ be a given parameter. When $\zeta_n<g_{tol}$,
  \iffalse{\color{red}{which means the current  search directions can not make the residual decrease sufficiently,}}\fi
  we set $\beta_n=0$ to restart our conjugate gradient algorithms with initial search direction $F_{n}=-\nabla_G E(U_n)$.
  Based on this indicator, we propose our restarted algorithm.%s, and denote them as Restarted CG-WY, Restarted CG-QR and Restarted CG-PD, respectively.

\begin{algorithm}
%\caption{Restarted CG-WY/CG-QR/CG-PD}
\caption{Restarted conjugate gradient method}
 Given $\epsilon, \theta, t, \eta ,g_{tol} \in (0,1)$, %$\epsilon \in (0,1)$, $\theta \in (0,1)$, $t \in (0,1)$, $\eta \in (0,1)$,
  initial data $U_0, \ s.t. \   U_0^TU_0  = I_N$, $F_{-1}= 0$, $dg_{-1}= 0$, $dg_0= 0, \zeta_0=0$,
  calculate the gradient $\nabla_G E(U_0)$, let $n = 0$\;
 \While{$\vertiii \nabla_G E(U_n) \vertiii> \epsilon$}{

  Calculate the conjugate gradient parameter
  \begin{equation*}
  \beta_n=   \begin{cases}

   \frac{\text{tr}( (\nabla_G E(U_n) - \nabla_G E(U_{n-1}))^T\nabla_G
              E(U_n) )}{\vertiii \nabla_G E(U_{n-1})\vertiii^2} &\zeta_n\geq g_{tol},\\

   0 &\zeta_n<g_{tol}.

   \end{cases}
  \end{equation*}
      Let $F_n = -\nabla_G E(U_n)+\beta_nF_{n-1}$\;
  Project the search direction to the tangent space of $U_n$:
      $D_n = F_n - U_n( U_n^TF_n)$\;
  Set $F_n = -F_n\mbox{sign}(\text{tr} (\nabla_G E(U_n)^TD_n
      )$, \\ $D_n = -D_n\text{sign}(\mbox{tr}(\nabla_G
      E(U_n)^TD_n )$\;
  Calculate the step size $\tau_n$ by the {\bf Hessian based strategy$(\theta, t, \eta)$};

 Set $U_{n+1}=\text{ortho}(U_n,D_n, \tau_n)$;

       Let $n=n+1$, calculate the gradient $\nabla_G E(U_n)$, the relative change of the residual $dg_n$ by \eqref{re-change} and
       the  indicator $\zeta_n$ by \eqref{rest-ind}\;
 }
\end{algorithm}

We denote the restarted version of algorithms CG-WY, CG-QR and CG-PD as rCG-WY, rCG-QR and rCG-PD, respectively.

Note that the only difference between the restarted algorithms and the original ones is the choice of the parameter $\beta_n$ while the particular form of $\beta_n$ is only used in Lemma \ref{lema3}. We can easily see that every conclusion we get before Lemma \ref{lema3} also holds true for the restarted algorithms. %We see from Lemma \ref{lema3} and its proof that adding some zero terms in the sequence $\{\beta_i\}_{i=1}^\infty$ does not effect on the proof and its conclusion. Consequently, we can conclude the convergence result of our modified algorithms as the following theorem.
We see from Lemma \ref{lema3} and its proof that the restart procedure does not affect the proof and the conclusion of Lemma \ref{lema3}. More specifically, the condition that $
|\text{tr}( (\nabla_G E(U_n) - \nabla_G E(U_{n-1}))^T\nabla_G E(U_n))|\ge |\beta_n|\delta^2$ in the proof of Lemma \ref{lema3}
 always holds true at each iteration for both $\beta_n=0$ and
\begin{equation*}
\beta_n=\frac{\text{tr}( (\nabla_G E(U_n) - \nabla_G E(U_{n-1}))^T\nabla_G
              E(U_n) )}{\vertiii \nabla_G E(U_{n-1})\vertiii^2}.
\end{equation*}
Consequently, we can conclude the convergence result of our restarted algorithms as the following theorem.
\begin{theorem}\label{thm-convm}
Let Assumptions \ref{assum_lip} and \ref{assum_hess} hold true.  For the sequence  $\{U_n\}_{n\in \mathbb{N}_0}$  generated by
Algorithm rCG-WY, or Algorithm rCG-QR, or Algorithm rCG-PD. If $[U_0] \in B([U^*], \delta_2)\cap \mathcal{L}$,  then
\begin{equation}\label{eq-convm}
\lim_{n\to\infty} \textup{dist}([U_n], [U^*]) = 0,
\end{equation}
which means that $[U_n]$ converge to $[U^*]$ on the Grassmann manifold $\mathcal{G}^N_{N_g}$.

%Consequently, we have
%\begin{equation}\label{conc-derm}
%\lim_{n\to\infty} \vertiii \nabla_G E(U_n)\vertiii  = 0.
%\end{equation}
\end{theorem}
\section{Numerical experiments} \label{sec-num}

Our algorithms are implemented on the software package
Octopus\footnote[1]{Octopus:www.tddft.org/programs/octopus.} (version 4.0.1), and all numerical experiments are carried out on LSSC-III in the State Key Laboratory of
Scientific and Engineering Computing of the Chinese Academy of Sciences. We choose LDA to approximate $v_{xc}(\rho)$ \cite{PeZu} and use the Troullier-Martins
norm conserving pseudopotential \cite{TrMa}.
The initial guess of the orbitals is generated by linear combination of the atomic orbits (LCAO)
method.

Our examples include several typical molecular systems: benzene ($C_6H_6$), aspirin ($C_9H_8O_4)$,
fullerene ($C_{60}$), alanine chain $(C_{33}H_{11}O_{11}N_{11})$, carbon
nano-tube ($C_{120}$), carbon clusters $C_{1015}H_{460}$ and $C_{1419}H_{556}$.
We compare our results with those obtained by the gradient type optimization algorithm proposed recently in
\cite{ZZWZ}, where some numerical results were given to show the advantage of the
new optimization method to the traditional SCF iteration for electronic structure calculations.
We choose OptM-QR algorithm,  the algorithm that performs best
in \cite{ZZWZ}, for comparison in our paper.
% We denote $\vertiii \nabla_G E\vertiii $ by `resi', and `time' denotes the wall
%clock time measured in seconds.
 We use the criterion that if $\vertiii \nabla_G E\vertiii $ is small enough to check the convergence.
 For small systems, the convergence criteria is
$\vertiii \nabla_G E\vertiii <1.0\times 10^{-12}$, while for the two large systems $C_{1015}H_{460}$ and $C_{1419}H_{556}$,
the convergence criteria is set to be $\vertiii \nabla_G E\vertiii < 1.0\times 10^{-11}$ %for
%$C_{1015}H_{460}$ and $\vertiii \nabla_G E\vertiii < 1.01\times 10^{-4}$ for $C_{1419}H_{556}$
(note that $\vertiii \nabla_G E\vertiii $ is
the absolute value, this setting is reasonable). During our numerical tests, we find that it
takes too much time to carry out the projection $D_n = F_n - U_n ({U_n}^T
F_n)$ in Algorithm CG-QR and Algorithm CG-PD, and there is no obvious difference between
the projected and non-projected algorithms. Therefore, this step is omitted in our numerical tests, that is, we set
$D_n = F_n$. The detailed results are shown in Table \ref{t1}.

 We should point out that the backtracking for the step size is not used by our algorithms in our numerical
experiments. Namely, we use $\tau_n = \tilde{\tau}_n$ for the CG algorithms. Therefore, we do not show how we choose $t$ and $\eta$ in Table \ref{t1}.
We should also point out that for  calculating the  step size $\tilde{\tau}_n$, we use the approximate Hessian \eqref{hes-pra} other than the exact Hessian \eqref{hes}\footnote[2]{We conclude from our numerical experiments in Appendix B that the exact Hessian method requires as many iterations as the approximate one while it spends more time than the approximate one, see Table \ref{t-all} in  Appendix B for details. Therefore, taking the cost and
the accuracy into account, we recommend to use the approximate Hessian  \eqref{hes-pra} instead of
the exact Hessian  \eqref{hes}.}.

\begin{center}
\begin{table}[!htbp]
\caption{The numerical results for systems with different sizes obtained by different algorithms, $\theta = 0.8$.}
\label{t1}
\begin{center}
{\small
\begin{tabular}{|c| c c c c|}
\hline
algorithm & energy (a.u.) & iter &  $\vertiii \nabla_G E\vertiii $ & wall clock time (s)\\
\hline
\multicolumn{5}{|c|}{benzene($C_6H_6) \ \ \  N_g = 102705 \ \ \  N = 15 \ \ \  cores = 8$} \\
\hline
%SCF     & -3.74246025E+01 & 21  & 8.58E-07 & 7.78   \\
OptM-QR & -3.74246025E+01 & 2059 & 9.81E-13 & 170.80  \\
CG-WY   & -3.74246025E+01 & 251  & 9.02E-13 & 11.54   \\
CG-QR   & -3.74246025E+01 & 251  & 9.01E-13 & 10.85   \\
CG-PD   & -3.74246025E+01 & 251  & 9.00E-13 & 11.23   \\
\hline
\multicolumn{5}{|c|}{aspirin($C_9H_8O_4) \ \ \  N_g = 133828 \ \ \  N = 34 \ \ \  cores = 16$} \\
\hline
%SCF     & -1.20214764E+02 & 27  & 9.89E-07 & 21.88 \\
OptM-QR & -1.20214764E+02 & 1898 & 8.71E-13 & 357.47 \\
CG-WY   & -1.20214764E+02 & 246  & 9.21E-13 & 30.71 \\
CG-QR   & -1.20214764E+02 & 246  & 9.21E-13 & 29.21 \\
CG-PD   & -1.20214764E+02 & 246  & 9.22E-13 & 28.81 \\
\hline
\multicolumn{5}{|c|}{$C_{60} \ \ \ N_g = 191805  \ \ \  N = 120 \ \ \  cores = 16$} \\
\hline
%SCF     & -3.42875137E+02 & 44    & 9.74E-07 & 203.67 \\
OptM-QR & -3.42875137E+02 & 2017  & 9.60E-13 & 1578.49  \\
CG-WY   & -3.42875137E+02 & 391   & 9.45E-13 & 227.60 \\
CG-QR   & -3.42875137E+02 & 391   & 9.45E-13 & 201.69 \\
CG-PD   & -3.42875137E+02 & 391   & 9.50E-13 & 210.45 \\
\hline
\multicolumn{5}{|c|}{alanine chain$(C_{33}H_{11}O_{11}N_{11})\ \ \  N_g = 293725 \ \ \  N = 132
\ \ \  cores = 32$} \\
\hline
%SCF     & -4.85260317E+02 & 200  & 7.31E-02 & 2940.89 \\
OptM-QR & -4.78562217E+02 & 12276 & 9.93E-13 & 16028.13 \\
CG-WY   & -4.78562217E+02 & 2133  & 9.98E-13 & 1859.13 \\
CG-QR   & -4.78562217E+02 & 2100  & 9.88E-13 & 1658.16 \\
CG-PD   & -4.78562217E+02 & 2124  & 9.89E-13 & 1745.39 \\
\hline
\multicolumn{5}{|c|}{$C_{120} \ \ \ N_g = 354093 \ \ \  N = 240 \ \ \  cores = 32$} \\
\hline
%SCF     & -6.84652247E+02 & 200  & 5.55E-03 & 8723.43 \\
OptM-QR & -6.84467048E+02 & 15000(fail) & 9.70E-10 & 33184.12 \\
CG-WY   & -6.84467048E+02 & 3369        & 9.99E-13 & 5679.66 \\
CG-QR   & -6.84467048E+02 & 3518        & 9.95E-13 & 5016.26 \\
CG-PD   & -6.84467048E+02 & 3359        & 9.95E-13 & 5094.17 \\
%\hline
%\multicolumn{5}{|c|}{2JMO$(C_{178}H_{283}O_{50}N_{57}S) \ \ \  N_g = 1226485 \ \ \  N = 793
% \ \ \  cores = 128$} \\
%\hline
%%SCF     & -8.95732889E+03 & 200   & 8.58E+00 & 74578.56  \\
%OptM-QR & -2.56413551E+03 & 1502  & 4.03E-05 & 11998.30  \\
%CG-QR & -2.56413551E+03 & 1846  & 4.05E-05 & 18201.79  \\
%\hline
%\multicolumn{5}{|c|}{FAS2$(C_{276}H_{442}O_{90}N_{88}S_{10}) \ \ \  N_g = 1903841 \ \ \ N = 1293
% \ \ \  cores = 256$} \\
%\hline
%%SCF      & -1.89384373E+04 &  200  & 1.37E+01 & 149870.70  \\
%OptM-QR  & -4.26018875E+03 & 2050  & 6.46E-05 & 30152.76   \\
%CG-QR  & -4.26018881E+03 & 2094  & 6.48E-05 & 44952.46  \\
\hline
\multicolumn{5}{|c|}{$C_{1015}H_{460} \ \ \ N_g = 1462257 \ \ \  N = 2260 \ \ \  cores = 256$} \\
\hline
%SCF     & -6.84652247E+02 & 200  & 5.55E-03 & 8723.43 \\
OptM-QR & -6.06369982E+03 & 1000(fail) & 5.28E-08 & 69805.53 \\
CG-WY   & -6.06369982E+03 & 266        & 9.15E-12 & 15550.53 \\
CG-QR   & -6.06369982E+03 & 266        & 9.17E-12 & 11180.94 \\
CG-PD   & -6.06369982E+03 & 266        & 9.29E-12 & 19138.82 \\
%\hline
%\multicolumn{5}{|c|}{$C_{1419}H_{556} \ \ \ N_g = 1828847 \ \ \  N = 3116 \ \ \  cores = 256$} \\
%\hline
%%SCF     & -6.84652247E+02 & 200  & 5.55E-03 & 8723.43 \\
%%OptM-QR & -8.43085432E+03 & 1300 (fail) & 1.43E-08 & 170972.94 \\
%OptM-QR & -8.43085432E+03 & 1000 (fail) & 2.45E-08 & 127262.79 \\
%CG-WY   & -8.43085432E+03 & 272         & 9.69E-12 & 22730.36 \\
%CG-QR   & -8.43085432E+03 & 272 & 9.69E-12 & 22730.36 \\
%CG-PD   & -8.43085432E+03 & 272 & 9.98E-12 & 42980.13 \\
\hline
\multicolumn{5}{|c|}{$C_{1419}H_{556} \ \ \ N_g = 1828847 \ \ \  N = 3116 \ \ \  cores = 320$} \\
\hline
%SCF     & -6.84652247E+02 & 200  & 5.55E-03 & 8723.43 \\
%OptM-QR & -8.43085432E+03 & 1300 (fail) & 1.43E-08 & 170972.94 \\
OptM-QR & -8.43085432E+03 & 1000 (fail) & 1.42E-08 & 130324.49 \\
CG-WY   & -8.43085432E+03 & 272         & 9.80E-12 & 29832.56 \\
CG-QR   & -8.43085432E+03 & 272         & 9.71E-12 & 20533.83 \\
CG-PD   & -8.43085432E+03 & 273         & 9.38E-12 & 40391.35 \\
\hline
\end{tabular}}
\end{center}
\end{table}
\end{center}

We see from  Table \ref{t1}
  that the conjugate gradient algorithms proposed in this paper
always need less iterations and less computational time than OptM-QR to reach the same accuracy.
%Among the three conjugate gradient algorithms,
It is also shown by Table \ref{t1} that the numbers of iterations required for the
convergence for the three conjugate gradient algorithms are almost the same. When it comes to the computational
time, CG-QR usually outperforms the other two algorithms especially for large systems.
The reason for the bad performance of Algorithm CG-PD for large systems is that the
eigen-decomposition of the matrix $  \widetilde{U}(\tau)^T\widetilde{U}(\tau) $
is too expensive %costs too much time,
which becomes the major computation in each iteration. In conclusion,
we recommend to use Algorithm CG-QR, especially for large systems.

%We should point out that the backtracking for the step size is not used in our numerical
%experiments. Namely, we use $\tau_n = \tau_n_1$ for the CG algorithms.
%
%We should also point out that for  calculating the time step $\tau_n_1$, we use the approximate Hessian \eqref{hes-pra} other than the exact Hessian \eqref{hes}\footnote{We conclude from our numerical experiments in the Appendix that the exact Hessian method requires as much iterations as the approximate one while it spends more time than the approximate one, see Table 2 in the Appendix for details. Therefore, taking the cost and
%the accuracy into account, we recommend to use the approximate Hessian  \eqref{hes-pra} instead of
%the exact Hessian  \eqref{hes}.}.

%In fact, we have done some tests by using the exact Hessian formula \eqref{hes}, and found that \eqref{hes-pra} is a good approximation
%of the Hessian. More detailed comparison for results obtained by using the exact Hessian \eqref{hes} and those obtained by using the approximate
%  Hessian \eqref{hes-pra} can be found in Table \ref{t-all} in the Appendix. We may conclude from the numerical
%results that the exact Hessian method needs as much iterations as the approximate
%one while it spends more time than the approximate one. Therefore, taking the cost and
%the accuracy into account, we recommend to use the approximate Hessian  \eqref{hes-pra} instead of
%the exact Hessian  \eqref{hes}.

We should also emphasize that the comparison in Table \ref{t1} is between OptM-QR with BB step size and our new CG  algorithms,
while BB step size is almost the most suitable step size choice for the gradient method.
 To show this, we list more results in Table \ref{t-all} in  Appendix B,
including those obtained by OptM-QR with Hessian based step size \eqref{stepsize} and those obtained by CG algorithms with BB step size.  The results show that OptM-QR with
BB step size performs much better than OptM-QR with the Hessian based step size, while for our CG algorithms, BB step size is less efficient than the Hessian based step size.
% If we compare CG with OptM-QR  both  using Hessian-based step size, we can see our new CG  algorithms outperform OptM-QR much, which is consistent with the widely known theory.
%
%To avoid the biases of comparisons, in Table \ref{t-all}, we also compare the numerical results
%for the algorithms using different step sizes: OptM-QR with Hessian based step size \eqref{stepsize},
%and CG algorithms with BB step sizes used in \cite{ZZWZ}. The results show that OptM-QR with
%BB step size performs much better than OptM-QR with the Hessian based step size, while for our
%CG algorithms, BB step size is less efficient than the Hessian based step size. We have compared
%the most efficient step size strategies for both OptM-QR and CG algorithms in
%Table \ref{t1}, and refer to the Appendix for more details.

\begin{figure}
\centering
\caption{Convergence curves for  $\vertiii \nabla_G E\vertiii $  and  the error of energy $E(U_n) - E_{min}$  obtained by OptM-QR,
CG-WY, CG-QR and CG-PD for $C_{120}$.}
%\includegraphics[width=0.45\textwidth]{c120z3wene.eps}
%\includegraphics[width=0.45\textwidth]{c120z3wres.eps}\\
%OptM-QR\\
%\includegraphics[width=0.45\textwidth]{c120wyene.eps}
%\includegraphics[width=0.45\textwidth]{c120wyres.eps} \\
%CG-WY\\
%\includegraphics[width=0.45\textwidth]{c120qrene.eps}
%\includegraphics[width=0.45\textwidth]{c120qrres.eps} \\
%CG-QR\\
%\includegraphics[width=0.45\textwidth]{c120pdene.eps}
%\includegraphics[width=0.45\textwidth]{c120pdres.eps}
%CG-PD
\includegraphics[width=0.8\textwidth]{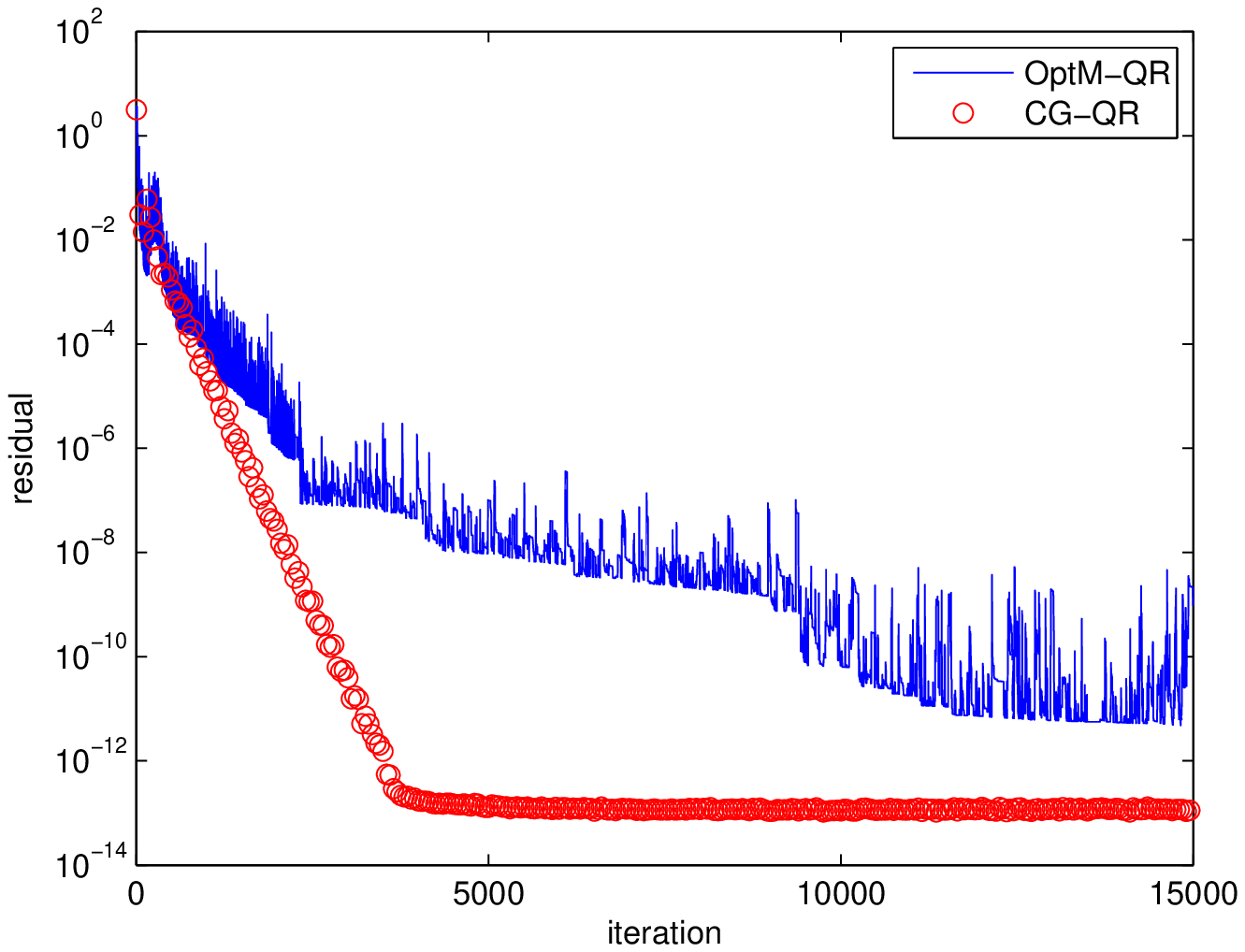} \\
\includegraphics[width=0.8\textwidth]{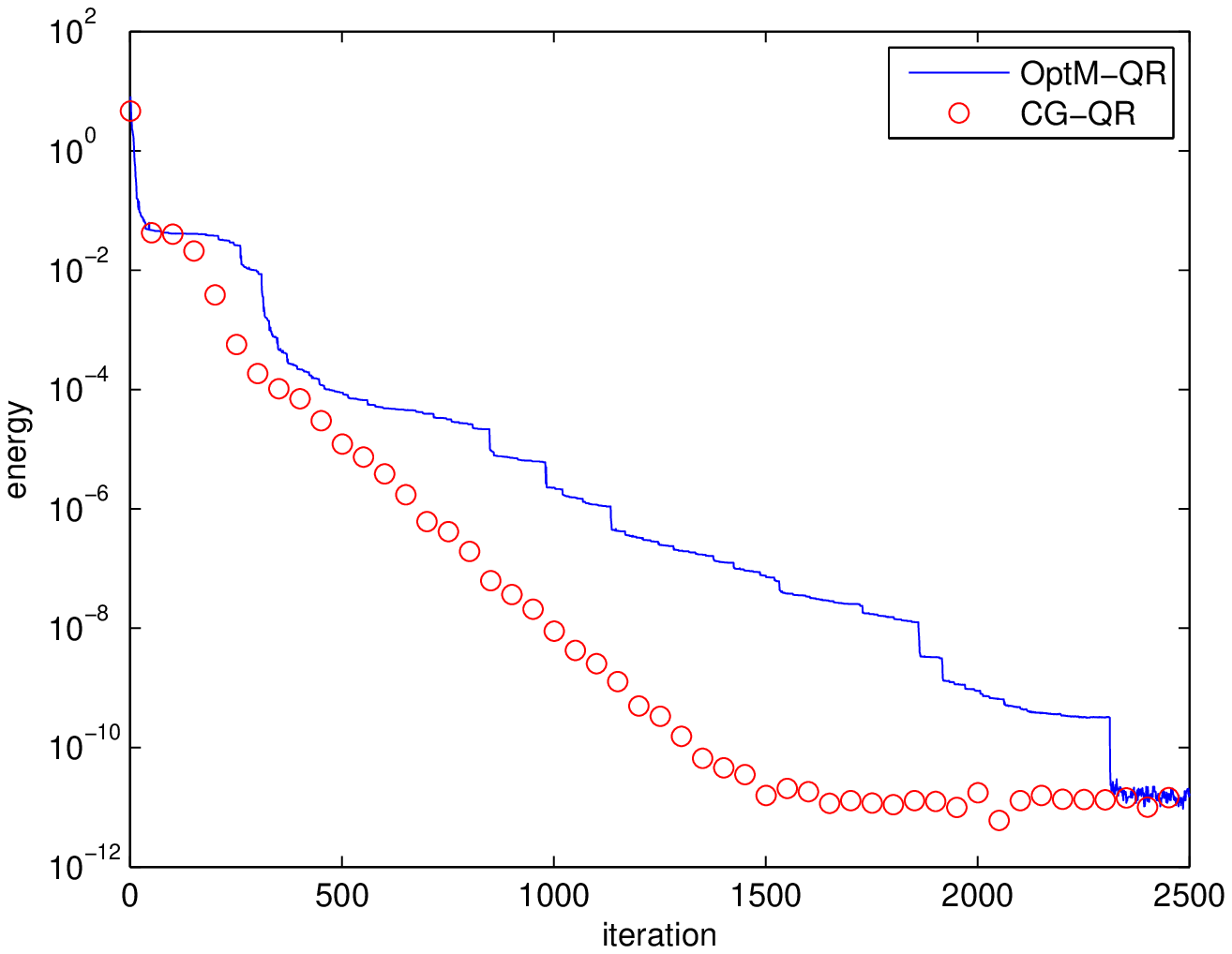}
\label{fig-ene-res}
\end{figure}

We then see the convergence curves for the residual $\vertiii \nabla_G E(U_n)\vertiii$ and the error of the total energy. We take $C_{120}$ as an example. We understand from Table \ref{t1} that the numbers of iterations required for the convergence for the three conjugate gradient methods are almost the same, namely, they share the similar convergence curves. As a result, we only show the results obtained by CG-QR for illustration.
%The convergence curves for the residual and the error of the total energy for  $C_{120}$ obtained by different algorithms
The corresponding results are shown in Figure \ref{fig-ene-res}, where the x-axis is the
number of iterations, the y-axis for the top part is  the residual $\vertiii\nabla_GE(U_n)\vertiii$,
%the error of the energy $E(U_n) - E_{min}$ ($E_{min}$ is a high-accuracy approximation of the exact total energy),
 and the y-axis for the bottom part is the error of the energy $E(U_n) - E_{min}$ ($E_{min}$ is a high-accuracy approximation of the exact total energy).
We can see that the curves for both the residual and the error of the total energy obtained by Algorithm CG-QR %our CG type algorithms
are smoother than those obtained by Algorithm OptM-QR, which indicates that our algorithms are more stable than Algorithm OptM-QR. We understand that the oscillation of the $\vertiii \nabla_G E\vertiii $ curve for Algorithm OptM-QR is caused by
the nonmonotonic behavior of the BB step size \cite{Dai2}.

\iffalse    {\color{red}{In Table \ref{XXX}, we show the numerical results of the Modified algorithms} \fi
% To see it We display energy reduction $E(U_n) - E_{min}$ and $\vertiii \nabla_G E(U_n)\vertiii $ for $C_{120}$
%in Figure \ref{fig-ene-res}, where $E_{min}$ is a reliable minimum of the total energy. We may observe
%that the curves of the CG algorithms are smoother than those obtained by
%Algorithm OptM-QR, which indicates that our algorithms are more stable than Algorithm OptM-QR.
%We understand that the oscillation of the $\vertiii \nabla_G E\vertiii $ curve for Algorithm OptM-QR is due to
%the nonmonotonic behavior of the Barzilai-Borwein step size \cite{Dai2}.

We now turn to illustrate the advantages of our Restarted Algorithms. We choose $g_{tol}=5.0\times10^{-3}$ in our experiments. Other parameters are the same as those we chose in the previous experiments. We only focus on the algorithm  rCG-QR   since the QR strategy is better than others from our previous experiments. We observe that rCG-QR outperforms CG-QR much for alanine chain and $C_{120}$ for which plenty of iterations are required, while for other systems, rCG-QR performs similar with CG-QR. Therefore, we only show the detailed results for alanine chain and $C_{120}$ in Table \ref{t11}. To see the behavior of OptM-QR, CG-QR and rCG-QR more clearly, we also take
$C_{120}$ as an example and plot the convergence curves of the residual and the error of total energy for the three algorithms in Figure
\ref{fig-ene-res1}, from which we see that
rCG-QR converges faster than CG-QR.

\begin{center}
\begin{table}[!htbp]
\caption{The numerical results for alanine chain and $C_{120}$ obtained by different algorithms, $\theta = 0.8$.}
\label{t11}
\begin{center}
{\small
\begin{tabular}{|c| c c c c|}
\hline
algorithm & energy (a.u.) & iter &  $\vertiii \nabla_G E\vertiii $ & wall clock time (s)\\
\hline

\multicolumn{5}{|c|}{alanine chain$(C_{33}H_{11}O_{11}N_{11})\ \ \  N_g = 293725 \ \ \  N = 132
\ \ \  cores = 32$} \\
\hline
%SCF     & -4.85260317E+02 & 200  & 7.31E-02 & 2940.89 \\
OptM-QR & -4.78562217E+02 & 12276 & 9.93E-13 & 16028.13 \\
CG-QR   & -4.78562217E+02 & 2100  & 9.88E-13 & 1658.16 \\
rCG-QR  & -4.78562217E+02 & 1493  & 9.80E-13 & 1192.14 \\
\hline
\multicolumn{5}{|c|}{$C_{120} \ \ \ N_g = 354093 \ \ \  N = 240 \ \ \  cores = 32$} \\
\hline
%SCF     & -6.84652247E+02 & 200  & 5.55E-03 & 8723.43 \\
OptM-QR & -6.84467048E+02 & 15000(fail) & 9.70E-10 & 33184.12 \\
CG-QR   & -6.84467048E+02 & 3518        & 9.95E-13 & 5016.26 \\
rCG-QR  & -6.84467048E+02 & 1846        & 9.68E-13 & 2946.68 \\

\hline
\end{tabular}}
\end{center}
\end{table}
\end{center}

\begin{figure}
\centering
\caption{Convergence curves for $\vertiii \nabla_G E\vertiii $  and  the error of energy $E(U_n) - E_{min}$ obtained by OptM-QR,
 CG-QR and rCG-QR for $C_{120}$.}
%\includegraphics[width=0.45\textwidth]{c120z3wene.eps}
%\includegraphics[width=0.45\textwidth]{c120z3wres.eps}\\
%OptM-QR\\
%\includegraphics[width=0.45\textwidth]{c120wyene.eps}
%\includegraphics[width=0.45\textwidth]{c120wyres.eps} \\
%CG-WY\\
%\includegraphics[width=0.45\textwidth]{c120qrene.eps}
%\includegraphics[width=0.45\textwidth]{c120qrres.eps} \\
%CG-QR\\
%\includegraphics[width=0.45\textwidth]{c120pdene.eps}
%\includegraphics[width=0.45\textwidth]{c120pdres.eps}
%CG-PD
\includegraphics[width=0.8\textwidth]{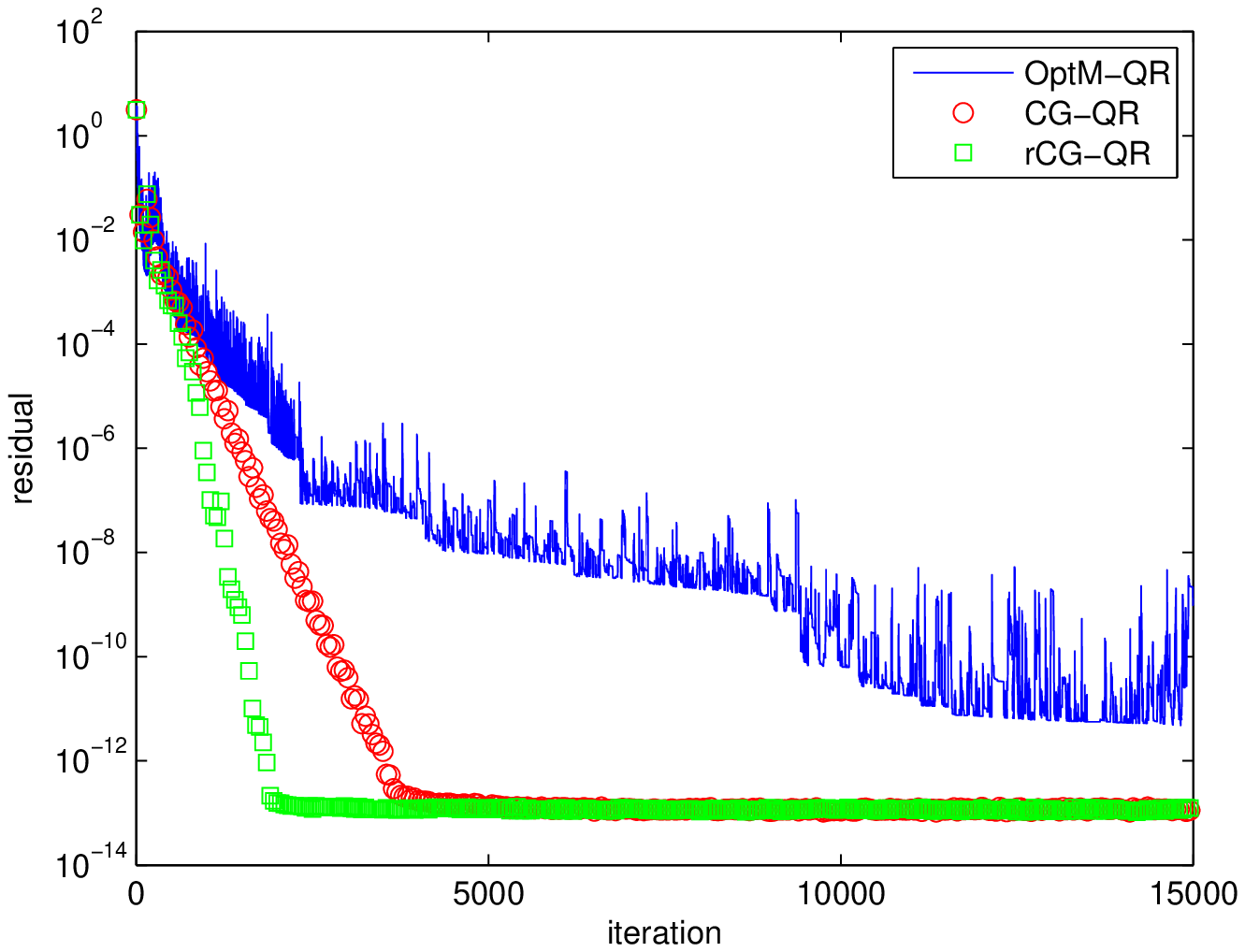} \\
\includegraphics[width=0.8\textwidth]{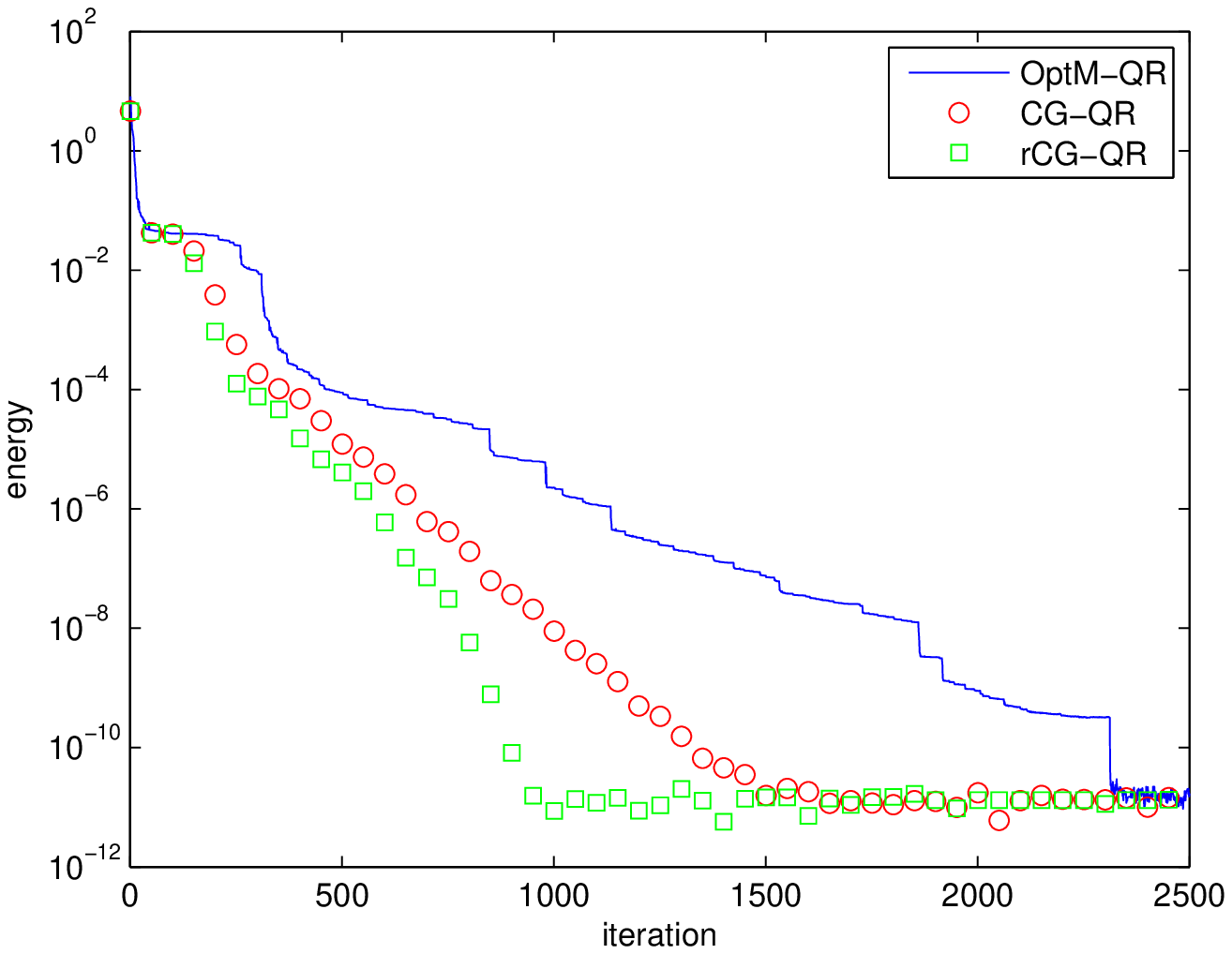}
\label{fig-ene-res1}
\end{figure}

\section{Concluding remarks} \label{sec-cln}
We have proposed a conjugate gradient method for electronic structure calculations in this paper.
Under some reasonable assumptions,
we have proved the local convergence of our algorithms. It is shown by our numerical experiments  that our
algorithms are efficient. We believe these conjugate gradient algorithms can be
further improved by using some preconditioners, which is indeed our on-going project.
We should also point out that the choice in the Hessian based strategy may not be good when
\eqref{pos-hess} fails, there should be a better strategy,
which is also our on-going work.

Our algorithms can be naturally applied to extreme eigenvalue calculations for large scale
symmetric matrices, where the smallest or largest $N$ eigenvalues and corresponding
eigenvectors need to be computed. For matrix eigenvalue problems, Assumption \ref{assum_lip}
is always satisfied, and Assumption \ref{assum_hess} is satisfied when there is a gap between the $N$-th
and $(N+1)$-th eigenvalues \cite{SRNB}, and thus our convergence result also holds.
Our algorithms are simpler to implement than the projected preconditioned conjugate gradient
algorithm proposed in \cite{VYP}, where several practical issues should be taken into account so as to
achieve good performance and some choices of parameters are problems based.

In addition, our method can also be applied to other orthogonality constrained
optimization problems which satisfy our assumptions, for instance, the low rank nearest correlation estimation
\cite{SA}, the quadratic assignment problem \cite{BKF, WY}, etc. For these nonlinear problems where
the calculations of Hessian are not easy, some quasi-Newton method may be used to update
the Hessian.
Anyway, our method should be a quite general approach for optimization problems with
orthogonality constraint.

Finally, we should mention that the algorithm OptM-QR proposed in \cite{ZZWZ} needs less memories at each iteration since it is a gradient type method with BB step sizes. %We understand from \cite{F05} that the BB method may outperform the conjugate gradient method in some cases. %As a result, the algorithm OptM-QR may perform better under some other exchange correlation functionals or for some particular systems, which will derive our attention in the forthcoming investigation.
 Besides, although the conjugate gradient method is superior in terms of asymptotic performance, the BB method gives reasonable results for relatively high tolerances, which makes the BB method useful in some applications. We also understand from \cite{F05} that the BB method may outperform the conjugate gradient method in some other cases.

%{\color{red}{Finally, we should mention that there are also several practical issues that should be taken into account. For instance, we note that our conjugate gradient algorithms may be further modified by using the so-called restart approach in implementation that improve both the accuracy and efficient. We refer to Appendix C for some illustrations.}}

\section*{Acknowledgements} The authors would like to thank  Professor Xin Liu for his comments and suggestions
on the preprint and also for driving our attention to Algorithm 13 in \cite{AbMaSe}.
 The authors  would also like to thank the anonymous referees for
their useful comments and suggestions that improve the presentation of this paper and motivate the authors to present the restarted algorithms.
%especially for the suggestions that inspire us a better way to show our numerical results
%as well as the restarted version of our algorithms.

\Appendix
\renewcommand{\appendixname}{Appendix~\Alph{section}}

\section{Proof of Lemma \ref{lema-uniq}}\label{app-proof}
To prove Lemma \ref{lema-uniq}, we first introduce some notation and a lemma\iffalse which will be used in our analysis \fi. For a diagonal matrix $\mathcal{D} =\text{diag}(d_1, d_2, \cdots, d_N)$, we use $\sin{\mathcal{D}}$  to denote $\text{diag}(\sin{d_1}, \sin{d_2}, \cdots, \sin{d_N})$,  with similar notation for $\cos{\mathcal{D}}$, $\arcsin{\mathcal{D}}$ and $\arccos{\mathcal{D}}$.
\begin{lemma}\label{lema-geod}
For $\Psi = (\psi_1,\cdots,\psi_N)\in \mathcal{M}_{N_g}^N$ and
$\Phi = (\phi_1,\cdots,\phi_N)\in \mathcal{M}_{N_g}^N$, $[\Psi] \ne
[\Phi]$, there exists a curve
$\Gamma(t) \in \mathcal{M}_{N_g}^N, \ t \in [0,1]$, such that $[\Gamma(0)]=[\Psi]$, $[\Gamma(1)] = [\Phi]$,
$\textup{dist}([\Gamma(t)], [\Psi]) \le \textup{dist}
([\Phi], [\Psi])$, and
\begin{equation}\label{curve-der}
  \Gamma(t)^T \Gamma'(t)   = 0.
\end{equation}
\end{lemma}
\begin{proof}
%Our proof here is inspired by the geodesic formula on page 11 of \cite{AbMaSe04}, where the
%invertibility of $  \Psi^T \Phi  $ is required. However, we do not need
%this condition to get $\Gamma(t)$.
Let $  \Psi^T \Phi   = ASB^T$ be the SVD of $  \Psi^T \Phi  $.
We obtain from $\Psi \in \mathcal{M}_{N_g}^N$ and $\Phi \in \mathcal{M}_{N_g}^N$ that
 $S = \text{diag}(s_1, s_2, \cdots, s_N)$ is a diagonal matrix with $s_i \in [0, 1]$,
and hence we denote the diagonal matrix $\arccos S$ by $\Theta$, which means that
$\Theta = \text{diag}(\theta_1, \theta_2, \cdots, \theta_N)$ with $\theta_i = \arccos s_i$.
Let $A_2 S_2B_2^T$ be the SVD of
$\Phi - \Psi ( \Psi^T \Phi )$, where $A_2 \in
\mathcal{M}_{N_g}^N$, $S_2$ is a diagonal matrix containing the singular values, and
$B_2 \in \mathcal{O}^{N\times N}$. Then there holds
\begin{eqnarray*}
  B_2 S_2^2 B_2^T &=&   (\Phi - \Psi ( \Psi^T \Phi ))^T
                  (\Phi - \Psi ( \Psi^T \Phi )) \\
                  &=& I_N - ( \Psi^T \Phi )^T( \Psi^T \Phi )
                   = I_N - BS^2B^T = B(\sin\Theta)^2B^T.
\end{eqnarray*}
As a result, we may choose $S_2 = \sin \Theta$, $B_2 = B$, and obtain $\Phi - \Psi ( \Psi^T \Phi )
= A_2\sin \Theta B^T$. Let
\begin{equation}
\Gamma(t) = \Psi A \cos \Theta t + A_2 \sin \Theta t.
\end{equation}
It is easy to verify that $\Gamma(t) \in \mathcal{M}_{N_g}^N$, $[\Gamma(0)]=[\Psi]$, $[\Gamma(1)] = [\Phi]$,
and $  \Gamma(t)^T \Gamma'(t)   = 0$. Furthermore,
\begin{eqnarray*}
\vertiii \Gamma(t) - \Psi A \vertiii^2 &=& \vertiii \Psi A (\cos \Theta t - I_N) +
A_2 \sin \Theta t \vertiii^2 \\
 & = & \text{tr}((\cos \Theta t - I_N)^2 + (\sin \Theta t)^2) \\
 & = & \text{tr}(2I_N - 2\cos\Theta t).
\end{eqnarray*}
We understand from $s_i \in [0, 1]$ that $\theta_i \in [0, \frac{\pi}{2}]$, which implies that
$\vertiii \Gamma(t) - \Psi A \vertiii^2$ is monotonically non-decreasing
for $t \in [0, 1]$. We then get from the proof of Lemma \ref{lema-dist} that
\begin{equation*}
\vertiii \Gamma(t) - \Psi A \vertiii^2 \le \vertiii \Gamma(1) - \Psi A \vertiii^2
 =\text{tr}(2I_N - 2\cos\Theta) = \text{dist}([\Phi], [\Psi]).
\end{equation*}
Hence we have $\text{dist}([\Gamma(t)], [\Psi]) \le
\text{dist}([\Phi], [\Psi])$. This completes the proof.
\end{proof}

 We should point out that our proof here is inspired by the geodesic formula on page 11 of \cite{AbMaSe04}, where the
invertibility of $  \Psi^T \Phi  $ is required. However, we do not need
this condition to get $\Gamma(t)$.

Now we turn to prove Lemma \ref{lema-uniq}.

\textbf{Proof of Lemma \ref{lema-uniq}:}
\begin{proof}
Let us prove the first conclusion by contradiction. Assume that there exists
$[V^*] \in B([U^*], \delta_1)$, $[V^*] \ne
[U^*]$, such that $\nabla_G E(V^*) = 0$. We see from Assumption
\ref{assum_hess} that $[V^*]$ is also a local minimizer of the energy functional.
We obtain from Lemma \ref{lema-geod} that
there exists a curve $\Gamma(t) \in \mathcal{M}_{N_g}^N, \ t \in[0,1]$, such that $[\Gamma(0)]=[U^*]$, $[\Gamma(1)] = [V^*]$,
$[\Gamma(t)] \in B([U^*], \delta_1)$,
and
\begin{equation*}
  (\Gamma(t))^T\Gamma'(t)  = 0,
\end{equation*}
which implies that $\Gamma'(t)$ is in the tangent space of $\Gamma(t)$ on the Grassmann
manifold. Since $E(\Gamma(t))$ is continuous in $[0, 1]$, and $[U^*]$,
$[V^*]$ are local minimizers, we have that there is a $t_0 \in (0,1)$,
such that $\Gamma(t_0)$ is the maximizer of $E(\Gamma(t)), \ t \in [0, 1]$. This indicates that
\begin{equation*}
\text{tr} (\nabla_G E(\Gamma(t_0))^T\Gamma'(t_0)) = 0.
\end{equation*}
Further, since $\Gamma(t_0)$ is the maximizer of $E(\Gamma(t))$ on this curve,
we see that
\begin{equation*}
 \text{Hess}_GE (\Gamma(t_0))[\Gamma'(t_0), \Gamma'(t_0)] \le 0,
\end{equation*}
which contradicts with the coercivity assumption in Assumptions \ref{assum_hess}. Namely,
there exists only one stationary point $[U^*]$ in $B([U^*],
\delta_1)$ on the Grassmann manifold.

Now we prove the second conclusion by contradiction too. If \eqref{dist-conv} is not true,
then there exists $\hat{\delta} > 0$ and a subsequence $\{V_{n_k}\}_{k=1}^\infty$ of
$\{V_n\}_{n=1}^\infty$, such that
\begin{equation*}
\text{dist}([V_{n_k}], [U^*]) \ge \hat{\delta}, \ \forall \ k\ge 0.
\end{equation*}
We obtain from Lemma \ref{lema-dist} that for each $k$, there exists
$P_{n_k}\in \mathcal{O}^{N\times N}$ satisfying
\begin{equation*}
\vertiii  V_{n_k}P_{n_k} - U^* \vertiii  = \text{dist}([V_{n_k}],
[U^*]) \ge \hat{\delta}.
\end{equation*}
Since $\{V_{n_k}P_{n_k}\}_{k=1}^\infty$ are bounded and $\mathcal{M}^N_{N_g}$ is compact,
we get that there exists a subsequence $\{V_{n_{k_j}}P_{n_{k_j}}\}_{j=1}^\infty$ and $U_0$,
such that $\lim\limits_{j\to\infty}\vertiii V_{n_{k_j}}P_{n_{k_j}} - U_0\vertiii  = 0$. Then, we obtain
from $\lim\limits_{n\to \infty} E(V_n) = E(U^*)$ that
\begin{equation}\label{ene-min}
E(U_0) = \lim_{j\to\infty} E(V_{n_{k_j}}P_{n_{k_j}}) = \lim_{j\to\infty}
E(V_{n_{k_j}}) = E(U^*).
\end{equation}
Due to $[V_{n_{k_j}}P_{n_{k_j}}] \in B([U^*], \delta_1)$, we have that
$[U_0] \in B([U^*], \delta_1)$, which together  with
\eqref{ene-min} yields that $[U_0]$ is also a minimizer of the energy functional in
$B([U^*], \delta_1)$. Therefore, we obtain $\nabla_G E(U_0) = 0$.
We conclude from the proof for the first conclusion  that $[U_0] = [U^*]$,
which contradicts with $\text{dist}([V_{n_{k_j}}], [U^*]) \ge \hat{\delta}$.
This completes the proof.
\end{proof}

\section{Numerical tests of the step sizes}\label{app-numtest}
In this appendix, we will report some more numerical results for the algorithms using different step sizes and different calculation
formulas for
  the Hessian, which lead to our recommendation. The detailed results are listed  in Table
\ref{t-all}.

First, we introduce some notation used in Table \ref{t-all}.%, where ** represents WY, or QR, or PD.
\begin{itemize}
\item OptM-QR-BB: the OptM-QR method with BB step size proposed in \cite{ZZWZ};
\item  OptM-QR-aH: OptM-QR with the Hessian based step size \eqref{stepsize}, where $D_n$ is replaced by $\nabla_GE(U_n)$ and the approximate Hessian
 \eqref{hes-pra} is used;
\item  OptM-QR-H: OptM-QR with the Hessian based step size
\eqref{stepsize}, where $D_n$ is replaced by $\nabla_GE(U_n)$, and the exact Hessian
 \eqref{hes} is used;
%\item   CG-**-BB: the CG-** algorithm with BB step size used in \cite{ZZWZ};
\item   CG-QR-BB: the CG-QR algorithm with BB step size;
%\item CG-**-aH: the CG-** algorithm with Hessian based step size \eqref{stepsize} and approximate
%Hessian  \eqref{hes-pra}  being used;
\item CG-QR-aH: the CG-QR algorithm with the Hessian based step size \eqref{stepsize} and the approximate
Hessian  \eqref{hes-pra}  being used;
%\item  CG-**-H: the CG-** algorithm with Hessian based step size \eqref{stepsize} and exact
%Hessian  \eqref{hes}  being used.
\item  CG-QR-H: the CG-QR algorithm with the Hessian based step size \eqref{stepsize} and the exact
Hessian  \eqref{hes}  being used.
\end{itemize}

We should point out that CG-QR-aH here is just CG-QR in the former part of this paper, and OptM-QR-BB here is just OptM-QR in the former part of this paper.

%We first take a look at the results obtained by our three CG algorithms using the approximate Hessian \eqref{hes-pra}  and
We first take a look at the results obtained by our CG algorithms using the approximate Hessian \eqref{hes-pra}  and
the exact Hessian \eqref{hes}, respectively. By the comparison, we can see that the algorithms using the exact Hessian \eqref{hes}
require as much iterations as the algorithms using the approximate Hessian \eqref{hes-pra}, while the former spends of course more time than the latter one.
To understand this numerical phenomenon, we show the changes of the approximate Hessian
\eqref{hes-pra}, the Hartree term (second line of \eqref{hes}) and the exchange and
correlation term (third line of \eqref{hes}) as  a function of the number of iteration in Figure \ref{fig-hes}.
We can see that in most cases, the last two terms are much smaller than the approximate
Hessian term, which means that they contribute little to the exact Hessian  and can be neglected.
Therefore, taking the cost and the accuracy into account, we recommend to use the
approximate  Hessian \eqref{hes-pra} instead of the exact Hessian \eqref{hes}.

We then compare the efficiency of algorithms with different step sizes. We observe from Table \ref{t-all}
that when using the Hessian-based step size, CG algorithms outperform OptM-QR much. Furthermore, based on our numerical experiments,
we should point out that the BB step size is almost the most suitable one for the gradient
method while it may not be so good for the CG methods.
%That is also why we can not see much advantage of our CG algorithms
% to the OptM-QR method.

%
%Lastly, we can see from the numerical experiments that for small systems, Algorithm CG-QR-aH
%costs roughly the same amount of time compared with OptM-QR, CG-WY-aH, and CG-PD-aH,
%while for the large systems, CG-QR-aH outperform the other three algorithms. Therefore,
%we recommend to use Algorithm CG-QR, especially for large scale systems.

\begin{center}
\begin{table}%[!htbp]
\caption{Numerical results for systems with different sizes obtained  by   different algorithms.}\label{t-all}
\begin{center}
{\small
\begin{tabular}{|c| c c c c|}
\hline
algorithm & energy (a.u.) & iter & $\vertiii \nabla_G E \vertiii $ & wall clock time  (s)\\
\hline
\multicolumn{5}{|c|}{benzene($C_6H_6) \ \ \  N_g = 102705 \ \ \  N = 15 \ \ \  cores = 8$} \\
\hline
%SCF            & -3.74246025E+01 & 21  & 8.58E-07 & 7.78   \\
OptM-QR-BB      & -3.74246025E+01 & 164  & 7.49E-07 & 6.63    \\
OptM-QR-aH       & -3.74246025E+01 & 1753 & 9.92E-07 & 67.83   \\
OptM-QR-H   & -3.74246025E+01 & 1771 & 9.92E-07 & 86.33   \\
\hline
%CG-WY-BB        & -3.74246025E+01 & 291  & 5.15E-07 & 14.65   \\
%CG-WY-aH        & -3.74246025E+01 & 118  & 9.94E-07 & 5.84   \\
%CG-WY-H         & -3.74246025E+01 & 119  & 9.10E-07 & 7.50    \\
%\hline
CG-QR-BB        & -3.74246025E+01 & 249  & 9.98E-07 & 10.41   \\
CG-QR-aH        & -3.74246025E+01 & 118  & 9.94E-07 & 5.56    \\
CG-QR-H         & -3.74246025E+01 & 119  & 9.10E-07 & 7.26    \\
\hline
%CG-PD-BB        & -3.74246025E+01 & 210  & 9.10E-07 & 9.96   \\
%CG-PD-aH        & -3.74246025E+01 & 118  & 9.94E-07 & 5.58   \\
%CG-PD-H         & -3.74246025E+01 & 119  & 9.10E-07 & 7.25    \\
%\hline
\multicolumn{5}{|c|}{aspirin($C_9H_8O_4) \ \ \  N_g = 133828 \ \ \  N = 34 \ \ \  cores = 16$} \\
\hline
%SCF            & -1.20214764E+02 & 27  & 9.89E-07 & 21.88 \\
OptM-QR-BB      & -1.20214764E+02 & 153  & 9.89E-07 & 14.63  \\
OptM-QR-aH      & -1.20214764E+02 & 1271 & 9.92E-07 & 137.97 \\
OptM-QR-H       & -1.20214764E+02 & 1275 & 9.89E-07 & 165.61 \\
\hline
%CG-WY-BB        & -1.20214764E+02 & 197  & 8.62E-07 & 22.75  \\
%CG-WY-aH        & -1.20214764E+02 & 118  & 9.68E-07 & 15.88 \\
%CG-WY-H         & -1.20214764E+02 & 118  & 9.51E-07 & 18.16  \\
%\hline
CG-QR-BB        & -1.20214764E+02 & 220  & 7.94E-07 & 25.28  \\
CG-QR-aH        & -1.20214764E+02 & 118  & 9.68E-07 & 15.15  \\
CG-QR-H         & -1.20214764E+02 & 118  & 9.50E-07 & 17.94  \\
\hline
%CG-PD-BB        & -1.20214764E+02 & 204  & 7.82E-07 & 22.20  \\
%CG-PD-aH        & -1.20214764E+02 & 118  & 9.68E-07 & 14.44 \\
%CG-PD-H         & -1.20214764E+02 & 118  & 9.50E-07 & 18.14  \\
%\hline
\multicolumn{5}{|c|}{$C_{60} \ \ \ N_g = 191805  \ \ \  N = 120 \ \ \  cores = 16$} \\
\hline
%SCF            & -3.42875137E+02 & 44   & 9.74E-07 & 203.67 \\
OptM-QR-BB      & -3.42875137E+02 & 234  & 7.17E-07 & 92.17   \\
OptM-QR-aH      & -3.42875137E+02 & 3155 & 9.90E-07 & 1617.35 \\
OptM-QR-H       & -3.42875137E+02 & 3096 & 9.97E-07 & 1642.80 \\
\hline
%CG-WY-BB        & -3.42875137E+02 & 384  & 9.68E-07 & 213.73  \\
%CG-WY-aH        & -3.42875137E+02 & 184  & 9.78E-07 & 108.89 \\
%CG-WY-H         & -3.42875137E+02 & 184  & 9.74E-07 & 118.49   \\
%\hline
CG-QR-BB        & -3.42875137E+02 & 332  & 9.20E-07 & 152.85  \\
CG-QR-aH        & -3.42875137E+02 & 184  & 9.78E-07 & 98.64   \\
CG-QR-H         & -3.42875137E+02 & 184  & 9.74E-07 & 104.29   \\
\hline
%CG-PD-BB        & -3.42875137E+02 & 417  & 8.98E-07 & 197.42  \\
%CG-PD-aH        & -3.42875137E+02 & 184  & 9.79E-07 & 99.55 \\
%CG-PD-H         & -3.42875137E+02 & 184  & 9.74E-07 & 107.44   \\
%\hline
%\end{tabular}}
%\end{center}
%\end{table}
%\end{center}

%\begin{center}
%\begin{table}
%\caption{The results using $\vertiii \nabla_G E\vertiii $ as convergence criteria, $\theta = 0.8$.}
%\begin{center}
%{\small
%\begin{tabular}{|c| c c c c|}
%\hline
%algorithm & energy (a.u.) & iter & $\vertiii \nabla_G E \vertiii $ & wall clock time  (s)\\
%\hline
\multicolumn{5}{|c|}{alanine chain$(C_{33}H_{11}O_{11}N_{11})\ \ \  N_g = 293725 \ \ \  N = 132
\ \ \  cores = 32$} \\
\hline
%SCF            & -4.85260317E+02 & 200  & 7.31E-02 & 2940.89 \\
OptM-QR-BB      & -4.78562217E+02 & 1558  & 9.49E-07 & 906.95 \\
OptM-QR-aH      & -4.78562217E+02 & 30000 & 4.42E-06 & 23404.10 \\
OptM-QR-H       & -4.78562216E+02 & 30000 & 7.54E-05 & 29684.30 \\
\hline
%CG-WY-BB        & -4.78562217E+02 & 2399  & 8.91E-07 & 1947.47 \\
%CG-WY-aH        & -4.78562217E+02 & 1034  & 9.96E-07 & 897.20 \\
%CG-WY-H         & -4.78562217E+02 & 1168  & 9.94E-07 & 1346.78 \\
%\hline
CG-QR-BB        & -4.78562217E+02 & 3190  & 9.59E-07 & 2338.59 \\
CG-QR-aH        & -4.78562217E+02 & 1017  & 9.93E-07 & 818.32 \\
CG-QR-H         & -4.78562217E+02 & 1232  & 1.00E-06 & 1204.82 \\
%\hline
%CG-PD-BB        & -4.78562217E+02 & 2735  & 9.99E-07 & 1950.20 \\
%CG-PD-aH        & -4.78562217E+02 & 1030  & 9.95E-07 & 838.13 \\
%CG-PD-H         & -4.78562217E+02 & 1177  & 1.00E-06 & 1243.25 \\
\hline
\multicolumn{5}{|c|}{$C_{120} \ \ \ N_g = 354093 \ \ \  N = 240 \ \ \  cores = 32$} \\
\hline
%SCF          & -6.84652247E+02 & 200  & 5.55E-03 & 8723.43 \\
OptM-QR-BB    & -6.84467048E+02 & 2024  & 9.53E-07 & 2119.55 \\
OptM-QR-aH    & -6.84467047E+02 & 30000 & 5.67E-05 & 41941.88 \\
OptM-QR-H     & -6.84467046E+02 & 30000 & 5.74E-05 & 44111.19 \\
\hline
%CG-WY-BB      & -6.84467048E+02 & 3907  & 9.08E-07 & 6418.26 \\
%CG-WY-aH      & -6.84467048E+02 & 1373  & 9.89E-07 & 2316.58 \\
%CG-WY-H       & -6.84467048E+02 & 1322  & 9.93E-07 & 2431.32 \\
%\hline
CG-QR-BB      & -6.84467048E+02 & 3573  & 9.88E-07 & 4426.20 \\
CG-QR-aH      & -6.84467048E+02 & 1490  & 9.93E-07 & 2164.48 \\
CG-QR-H       & -6.84467048E+02 & 1324  & 9.89E-07 & 2019.65 \\
\hline
%CG-PD-BB      & -6.84467048E+02 & 3043  & 8.12E-07 & 4211.23 \\
%CG-PD-aH      & -6.84467048E+02 & 1371  & 9.91E-07 & 2044.93 \\
%CG-PD-H       & -6.84467048E+02 & 1324  & 9.96E-07 & 2146.43 \\
%\hline
\multicolumn{5}{|c|}{$C_{1015}H_{460} \ \ \ N_g = 1462257 \ \ \  N = 2260 \ \ \  cores = 256$} \\
\hline
%SCF          & -6.84652247E+02 & 200  & 5.55E-03 & 8723.43 \\
OptM-QR-BB    &  -6.06369982E+03 & 137  & 3.38E-05 & 4813.73 \\
OptM-QR-aH    &  -6.06369982E+03 & 1078 & 7.57E-05 & 43975.26 \\
OptM-QR-H     &  -6.06369982E+03 & 1102 & 7.58E-05 & 45701.15 \\
\hline
%CG-WY-BB      &  -6.06369982E+03 & 177  & 6.18E-05 & 10793.74 \\
%CG-WY-aH      &  -6.06369982E+03 & 104  & 7.04E-05 & 5956.20 \\
%CG-WY-H       &  -6.06369982E+03 & 104  & 7.05E-05 & 6243.23 \\
%\hline
CG-QR-BB      &  -6.06369982E+03 & 163  & 6.36E-05 & 7511.18 \\
CG-QR-aH      &  -6.06369982E+03 & 104  & 7.03E-05 & 4640.49 \\
CG-QR-H       &  -6.06369982E+03 & 104  & 7.05E-05 & 4724.56 \\
\hline
%CG-PD-BB      &  -6.06369982E+03 & 214  & 3.63E-05 & 17006.92 \\
%CG-PD-aH      &  -6.06369982E+03 & 104  & 7.04E-05 & 7894.28 \\
%CG-PD-H       &  -6.06369982E+03 & 104  & 7.05E-05 & 8211.73 \\
%\hline
\multicolumn{5}{|c|}{$C_{1419}H_{556} \ \ \ N_g = 1828847 \ \ \  N = 3116 \ \ \  cores = 512$} \\
\hline
%SCF          & -6.84652247E+02 & 200  & 5.55E-03 & 8723.43 \\
OptM-QR-BB    & -8.43085432E+03 & 151  & 1.01E-04 & 7769.35 \\
OptM-QR-aH    & -8.43085432E+03 & 988  & 1.01E-04 & 56845.33 \\
OptM-QR-H     & -8.43085432E+03 & 1054 & 9.99E-05 & 61326.61 \\
\hline
%CG-WY-BB      & -8.43085432E+03 & 183 & 9.93E-05 & 16595.16\\
%CG-WY-aH      & -8.43085432E+03 & 103 & 1.00E-04 & 8813.65 \\
%CG-WY-H       & -8.43085432E+03 & 103 & 1.00E-04 & 9193.51 \\
%\hline
CG-QR-BB      & -8.43085432E+03 & 200 & 9.16E-05 & 13495.65 \\
CG-QR-aH      & -8.43085432E+03 & 103 & 1.00E-04 & 5568.38 \\
CG-QR-H       & -8.43085432E+03 & 103 & 1.00E-04 & 6578.48 \\
%\hline
%CG-PD-BB      & -8.43085432E+03 & 168 & 7.91E-05 & 27569.15 \\
%CG-PD-aH      & -8.43085432E+03 & 103 & 1.00E-04 & 14936.50 \\
%CG-PD-H       & -8.43085432E+03 & 103 & 1.00E-04 & 15447.2 \\
\hline
\end{tabular}}
\end{center}
\end{table}
\end{center}

%\newpage
\begin{figure}
\centering
\caption{Left: changes of each term in Hessian $\textup{Hess}_GE(U_n)[D_n, D_n]$
with iteration. Right: the quotient of the sum of the Hartree and exchange correlation
terms divided by the approximate Hessian. (for algorithm CG-QR-H).} \label{fig-hes}
\includegraphics[width=0.45\textwidth]{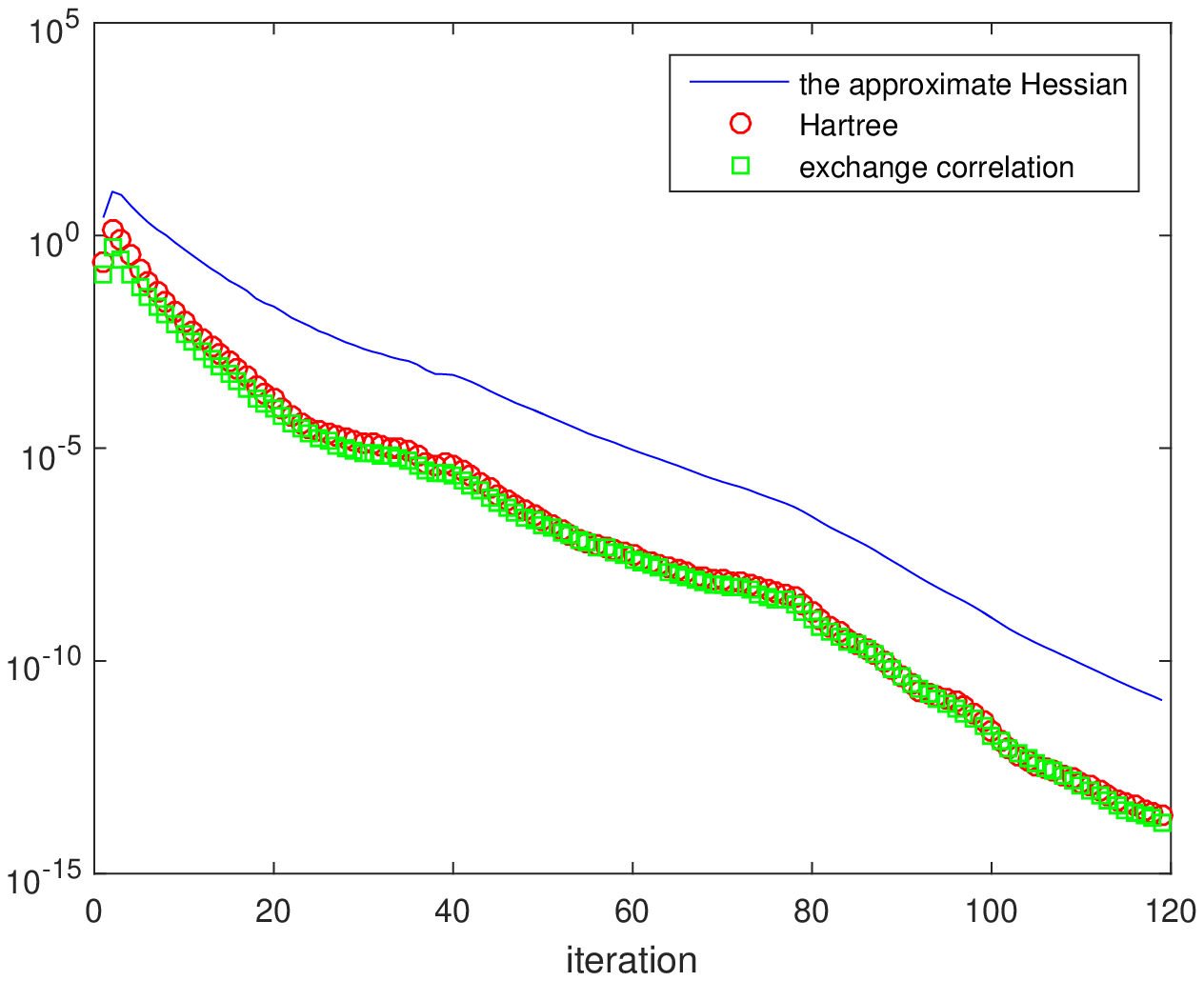}
\includegraphics[width=0.45\textwidth]{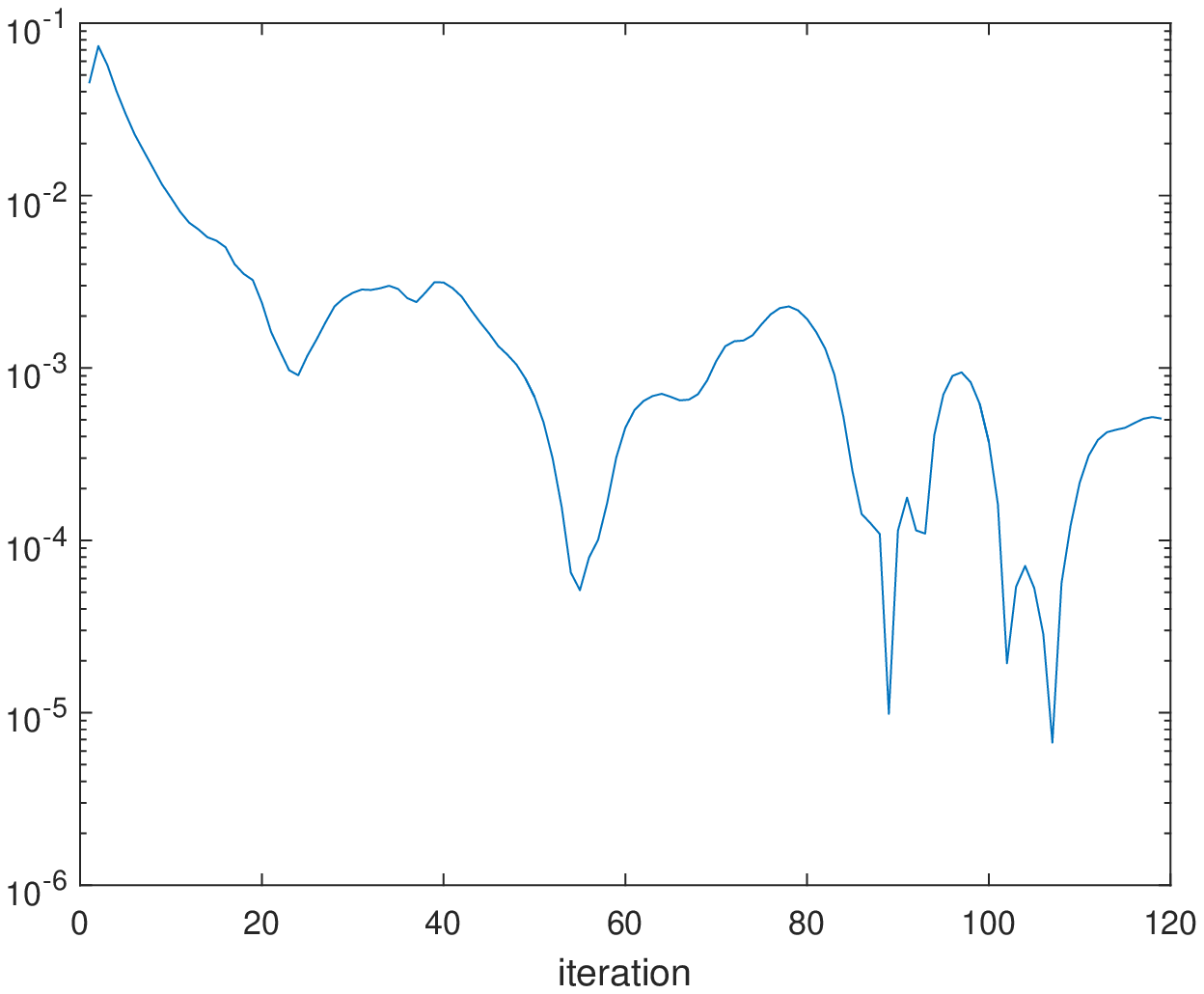}\\
benzene\\
\includegraphics[width=0.45\textwidth]{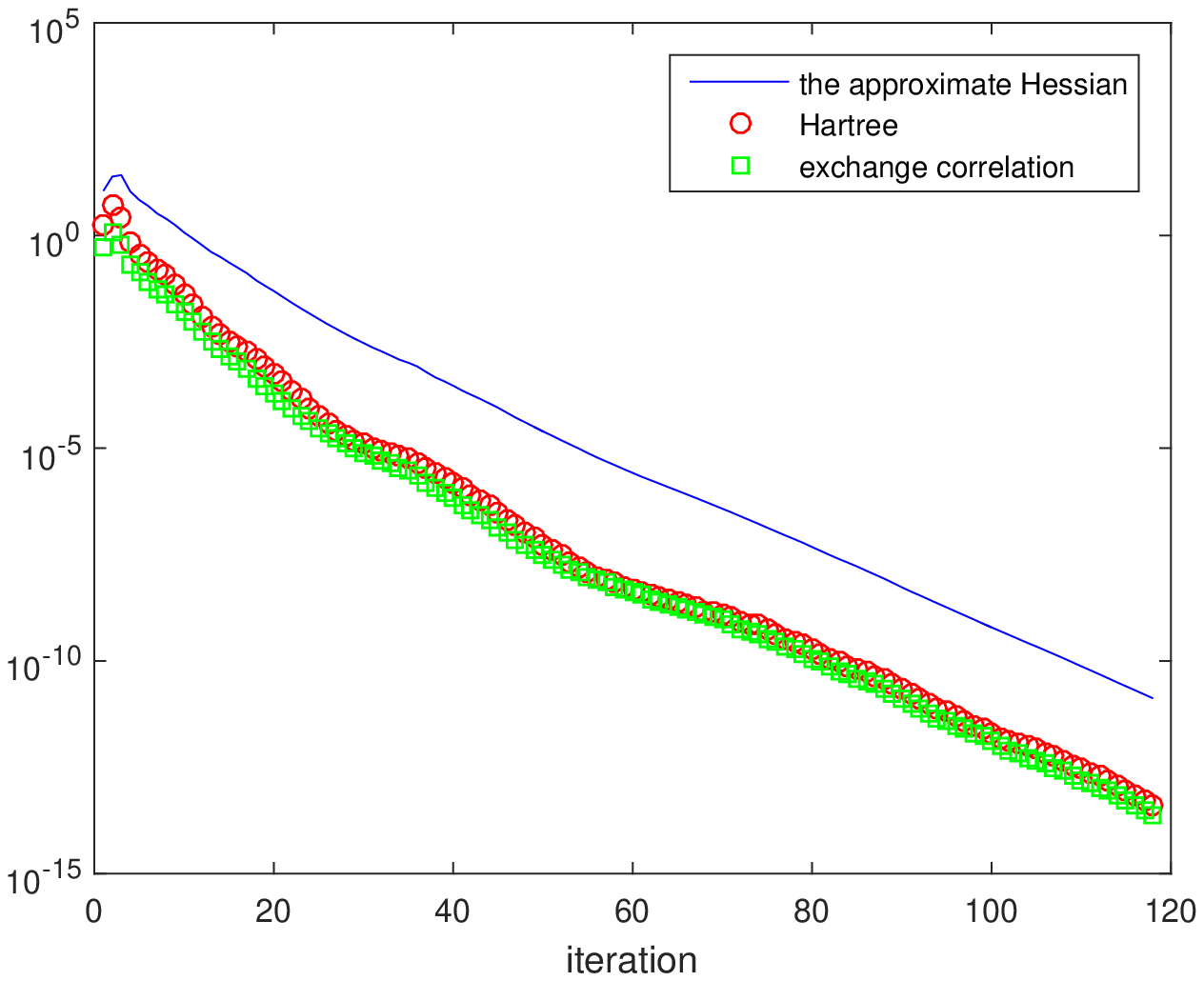}
\includegraphics[width=0.45\textwidth]{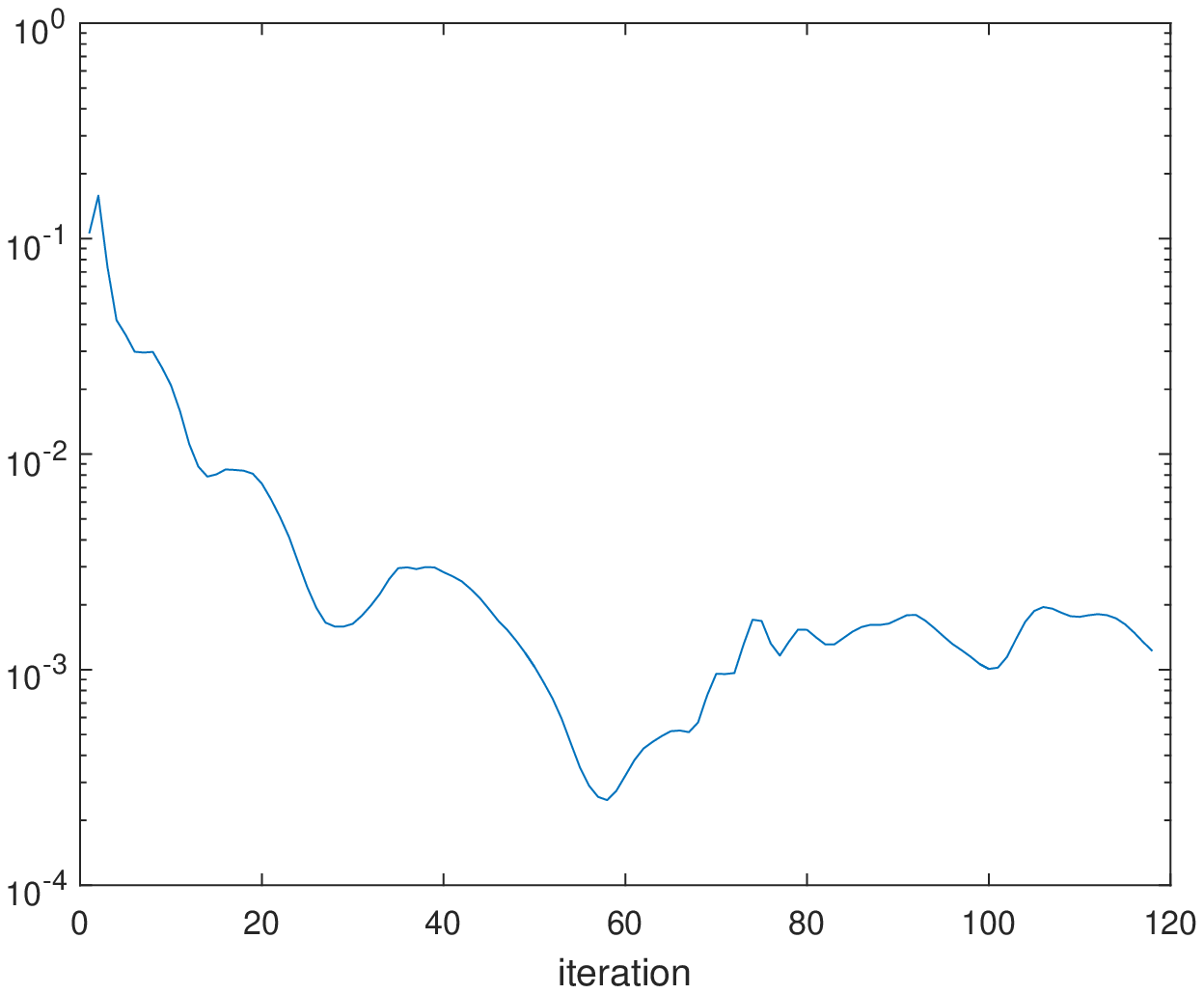} \\
aspirin\\
\includegraphics[width=0.45\textwidth]{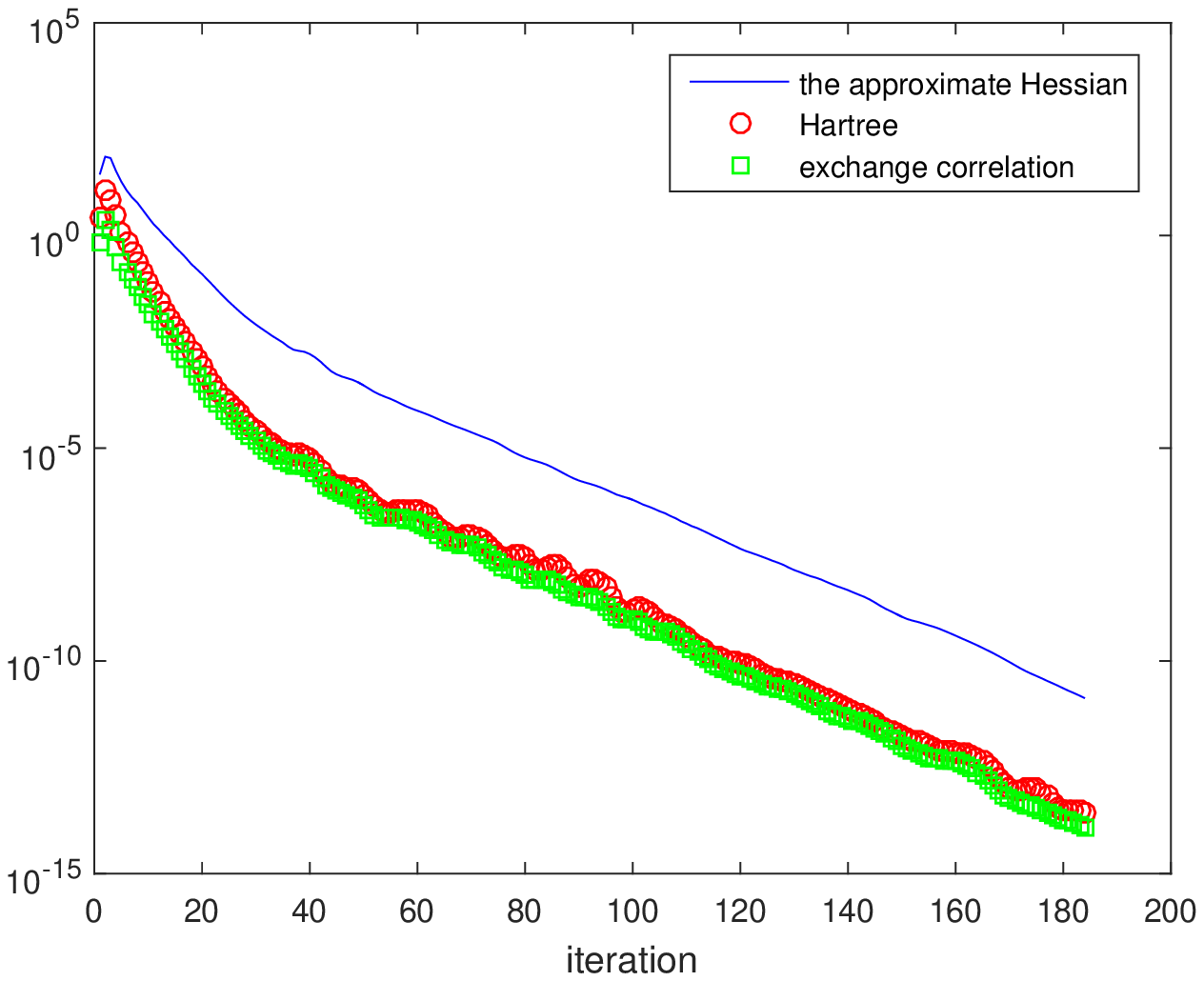}
\includegraphics[width=0.45\textwidth]{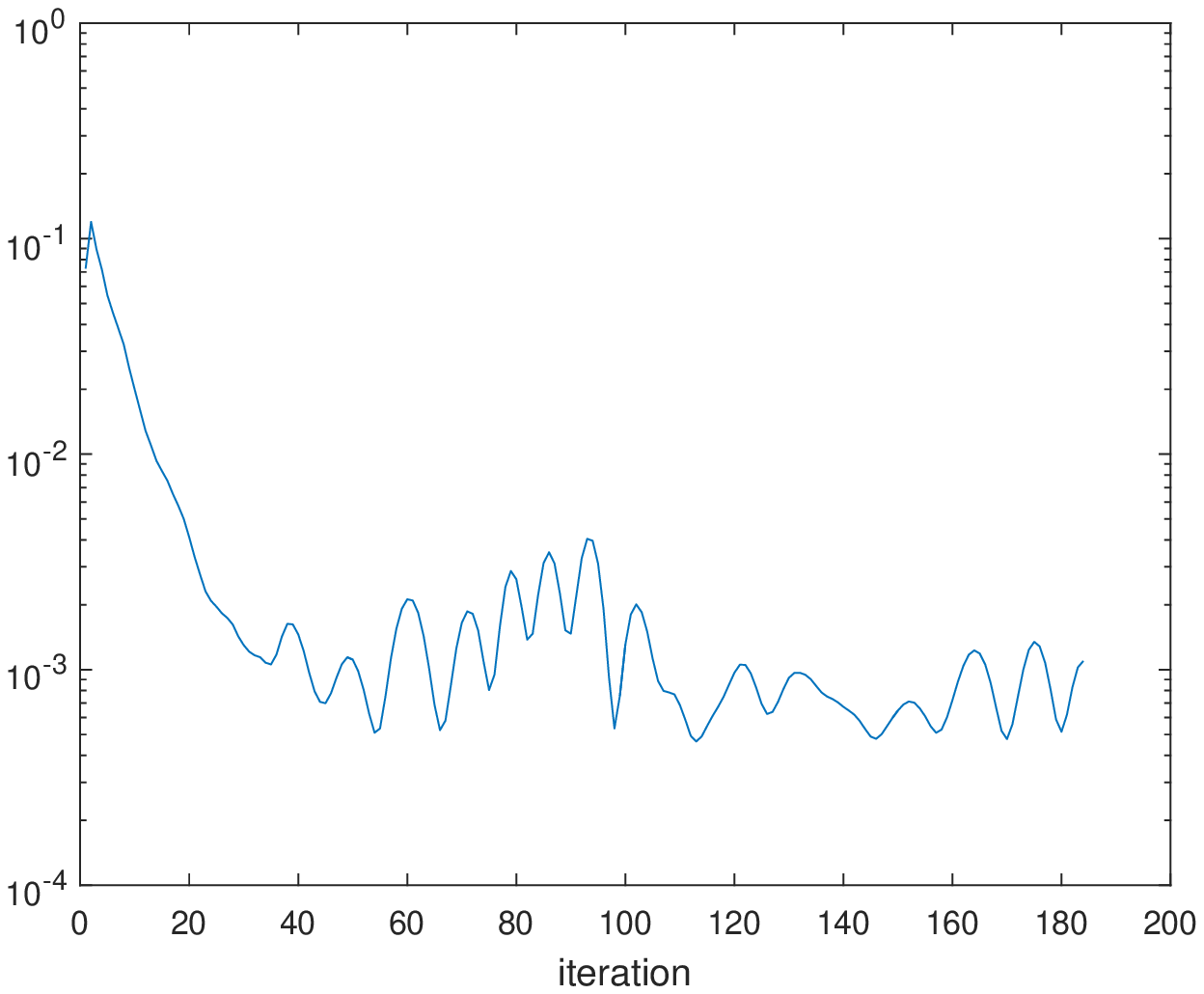} \\
$C_{60}$
\end{figure}
\begin{figure}
\centering
\includegraphics[width=0.45\textwidth]{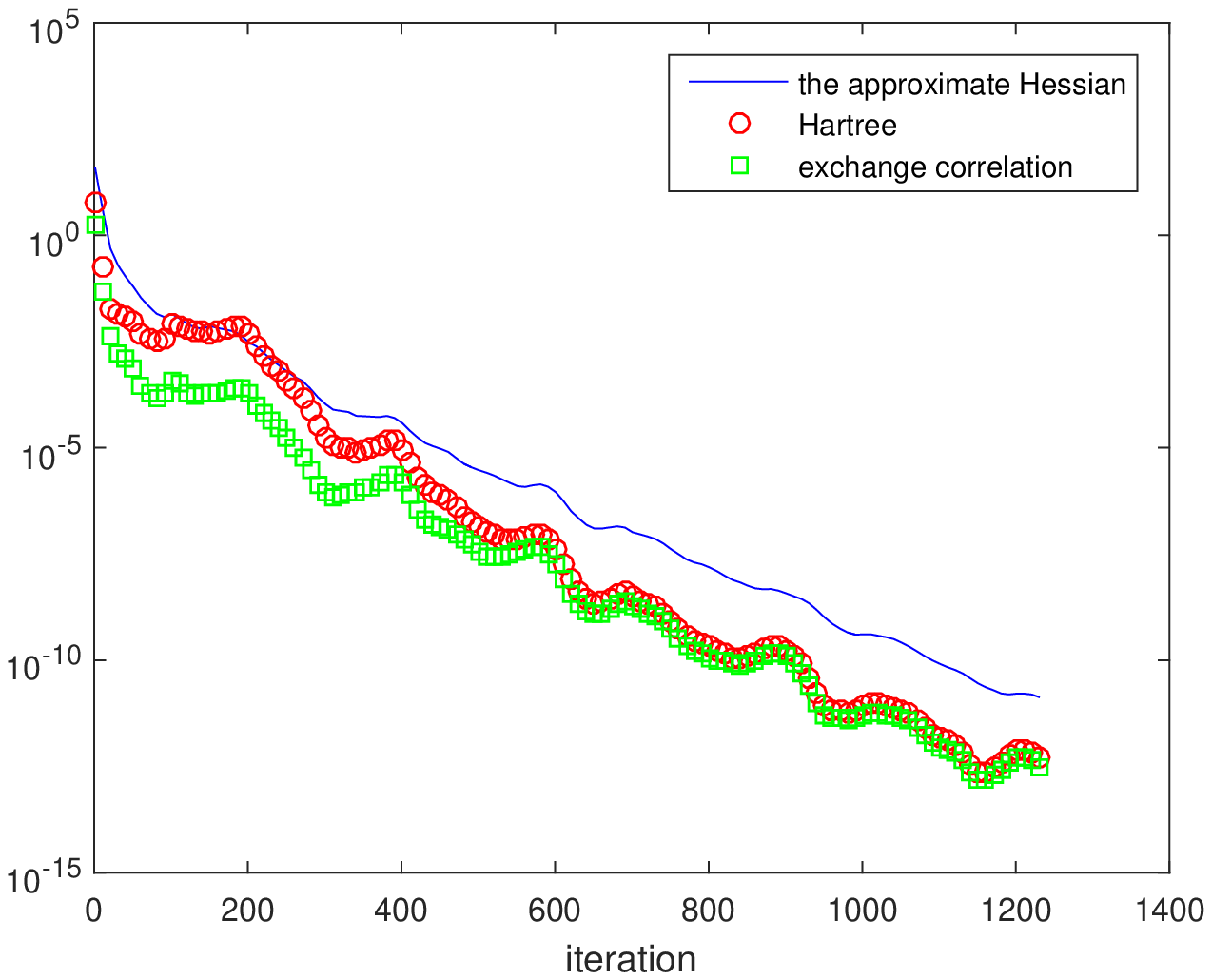}
\includegraphics[width=0.45\textwidth]{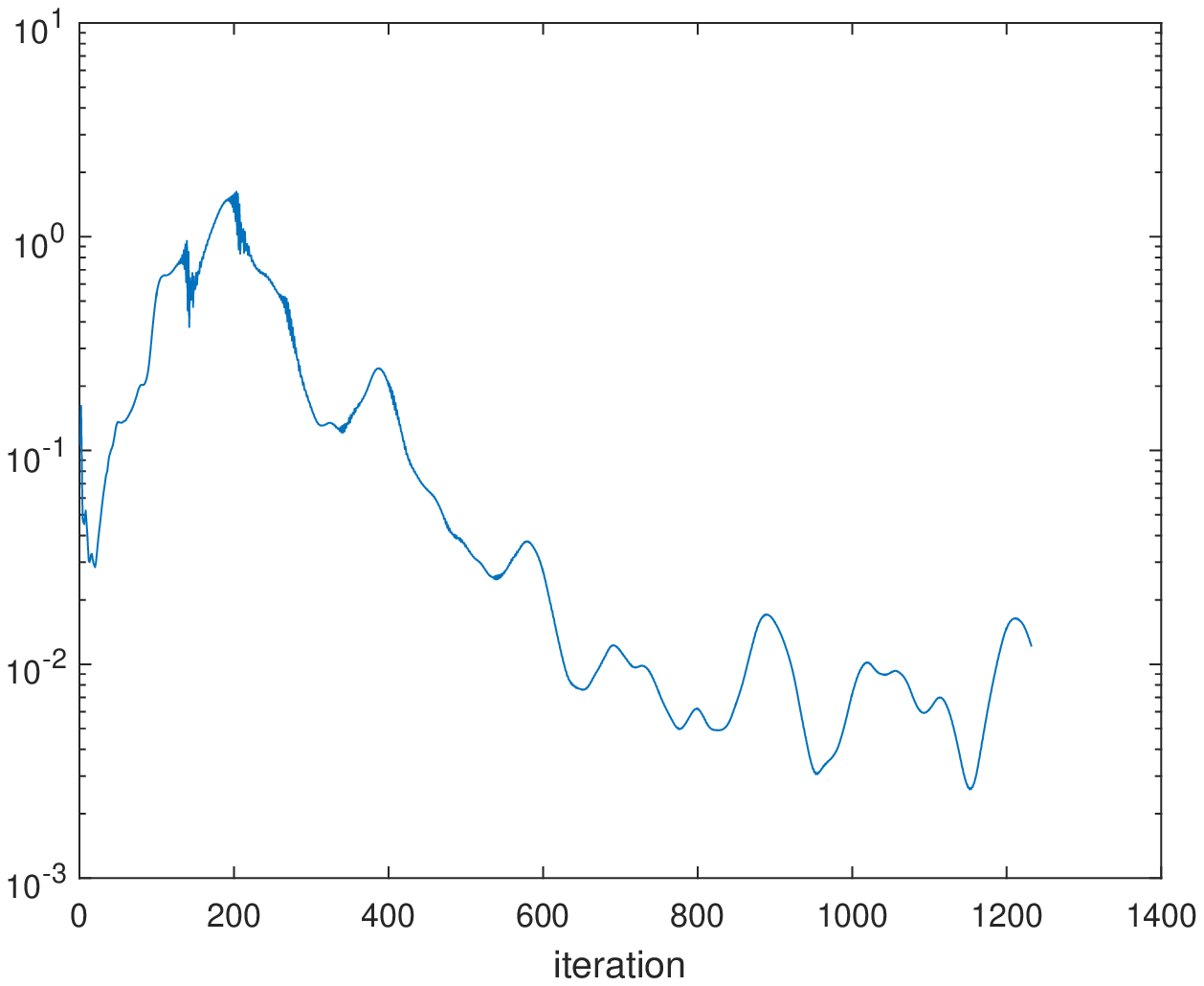} \\
alanine chain \\
\includegraphics[width=0.45\textwidth]{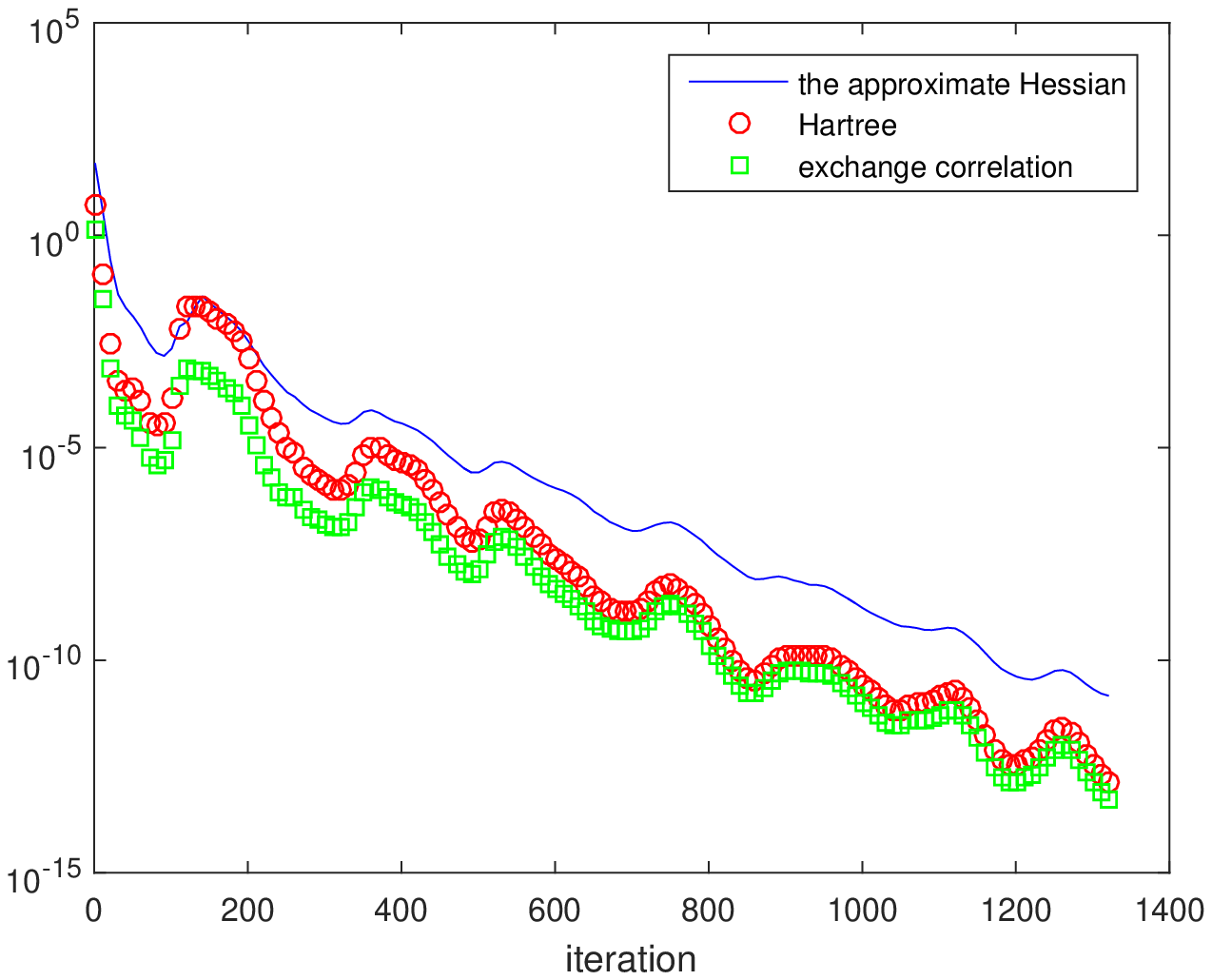}
\includegraphics[width=0.45\textwidth]{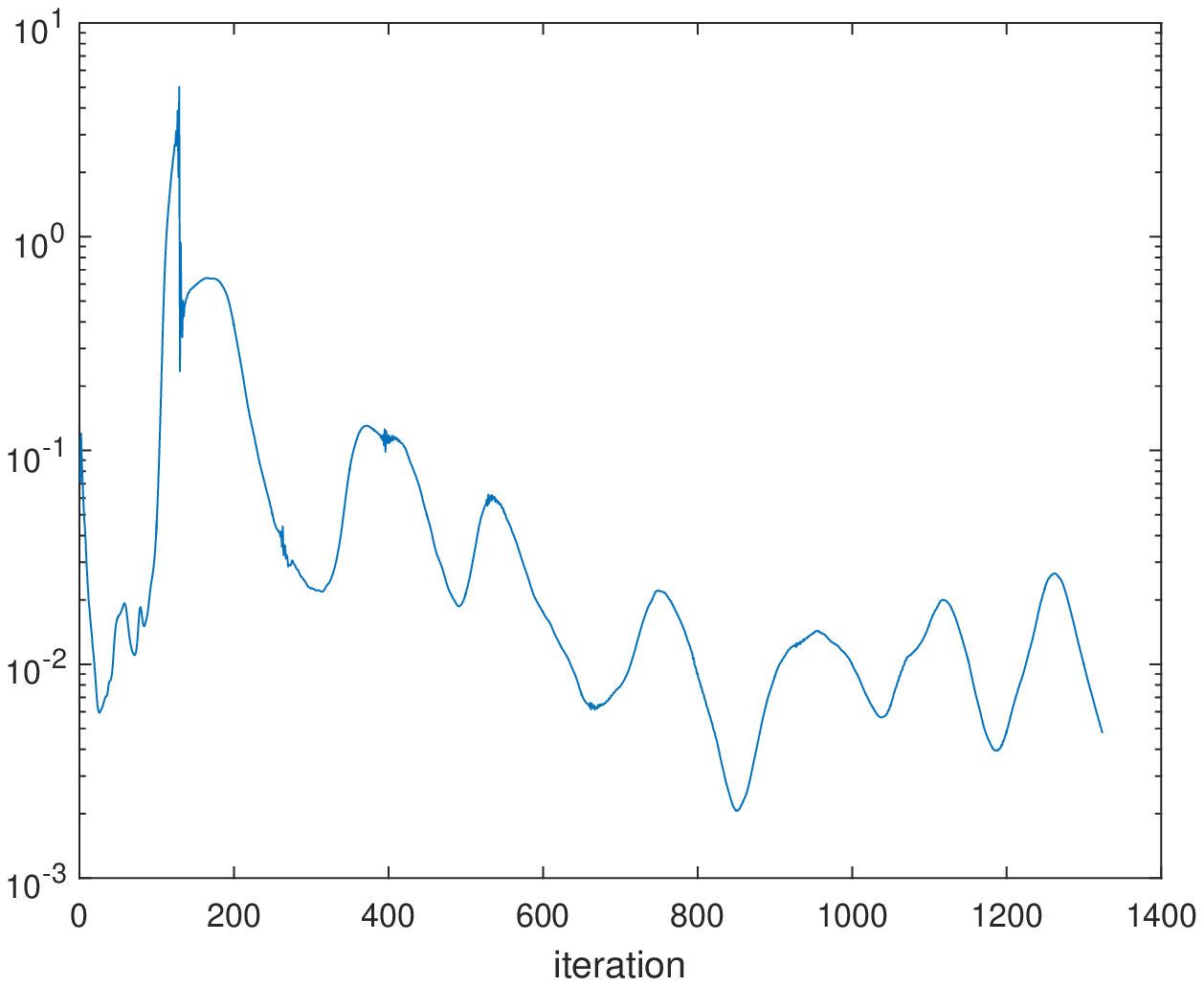} \\
$C_{120}$\\
\includegraphics[width=0.45\textwidth]{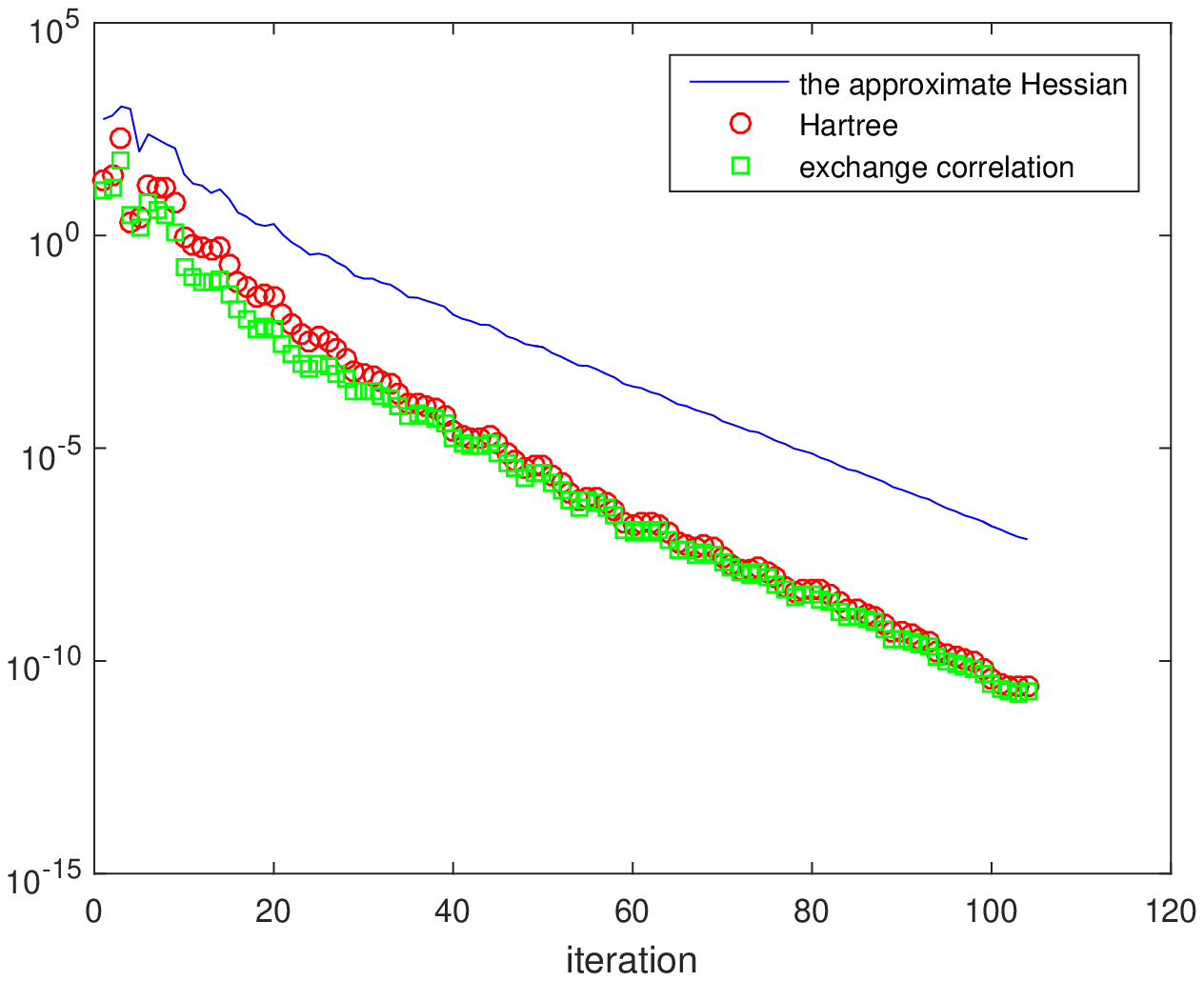}
\includegraphics[width=0.45\textwidth]{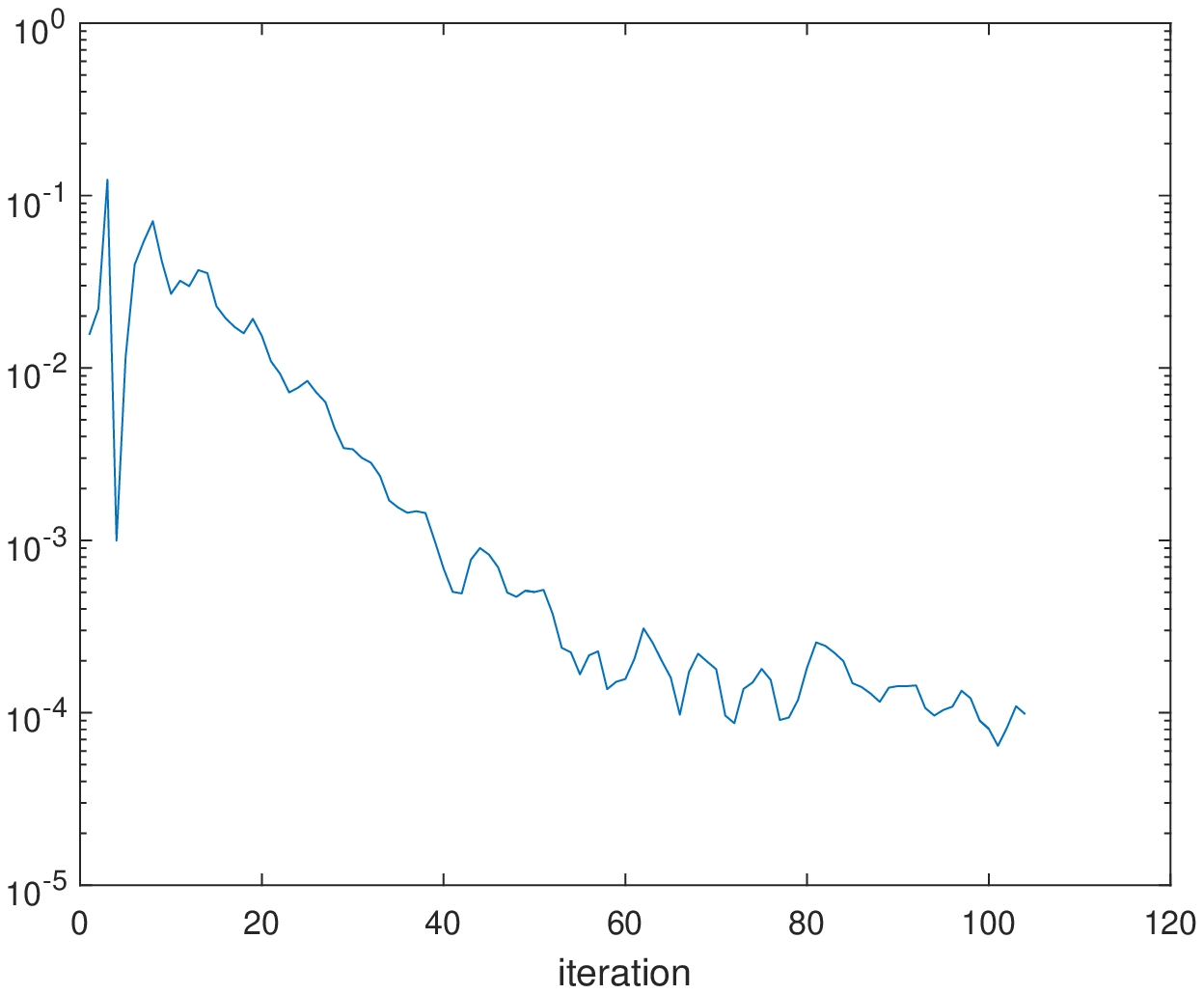} \\
$C_{1015}H_{460}$\\
\includegraphics[width=0.45\textwidth]{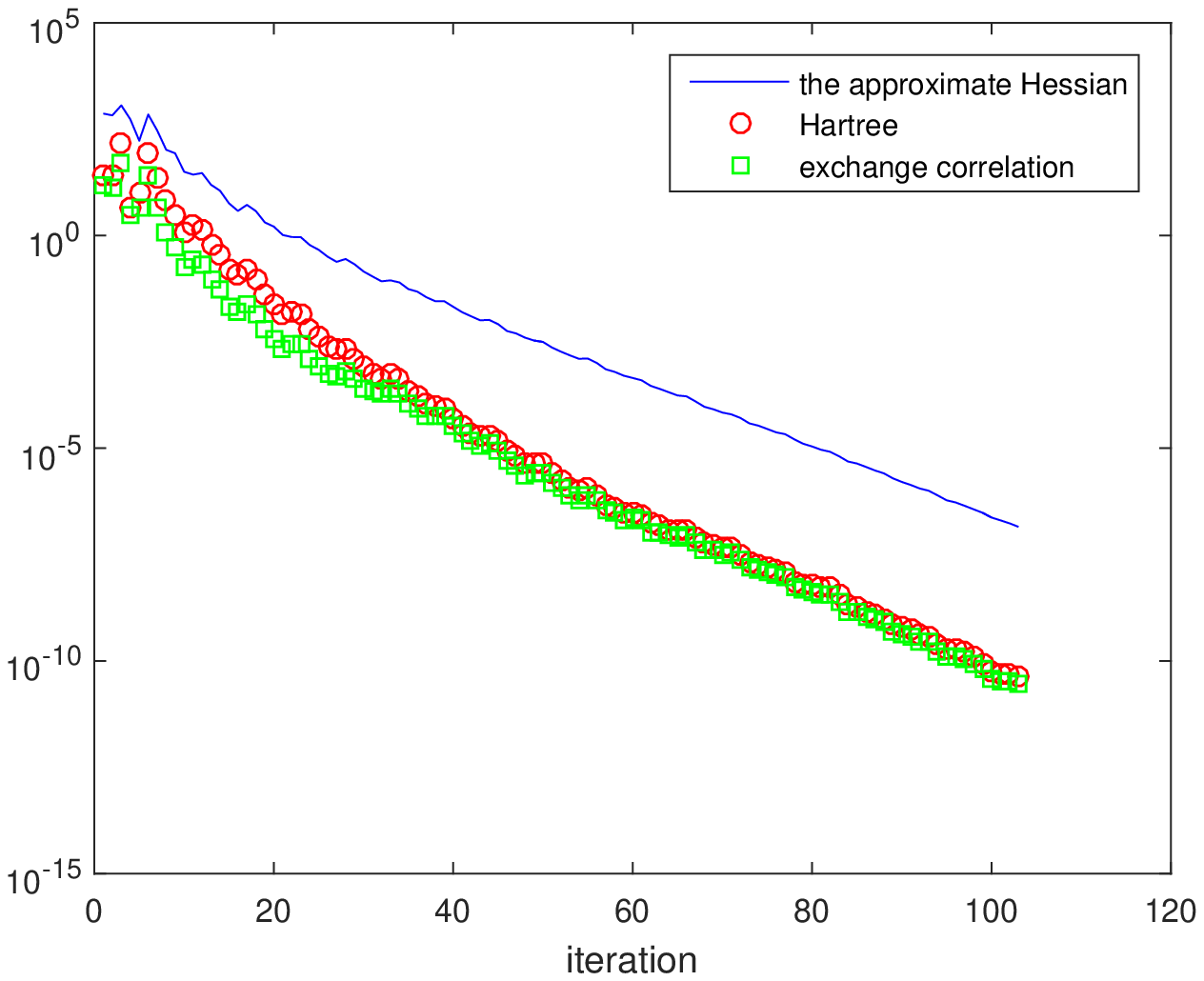}
\includegraphics[width=0.45\textwidth]{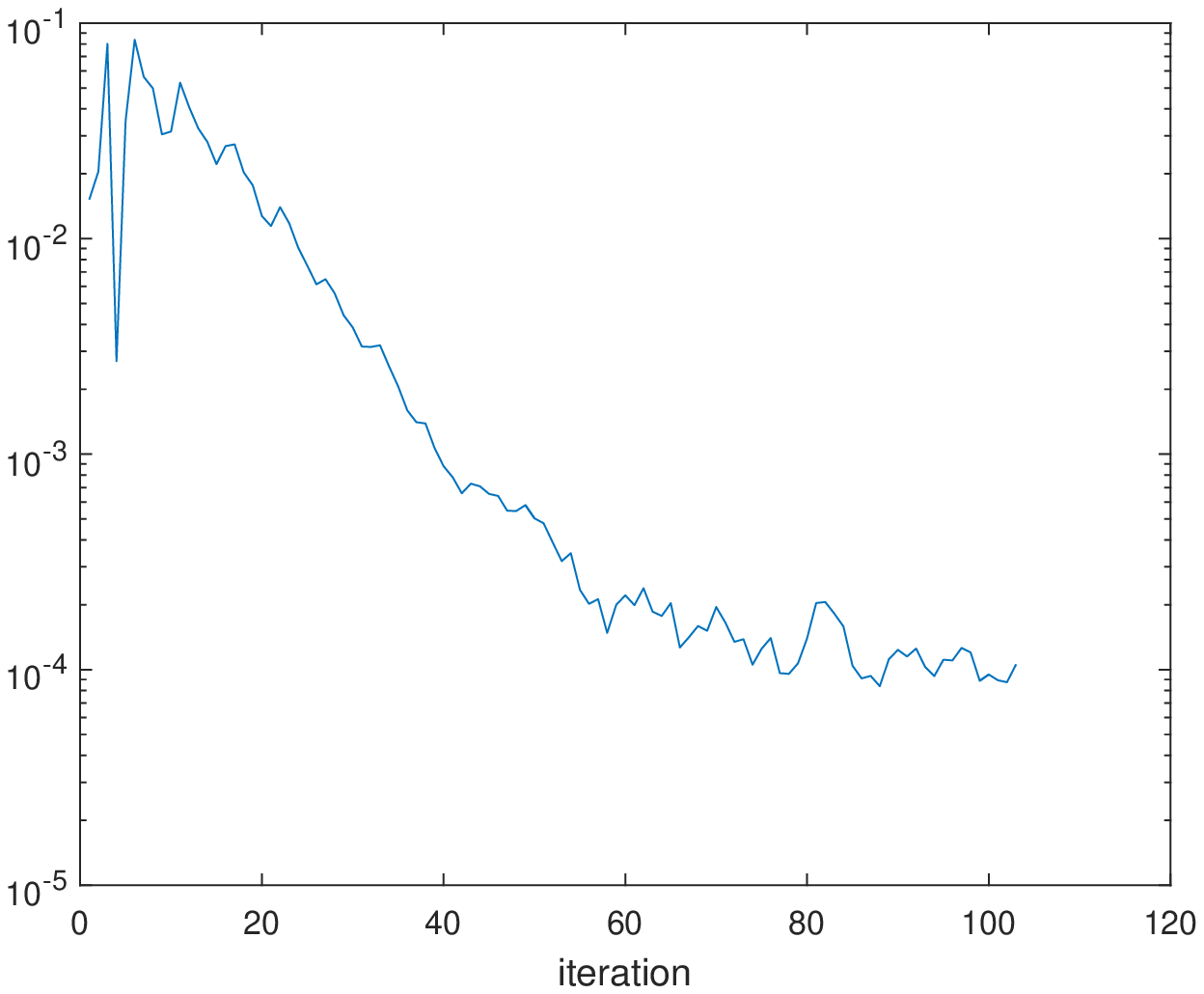} \\
$C_{1419}H_{556}$
\end{figure}

%\section{A modified algorithm}

\end{document}